\documentclass[11pt,reqno]{amsart}
\usepackage{txfonts}

\usepackage{amsmath}
\usepackage{amsfonts}
\usepackage{amssymb}
\usepackage[all,cmtip]{xy}
\usepackage{mathtools}
\usepackage{bbm}
\usepackage{csquotes}
\usepackage{bbding}

\usepackage{xcolor}
\usepackage[colorlinks=true, pdfborder={0 0 0}]{hyperref}

\usepackage{tikz-cd}
\usepackage[active]{srcltx}
\usepackage{tikz}
\usepackage{amscd}
\usepackage{bm}
\usepackage{mathrsfs}

\makeatletter
\let\uppercasenonmath\@gobble
\makeatother

\makeatletter
\renewcommand\section{\@startsection{section}{1}%
  \z@{.7\linespacing\@plus\linespacing}{.9\linespacing}%
  {\normalfont\Large\bfseries\raggedright}}
\makeatother

\makeatletter

\newtheorem{thm}{Theorem}[section]
\newtheorem{prop}[thm]{Proposition}
\newtheorem{lem}[thm]{Lemma}
\newtheorem{cor}[thm]{Corollary}
\newtheorem{prop-def}[thm]{Proposition-Definition}

\theoremstyle{definition}
\newtheorem{defn}[thm]{Definition}

\newtheorem{remark}[thm]{Remark}
\newtheorem{exam}[thm]{Example}

\newcommand{\Add}{{\rm Add}}

\newcommand{\gd}{{\rm gpd}}

\newcommand{\End}{{\rm End}}

\newcommand{\pd}{{\rm pd}}

\newcommand{\fd} {{\rm findim}}
\newcommand{\Fd} {{\rm Findim}}

\newcommand{\fpd} {{\rm fpd}}
\newcommand{\Fpd} {{\rm Fpd}}

\newcommand{\cpx}[1]{#1^{\bullet}}
\newcommand{\D}[1]{{\mathscr D}(#1)}

\newcommand{\Db}[1]{{\mathscr D}^b(#1)}

\newcommand{\Cb}[1]{{\mathscr C}^b(#1)}
\newcommand{\Dc}[1]{{\mathscr D}^c(#1)}
\newcommand{\K}[1]{{\mathscr K}(#1)}

\newcommand{\Kb}[1]{{\mathscr K}^b(#1)}

\newcommand{\modcat}{\ensuremath{\mbox{{\rm -mod}}}}
\newcommand{\Modcat}{\ensuremath{\mbox{{\rm -Mod}}}}

\newcommand{\pmodcat}[1]{#1\mbox{{\rm -proj}}}

\newcommand{\Pmodcat}[1]{#1\mbox{{\rm -Proj}}}
\newcommand{\Imodcat}[1]{#1\mbox{{\rm -Inj}}}

\newcommand{\opp}{^{\rm op}}
\newcommand{\otimesL}{\otimes^{\rm\mathbb L}}

\newcommand{\Hom}{{\rm Hom}}

\newcommand{ \Cone }{{\rm Con}}

\newcommand{\Ext}{{\rm Ext}}

\newcommand{\lra}{\longrightarrow}

\newcommand{\lraf}[1]{\stackrel{#1}{\lra}}

\newcommand{\ra}{\rightarrow}

\newcommand{\Tor}{{\rm Tor}}

\newcommand{\Coprod}{{\rm Coprod}}
\newcommand{\coprods}{{\rm coprod}}
\newcommand{\gpd}{{\rm gpd}}

\begin{document}

\title[]{\Large Finiteness of homological dimensions in triangulated categories}

\author[]{Hongxing Chen, Xiaohu Chen and Jinbi Zhang$^*$}

\address{Hongxing Chen, School of Mathematical Sciences, Capital Normal University, 100048 Beijing, P. R. China;
and Academy for Multidisciplinary Studies, Capital Normal University, 100048 Beijing, P. R. China}
\email{chenhx@cnu.edu.cn}

\address{Xiaohu Chen, School of Mathematical Sciences, Capital Normal University, 100048 Beijing, P. R. China}
\email{xiaohu.chen@cnu.edu.cn}

\address{Jinbi Zhang, School of Mathematical Sciences, Anhui University, 230601 Hefei, P. R. China}
\email{zhangjb@ahu.edu.cn}

\date{}

\begin{abstract}
In a general triangulated category, the finiteness of the finitistic dimension serves as a prerequisite for a categorical obstruction, via the singularity category, to the existence of bounded $t$-structures. In this paper, we investigate the finitistic, big finitistic, and global dimensions, and establish explicit inequalities that relate these dimensions of the middle category in a recollement of triangulated categories to those of the outer categories. This provides a unified framework for extending some known results on the homological dimensions of ordinary rings to weakly approximable triangulated categories.

\end{abstract}
\renewcommand{\thefootnote}{\alph{footnote}}
\setcounter{footnote}{-1} \footnote{ $^*$ Corresponding author.
Email: zhangjb@ahu.edu.cn.}
\renewcommand{\thefootnote}{\alph{footnote}}

\setcounter{footnote}{-1} 


\maketitle

\vspace{-.7cm}

 \tableofcontents

\vspace{-.7cm}

\allowdisplaybreaks

\section{Introduction}
In homological algebra and representation theory, different types of homological dimensions of algebras, such as global dimension, finitistic dimension and dominant dimension, play important roles in understanding the structure and complexity of modules over algebras. One of the main interests on these dimensions is concentrated on their relationships with classical homological conjectures. For example, the famous \emph{finitistic dimension conjecture} says that every Artin algebra has finite finitistic dimension (see \cite{bass60}), where the finitistic dimension of an algebra is defined as the supremum of the projective dimensions of all finitely generated modules having finite projective dimension. This conjecture implies several other homological conjectures including the Auslander-Reiten conjecture, the (generalized) Nakayama conjecture and the Wakamatsu tilting conjecture. Despite some new advances on (big) finitistic dimension in recent years (\cite{g22,g25,r19}), the finitistic dimension conjecture is still open.

The persistent difficulty of solving the finitistic dimension conjecture has inspired a shift in perspective: instead of studying the finitistic dimension of an algebra purely through a single module category, one seeks to understand it as an intrinsic property of a reasonable triangulated category. The  pioneering work was done by Krause, who introduced an important notion of finitistic dimension for Hom-finite triangulated categories in terms of the generativity of objects and provided the first way to characterize the finiteness of the (small) finitistic dimension of a ring as a property of the derived category of perfect complexes (\cite{k24}). Another completely different notion of finitistic dimension for a general triangulated category was introduced in \cite{bcrpz24}, where its finiteness serves as a prerequisite for a categorical obstruction (the singularity category) to the existence of bounded $t$-structures. Precisely, it was shown in \cite[Theorem 1.5]{bcrpz24} that, for an essentially small triangulated category, if its opposite category has finite finitistic dimension, then the existence of a bounded $t$-structure on it implies its regularity, that is, its singularity category vanishes. This result can be regarded as a categorical version of \cite[Theorem 0.1]{n24a} which is a vast generalization of the conjecture of Antieau, Gepner and Heller concerning the regularity of schemes (see \cite[Conjecture 1.5]{agh19}). To bound the finiteness of the finitistic dimension in the second type for specific classes of triangulated categories, the third type of big or small finitistic dimensions for compactly generated triangulated categories was introduced in \cite[Appendix B]{bcrpz24} in terms of the $t$-structures generated by compact generators. They are a natural generalization of the big or small finitistic dimensions of ordinary rings defined through projective dimensions of modules.

All the above categorical viewpoints on finitistic dimension not only encompass the classical module-theoretical case but also extend naturally to triangulated and geometric settings, providing a unified framework to study finiteness phenomena in homological algebra. In the present paper, we investigate the behavior of homological dimensions of triangulated categories linked by recollements of triangulated categories. In particular, we establish explicit inequalities of finitistic, big finitistic, and global dimensions of triangulated categories in recollements. This yields a unified homological/categorical  framework that extends greatly the main results of Chen and Xi in \cite{cx17} on homological dimensions of ordinary rings.

Before stating our results, we first introduce some definitions and notation.

\begin{defn}{\rm \cite[Definition 1.3]{bcrpz24}}\label{Intro-defn-fd}
Let $\mathcal{S}$ be a triangulated category with the shift functor $[1]$. The \emph{finitistic dimension} of $\mathcal{S}$ at an object $G$ is defined to be
$$\fd(\mathcal{S}, G)\coloneqq \inf\big\{n\in\mathbb{N}\mid G(-\infty, -1]^{\perp}\subseteq\langle G\rangle^{[0,+\infty)}[n]\big\},$$
where $G(-\infty, -1]^{\perp}\coloneqq\{X\in\mathcal{S}\mid \Hom_\mathcal{S}(G[n], X)=0,\, n\geq 1\}$ and $\langle G\rangle ^{[0,+\infty)}$ denotes the smallest full subcategory of $\mathcal{S}$ containing $G$ and closed under extensions, direct summands and negative shifts. Clearly,
$\fd(\mathcal{S}, G)\in\mathbb{N}\cup\{+\infty\}$.

We say that $\mathcal{S}$ has \emph{finite finitistic dimension}, denoted by $\fd(\mathcal{S})<+\infty$,  if there is some object $G\in \mathcal{S}$ such that $\fd(\mathcal{S}, G) < +\infty$.
\end{defn}

Recall from \cite{bcrpz24} that there are many classes of triangulated categories that have finite finitistic dimension: a triangulated category with an algebraic $t$-structure (that is, the $t$-structure is bounded and its heart is a length category with finitely many isomorphism classes of simple objects),
a triangulated category with a strong generator (or equivalently, with finite Rouquier dimension), the singularity category of a Gorenstein Artin algebra or of a self-injective differential graded (DG) algebra over a field, the derived category of perfect complexes on a finite-dimensional, Noetherian scheme with cohomology supported on a closed subset of the scheme and the perfect derived category of a DG algebra with some conditions on its cohomology.

Our first main result conveys that checking the finiteness of finitistic dimension can be reduced along half recollements of triangulated categories.

\begin{thm}\label{mainthm-fd-1}
Let $\mathcal{R}$, $\mathcal{S}$ and $\mathcal{T}$ be triangulated categories with a left recollement

\begin{align*}
\xymatrixcolsep{4pc}\xymatrix{\mathcal{R} \ar[r]|{i_*} &\mathcal{S} \ar@<-2ex>[l]|{i^*}  \ar[r]|{j^*}  &\mathcal{T}. \ar@<-2ex>[l]|{j_!}
}
\end{align*}
Suppose that $G_{\mathcal{R}}\in \mathcal{R}$, $G_{\mathcal{S}}\in \mathcal{S}$ and $G_{\mathcal{T}}\in \mathcal{T}$. Let $\langle G_{\mathcal{R}} \rangle$ and $\langle G_{\mathcal{S}} \rangle$ be the smallest full triangulated subcategories of $\mathcal{R}$ and $\mathcal{S}$ containing $G_\mathcal{R}$ and $G_{\mathcal{S}}$ closed under direct summands, respectively.

$(1)$ If $\Hom_{\mathcal{S}}(G_{\mathcal{S}}[n], G_{\mathcal{S}})=0$ for $n\gg 0$ and $i_*(G_{\mathcal{R}}), j_!(G_{\mathcal{T}})\in \langle G_{\mathcal{S}} \rangle$, then
$$\fd(\mathcal{S},G_{\mathcal{S}})\le \fd(\mathcal{R},G_{\mathcal{R}})+\fd(\mathcal{T},G_{\mathcal{T}})
+w_{G_{\mathcal{S}}}(i_{*}(G_{\mathcal{R}}))+w_{G_{\mathcal{S}}}(j_{!}(G_{\mathcal{T}}))
+o^{-}(G_{\mathcal{S}})+1,$$
where $w_{G_{\mathcal{S}}}(X)$ denotes the width of an object $X\in\mathcal{S}$ with respect to $G_{\mathcal{S}}$, and $o^{-}(G_{\mathcal{S}})$ denotes the lower self-orthogonal index of $G_{\mathcal{S}}$ in $\mathcal{S}$ (see Definition \ref{index}). In this case, if both $\fd(\mathcal{R},G_{\mathcal{R}})$ and $\fd(\mathcal{T},G_{\mathcal{T}})$ are finite, then so is $\fd(\mathcal{S},G_{\mathcal{S}})$.

$(2)$ If $i^*(G_{\mathcal{S}})\in \langle G_{\mathcal{R}} \rangle$, then
$\fd(\mathcal{R},G_{\mathcal{R}})\le \fd(\mathcal{S},G_{\mathcal{S}})+w_{G_{\mathcal{R}}}(i^*(G_{\mathcal{S}})).$
In this case, if $\fd(\mathcal{S},G_{\mathcal{S}})$ is finite, then so is $\fd(\mathcal{R},G_{\mathcal{R}})$.

\end{thm}

Theorem \ref{mainthm-fd-1} can be applied to determine the regularity of triangulated categories. Following \cite[Section 1.1]{bcrpz24}, we say that an essentially small triangulated category $\mathcal{S}$ is \emph{regular} if there exists an object $G \in \mathcal{S}$ such that $\mathcal{S}$ is equivalent to $\mathfrak{S}_{G}(\mathcal{S})$ as triangulated categories, where $\mathfrak{S}_{G}(\mathcal{S})$ denotes the completion of the triangulated category $\mathcal{S}$ (in the sense of Neeman) with respect to the good metric defined by the object $G$ (see \cite[Definition 1.4]{bcrpz24} for details). In many practical applications, we always choose $G$ to be a \emph{classical generator} of $\mathcal{S}$, that is, $\mathcal{S}$ is generated by $G$ via taking extensions of triangles, shifts and direct summands in $\mathcal{S}$. For example, the derived category of perfect complexes on a finite-dimensional, Noetherian scheme is regular at a classical generator if and only if the scheme is regular in the geometric sense.

A combination of \cite[Theorem 1.5(a)]{bcrpz24} and Theorem \ref{mainthm-fd-1} implies the following result.

\begin{cor}\label{cor to regular}
Let $\mathcal{R}$, $\mathcal{S}$, and $\mathcal{T}$ be essentially small triangulated categories with $\fd(\mathcal{R}^{{\rm op}})<+\infty$ and $\fd(\mathcal{T}^{{\rm op}})<+\infty$. Suppose that $\mathcal{R}$ and $\mathcal{T}$ have bounded $t$-structures and that there is a recollement of triangulated categories:
\[
\xymatrixcolsep{4pc}\xymatrix{
\mathcal{R} \ar[r]|{i_*=i_!} &
\mathcal{S} \ar@<-2ex>[l]|{i^*} \ar@<2ex>[l]|{i^!} \ar[r]|{j^!=j^*} &
\mathcal{T} \ar@<-2ex>[l]|{j_!} \ar@<2ex>[l]|{j_{*}}.
}
\]
Then $\mathcal{S}$ is regular.
\end{cor}

Theorem \ref{mainthm-fd-1} also generalizes some results in \cite{cx17} on finitistic dimensions of ordinary rings. For an associative ring $R$ with identity, we denote by $R\Modcat$ the category of left $R$-modules, by $\D{R\Modcat}$ the \emph{unbounded derived category} of $R$ and by $\Dc{R\Modcat}$ the full subcategory of $\D{R\Modcat}$ consisting of all compact objects. The last category is called the \emph{perfect derived category} of $R$. It is known that a complex of $R$-modules lies in $\Dc{R\Modcat}$ if and only if it is quasi-isomorphic to a bounded complex of finitely generated projective $R$-modules. Moreover, the \emph{finitistic dimension} of $R$, denoted by $\fd(R)$, is defined as the supremum of the projective dimensions of all $R$-modules having a projective resolution of finite length by finitely generated projective $R$-modules. By \cite[Lemma 4.1]{bcrpz24}, $\fd(R)=\fd(\Dc{R\Modcat}, R)$, which ensures the application of Theorem \ref{mainthm-fd-1} to the case of rings.

\begin{cor}{\rm \cite[Proposition 3.10]{cx17}}\label{CX}
Suppose that there is a half recollement of perfect derived
categories of the associative rings $R$, $S$ and $T$:

\begin{align*}
\xymatrixcolsep{4pc}\xymatrix{\Dc{R\Modcat}  \ar[r]|{i_*} &\Dc{S\Modcat} \ar@<-2ex>[l]|{i^*}  \ar[r]|{j^*}  &\Dc{T\Modcat}\ar@<-2ex>[l]|{j_!}
}
\end{align*}
Then  $\fd(S)\le \fd(R)+\fd(T)+w_S(i_*(R))+w_S(j_!(T))+1,$
where $w_S(i_*(R))$ and $w_S(j_!(T))$ denote the widths of the complexes $i_*(R)$ and $j_!(T)$ with respect to ${_S}S$, respectively.
\end{cor}

In practice, we are more interested in compactly generated triangulated categories.
So, we consider another kind of (big) finitistic dimensions for compactly generated triangulated categories introduced in \cite{bcrpz24} and establish explicit inequalities of those dimensions.

Let $\mathcal{S}$ be a compactly generated triangulated category with a compact generator $G$, and let $(\mathcal{S}^{\le 0},\mathcal{S}^{\ge 1})$ be the $t$-structure on $\mathcal{S}$ generated by $G$ (see Lemma \ref{lem-p-w-gen}). We denote by $\mathcal{S}^b$ and $\mathcal{S}^c$ the full subcategories of $\mathcal{S}$ consisting of all bounded objects and all compact objects, respectively. Here, an object $X\in\mathcal{S}$ is said to be bounded if there exists a nonnegative integer $n$ such that $X[n]\in \mathcal{S}^{\le 0}$ and $X[-n]\in\mathcal{S}^{\ge 1}$, and compact if the functor
$\Hom_\mathcal{S}(X,-):\mathcal{S}\to \mathbb{Z}\Modcat$ commutes with small coproducts. Note that both  $\mathcal{S}^b$ and $\mathcal{S}^c$ are thick subcategories of $\mathcal{S}$. Moreover, the object $G$ is a classical generator of $\mathcal{S}^c$, and $\mathcal{S}^c\subseteq\mathcal{S}^b$ if and only if $\Hom_\mathcal{S}(G,G[i])=0$ for $i\ll0$.

With the above preparations, we now recall the following definition.

\begin{defn}\label{Ffpd}\cite[Definition B.1 and B.2]{bcrpz24}
$(1)$ The \emph{projective dimension} of an object $X\in\mathcal{S}$ with respect to $G$ is defined as
$$\pd_G(X):=\inf\{n\in\mathbb{Z}\mid \Hom_{\mathcal{S}}(X, Y[i])=0, \;\forall i>n,\; Y\in \mathcal{S}^{b}\cap\mathcal{S}^{\le 0}\}$$
$$\quad\quad\quad\quad\, =\inf\{n\in\mathbb{Z}\mid \Hom_{\mathcal{S}}(X, Y[i])=0, \;\forall i>n,\; Y\in \mathcal{S}^{\geq 0}\cap\mathcal{S}^{\leq 0}\}.$$

$(2)$ The \emph{big finitistic dimension} and \emph{finitistic dimension} of $\mathcal{S}$ with respect to $G$ are defined as
\begin{align*}
{\rm Fpd}(\mathcal{S}, G)&\coloneqq \sup\{\pd_G(X)\mid X\in \mathcal{S}^{b}\cap\mathcal{S}^{\ge 0},\; \pd_G(X)<+\infty\},\\
{\rm fpd}(\mathcal{S}, G)&\coloneqq \sup\{\pd_G(X)\mid X\in \mathcal{S}^c\cap \mathcal{S}^{b}\cap\mathcal{S}^{\ge 0},\; \pd_G(X)<+\infty\}.
\end{align*}
\end{defn}

Note that the finiteness of ${\rm Fpd}(\mathcal{S}, -)$ and ${\rm fpd}(\mathcal{S}, -)$ is independent of the choice of different compact generators of $\mathcal{S}$. Moreover, Definition \ref{Ffpd}(2) generalizes the classical notions of the finitistic dimension and big finitistic dimension of an associative ring $R$ since there are equalities
$${\rm Fpd}(\D{R\Modcat}, R)=\Fd(R)\;\;\mbox{and}\;\; {\rm fpd}(\D{R\Modcat}, R)=\fd(R).$$

Weakly approximable triangulated category (see \cite{neeman18, n18} or Section \ref{WAPP}), introduced and developed by Neeman in recent years, is a very important class of compactly generated triangulated categories. In fact, many common triangulated categories have been shown to be weakly approximable. They include the unbounded derived category of an ordinary ring, the homotopy category of spectra (or more generally, the homotopy category of the stable $\infty$-category of left module spectra over a connective $\mathbb{E}_1$-ring), and the derived category of unbounded complexes of $\mathscr{O}_X$-modules with quasicoherent cohomology for a quasicompact, quasiseparated scheme $X$ (see Example \ref{EXWAPP}).

For weakly approximable triangulated categories, the finiteness of the two finitistic dimensions is equivalent. This is illustrated by the following result, which generalizes \cite[Proposition B.3(3)]{bcrpz24}.

\begin{prop}\label{Connection}{\rm (Proposition \ref{Relation})}
Let $\mathcal{S}$ be a weakly approximable triangulated category with a compact generator $G$ such that
$\Hom_{\mathcal{S}}(G, G[i])=0$ for $i\ll0$. Then
$$
\fd(\mathcal{S}^c, G)\leq \fpd(\mathcal{S},G)+a(G)\;\;\mbox{and}\;\; \fpd(\mathcal{S},G)\leq \fd(\mathcal{S}^{c}, G)+o^+(G),
$$ where $a(G)$ and $o^+(G)$ denote the $G$-approximable index and the upper self-orthogonal index of $G$, respectively.
In particular, $\fd(\mathcal{S}^c, G)$ is finite if and only if so is $\fpd(\mathcal{S},G)$.
\end{prop}

Our second main result on finitistic dimension is stated in the following theorem. For unexplained notation in the theorem, we refer to Sections \ref{NTN} and \ref{WAPP}.

\begin{thm}\label{mainthm-fd-2}
{\rm (Theorem \ref{thm-fd-small})}
Let $\mathcal{R}$, $\mathcal{S}$ and $\mathcal{T}$ be  weakly approximable triangulated categories with compact generators $G_{\mathcal{R}}$, $G_{\mathcal{S}}$ and $G_{\mathcal{T}}$, respectively. Suppose that there is a recollement of triangulated categories:
\begin{align*}
\xymatrixcolsep{4pc}\xymatrix{\mathcal{R} \ar[r]|{i_*=i_!} &\mathcal{S} \ar@<-2ex>[l]|{i^*} \ar@<2ex>[l]|{i^!} \ar[r]|{j^!=j^*}  &\mathcal{T}. \ar@<-2ex>[l]|{j_!} \ar@<2ex>[l]|{j_{*}}
}
\end{align*}

{\rm (1)}
Suppose that $\Hom_{\mathcal{S}}(G_{\mathcal{S}},G_{{\mathcal{S}}}[i])=0$ for $i\ll0$ and $i_*(G_{\mathcal{R}})\in\mathcal{S}^{c}$.
Then:

{\rm (i)}
$\fpd(\mathcal{S},G_{\mathcal{S}})\le \fpd(\mathcal{R},G_{\mathcal{R}})+w_{G_{\mathcal{S}}}(i_{*}(G_{\mathcal{R}}))+a(G_{\mathcal{R}})+\fpd(\mathcal{T},G_{\mathcal{T}})+w_{G_{\mathcal{S}}}(j_{!}(G_{\mathcal{T}}))+a(G_{\mathcal{T}})+o^{-}(G_{\mathcal{S}})+o^{+}(G_{\mathcal{S}})+1.$
In particular, if both $\fpd(\mathcal{R},G_{\mathcal{R}})$ and $\fpd(\mathcal{T},G_{\mathcal{T}})$ are finite, then so is $\fpd(\mathcal{S},G_{\mathcal{S}})$.

{\rm (ii)} $\fpd(\mathcal{R},G_{\mathcal{R}})\le \fpd(\mathcal{S},G_{\mathcal{S}})+w_{G_{\mathcal{R}}}(i^{*}(G_{\mathcal{S}}))+a(G_{\mathcal{S}})+o^{+}(G_{\mathcal{R}})$.
In particular, if $\fpd(\mathcal{S},G_{\mathcal{S}})$ is finite, then so is $\fpd(\mathcal{R},G_{\mathcal{R}})$.

{\rm (2)} Suppose that $j_!(\mathcal{T}^{\ge 0})\subseteq \mathcal{S}^{\ge -e}$ for some integer $e$.
Then $\fpd(\mathcal{T},G_{\mathcal{T}})\le\fpd(\mathcal{S},G_{\mathcal{S}})+e+d$, where $d$ is an integer
with $j_!(G_{\mathcal{T}}) \in \langle G_{\mathcal{S}} \rangle^{(-\infty,\, d]}$.
In particular, if $\fpd(\mathcal{S},G_{\mathcal{S}})$ is finite, then so is $\fpd(\mathcal{T},G_{\mathcal{T}})$.

\end{thm}

Our methods also lead to similar results on upper bounds for big finitistic dimension and global dimension of weakly approximable triangulated categories linked by recollements, see Theorems \ref{thm-fd-big} and \ref{global dim}. Moreover, some connections between finite global dimension and strong generation for compactly generated triangulated categories as well as the reduction of strong generation via recollements are given in Theorems \ref{main-thm-scg-gd} and \ref{thm-rec-scg}.

Applying Theorems \ref{mainthm-fd-2} and \ref{thm-fd-big} to unbounded derived categories, we re-obtain the following result.

\begin{cor}{\rm \cite[Theorem 1.1]{cx17}}\label{mainthm-cx17}
Let $R$, $S$ and $T$ be three rings.
Suppose that there is a recollement among the derived categories
$\D{R\Modcat}$, $\D{S\Modcat}$ and $\D{T\Modcat}$$:$
\begin{align*}
\xymatrixcolsep{4pc}\xymatrix{
\D{R\Modcat} \ar[r]|{i_*=i_!} &\D{S\Modcat} \ar@<-2ex>[l]|{i^*} \ar@<2ex>[l]|{i^!} \ar[r]|{j^!=j^*}  &\D{T\Modcat}. \ar@<-2ex>[l]|{j_!} \ar@<2ex>[l]|{j_{*}}
}
\end{align*}

{\rm (1)} Suppose that $i_*(R)$ is isomorphic in $\D{S\Modcat}$ to a bounded complex of finitely generated (respectively, arbitrary) projective $S$-modules. Then

{\rm (i)} $\fd(R)\le \fd(S)+w_R(i^*(S))$ (respectively, $\Fd(R)\le \Fd(S)+w_R(i^*(S))$).

{\rm (ii)} $\fd(S)\le \fd(R)+\fd(T)+w_S(i_*(R))+w_S(j_!(T))+1$ (respectively, $\Fd(S)\le \Fd(R)+\Fd(T)+w_S(i_*(R))+w_S(j_!(T))+1$).

{\rm (2)} Suppose that the functor $j_!$ restricts to $\Db{T\Modcat}\ra \Db{S\Modcat}$. Then
$$\fd(T)\le \fd(S)+cw(j^!(\Hom_{\mathbb{Z}}(S,\mathbb{Q}/\mathbb{Z})))$$ (respectively, $\Fd(T)\le \Fd(S)+cw(j^!(\Hom_{\mathbb{Z}}(S,\mathbb{Q}/\mathbb{Z})))$).
\end{cor}

In Corollary \ref{mainthm-cx17}, $w_R(-)$ and $w_S(-)$ denote the widths of objects (see Definition \ref{index}), while $cw(-)$ denotes the cowidth of objects (see Definition \ref{cowidth}).

\emph{ Structure of the paper:} In Section \ref{2}, we fix some notation and recall some basic definitions including $t$-structures, recollements and weakly approximable triangulated categories. In Section \ref{3}, we introduce the notion of strongly bounded objects in weakly approximable triangulated categories and give some basic properties of these objects. Further, we use the information of strongly bounded objects to characterize a recollement of weakly approximable triangulated categories which restricts to one of their full subcategories consisting of all bounded above objects (Theorem \ref{prop-rec-res1}). In Section \ref{4}, we establish a series of inequalities of (big) finitistic dimensions for weakly approximable triangulated categories in recollements (Theorems \ref{thm-fd-small} and \ref{thm-fd-big}). In Section \ref{5}, we give some reductions of global dimension (Proposition \ref{App-case}) and strong compact generation (Theorem \ref{thm-rec-scg}) for weakly approximable triangulated categories in recollements. Moreover, we establish a close relationship between global dimension and strong compact generation for compactly generated triangulated categories (Theorem \ref{main-thm-scg-gd}). This can be regarded as a categorical version of Kelly's theorem on unbounded derived categories of ordinary rings. In Section \ref{6}, we apply our main results to recollements of derived categories induced by idempotents of rings and provide a method for constructing non-positive dg algebras with finite (big) finitistic dimension (Corollary \ref{Idempotent}).

\section{Preliminaries}\label{2}
In this section, we briefly fix the notation and recall some definitions and basic facts used in the paper.

\subsection{Notation}\label{NTN}
Let $\mathcal{S}$ be a triangulated category with the shift functor $[1]$. Given two full subcategories $\mathcal{X}$ and $\mathcal{Y}$ of $\mathcal{S}$, their \emph{extension} in $\mathcal{S}$, denoted by $\mathcal{X} \ast \mathcal{Y}$, is defined to be the full subcategory of $\mathcal{S}$ consisting of all objects $Z$ such that there exists a triangle
$X \to Z \to Y \to X[1]$
in $\mathcal{S}$ with $X \in \mathcal{X}$ and $Y \in \mathcal{Y}$. We also define
$$\mathcal{X}^{\perp}\coloneqq\{M\in \mathcal{S}\mid \Hom_{\mathcal{S}}(X,M)=0, \;X\in \mathcal{X}\} \;\; \mbox{and}\;\;{}^{\perp}\mathcal{X}\coloneqq\{M\in \mathcal{S}\mid \Hom_{\mathcal{S}}(M, X)=0,\; X\in \mathcal{X}\}$$ are the full subcategories of $\mathcal{S}$ that are right and left orthogonal to $\mathcal{X}$, respectively.

Suppose that $\mathcal{S}$ admits (small) coproducts. We say that an object $X \in \mathcal{S}$ is \emph{compact} if the Hom-functor $\Hom_{\mathcal{S}}(X,-) \colon \mathcal{S} \to\mathbb{Z}\Modcat$ preserves coproducts. We denote by $\mathcal{S}^{c}$ the full subcategory of $\mathcal{S}$ consisting of all compact objects. The category $\mathcal{S}$ is said to be \emph{compactly generated} if there exists a set $\mathcal{G}\subseteq \mathcal{S}^{c}$ such that $\bigcap_{i\in\mathbb{Z}}(\mathcal{G}[i])^{\perp}=\{0\}$.
It is known that $\mathcal{S}$ is compactly generated if and only if there exists a set $\mathcal{G}\subseteq \mathcal{S}^{c}$ such that $\mathcal{S}$ is exactly the smallest full triangulated subcategory of $\mathcal{S}$ containing $\mathcal{G}$ and closed under coproducts.

We now recall the following notation and definition mainly from \cite[Reminder 1.12]{n18}. Whenever coproducts are needed, we tacitly assume that $\mathcal{S}$ admits them.

\begin{defn}\cite[Reminder 1.12]{n18}\label{notation} Let $\mathcal{S}$ be a triangulated category, $\mathcal{A} \subseteq \mathcal{S}$ a subcategory and $G \in \mathcal{S}$ an object.

{\rm (1)} $\text{smd}(\mathcal{A})$ (resp. $\text{add}(\mathcal{A})$, $\text{Add}(\mathcal{A})$) denotes the full subcategory of $\mathcal{S}$ consisting of all direct summands (resp. finite direct sums, coproducts) of objects in $\mathcal{A}$.

{\rm (2)} For $n>0$, define inductively:
\[\text{coprod}_{1}(\mathcal{A})\coloneqq\text{add}(\mathcal{A}),\;\;
\text{coprod}_{n+1}(\mathcal{A})\coloneqq\text{coprod}_{1}(\mathcal{A}) \ast \text{coprod}_{n}(\mathcal{A}),\]
and similarly,
\[\text{Coprod}_{1}(\mathcal{A})\coloneqq\text{Add}(\mathcal{A}),\;\;
\text{Coprod}_{n+1}(\mathcal{A})\coloneqq\text{Coprod}_{1}(\mathcal{A}) \ast \text{Coprod}_{n}(\mathcal{A}).\]
We set $\text{coprod}(\mathcal{A}) \coloneqq \bigcup_{n>0} \text{coprod}_n(\mathcal{A})$.

{\rm (3)} $\text{Coprod}(\mathcal{A})$ denotes the smallest full subcategory of $\mathcal{S}$ containing $\mathcal{A}$ and closed under taking coproducts and extensions. If $\mathcal{A}[1]\subseteq \mathcal{A}$ or $\mathcal{A} \subseteq\mathcal{A}[1]$, then by the Eilenberg swindle argument, $\text{Coprod}(\mathcal{A})$ is closed under direct summands in $\mathcal{S}$.

{\rm (4)} For integers $A \leq B$, let
$G[A,B] \coloneqq \{ G[-i] \mid i \in \mathbb{Z},\ A \leq i \leq B \}$.
For later use, we also extend this notation to allow $A$ or $B$ to be infinite; for example,
$G(-\infty, B] \coloneqq \{ G[-i] \mid i \in \mathbb{Z},\ i \leq B \}$.
For $n>0$, let
$$
{\langle G \rangle}^{[A,B]}_n\coloneqq\text{smd}(\text{coprod}_{n}(G[A,B])),\quad
{\langle G \rangle}^{[A,B]}\coloneqq\bigcup_{n>0} {\langle G \rangle}^{[A,B]}_{n},$$
$${\langle G \rangle}_{n}\coloneqq{\langle G \rangle}^{(-\infty, +\infty)}_n,\quad
\langle G \rangle\coloneqq \bigcup_{n>0}{\langle G \rangle}_n.
$$
In particular, $\langle G \rangle$ is the smallest full triangulated subcategory of $\mathcal{S}$ containing $G$ and closed under direct summands.

{\rm (5)} Let $A\leq B$ be integers (possibly infinite), and let $n$ be a positive integer. We define
$$\overline{\langle G \rangle}^{[A,B]}_n\coloneqq\text{smd}(\text{Coprod}_n(G[A,B]))\quad \mbox{and}\quad \overline{\langle G \rangle}^{[A,B]}\coloneqq\text{smd}(\text{Coprod}(G[A,B])).$$

{\rm (6)} The object $G$ in a triangulated category $\mathcal{S}$ with coproducts is called a \emph{compact generator} if $G\in \mathcal{S}^{c}$ and $\mathcal{S}=\overline{\langle G \rangle}^{(-\infty, +\infty)}$. In this case, if $\Hom_{\mathcal{S}}(G,G[i])=0$ for $i\ll 0$, then $G$ is further called a \emph{bounded compact generator} of $\mathcal{S}$.
\end{defn}

\begin{exam}
A typical example among triangulated categories generated by a single compact generator is the unbounded derived category of an ordinary ring, or more generally, the homotopy category of the stable $\infty$-category of left $R$-module spectra over an $\mathbb{E}_1$-ring $R$. Clearly, $R$ is a bounded compact generator if and only if the $n$-th homotopy group $\pi_n(R)$ of $R$ vanishes for $n\gg 0$.
\end{exam}

The following lemma is a direct consequence of  \cite[Proposition 1.9]{n21}, which will be used later.

\begin{lem}\label{Closure}
Suppose that $\mathcal{S}$ is a triangulated category with coproducts and with $G\in\mathcal{S}^c$. Let $A\leq B$ be integers (possibly infinite), and let $n$ be a positive integer. Then
$$
\mathcal{S}^c\cap\overline{\langle G\rangle}^{[A,B]}_n={\langle G \rangle}^{[A,B]}_n\quad \mbox{and}\quad \mathcal{S}^c\cap\overline{\langle G \rangle}^{[A,B]}={\langle G \rangle}^{[A,B]}.$$
\end{lem}

To bound the self-orthogonal indexes of objects in a triangulated category, we introduce the following definition.

\begin{defn}\label{index}
Let $\mathcal{S}$ be a triangulated category with an object $G$.

$(1)$ The \emph{lower and upper self-orthogonal indexes} of $G$ in $\mathcal{S}$ are defined as follows:
\[o^{-}(G) \coloneqq \inf \{ b \geq 0 \mid \Hom_{\mathcal{S}}(G[i], G) = 0 \text{ for all } i >b \},\]
\[
o^{+}(G) \coloneqq \inf \{ c \geq 0 \mid \Hom_{\mathcal{S}}(G, G[i]) = 0 \text{ for all } i > c \}.
\]

(2) For each $P \in \langle G\rangle$, we define the \emph{width} of $P$ with respect to $G$ as \[w_{G}(P)\coloneqq\min\{\beta-\alpha\mid P\in \langle G\rangle^{[\alpha,\beta]}\}.\]
When $\mathcal{S}$ has coproducts, for each $Q\in\overline{\langle G\rangle}^{[a,b]}$ for some $a\le b$, we also define
\[\overline{w_{G}}(Q)\coloneqq\min\{\beta-\alpha\mid Q\in\overline{\langle G\rangle}^{[\alpha,\beta]}\}.\]
\end{defn}

\begin{remark}\label{kuandu}
(1) Suppose that $\mathcal{S}$ has coproducts. In Definition \ref{index}(2), for any $P\in\langle G\rangle$, we have $\overline{w_{G}}(P)\leq w_{G}(P)$, and the equality holds if  $P\in\mathcal{S}^{c}$, due to Lemma \ref{Closure}.

(2) Let $R$ be an ordinary ring with identity. For any $P\in\mathscr{K}^{b}(R\text{-}\mathrm{proj})$ (resp. $\mathscr{K}^{b}(R\text{-}\mathrm{Proj})$), the width $w_{R}(P)$ (resp. $\overline{w_{R}}(P)$) in Definition \ref{index} coincides with the homological width defined in \cite[Section 3.1]{cx17}.
\end{remark}

\subsection{$t$-structures and recollements}
In this subsection, we recall the definitions of $t$-structures on triangulated categories and recollements of triangulated categories as well as their properties.

\begin{defn}\cite[Definition 1.3.1]{BBD}\label{T-structure}
Let $\mathcal{S}$ be a triangulated category. A pair of full subcategories $(\mathcal{S}^{\leq 0},\mathcal{S}^{\geq 1})$ in $\mathcal{S}$ is called a \emph{$t$-structure} on $\mathcal{S}$ if the following conditions are satisfied:

(1) $\mathcal{S}^{\leq 0}[1] \subseteq \mathcal{S}^{\leq 0}$ and $\mathcal{S}^{\geq 1} \subseteq\mathcal{S}^{\geq 1}[1]$;

(2) $\Hom_{\mathcal{S}}(\mathcal{S}^{\leq 0},\mathcal{S}^{\geq 1})=0$;

(3) For any $X \in \mathcal{S}$, there is a triangle
$X^{\leq 0} {\rightarrow} X {\rightarrow} X^{\geq 1} \rightarrow X^{\leq 0}[1]$
with $X^{\leq 0} \in \mathcal{S}^{\leq 0}$ and $X^{\geq 1} \in \mathcal{S}^{\geq 1}$.
\end{defn}

Let $(\mathcal{S}^{\leq 0},\mathcal{S}^{\geq 1})$ be a $t$-structure on $\mathcal{S}$.
For each $n \in \mathbb{Z}$, we denote
$$\mathcal{S}^{\leq n}\coloneqq \mathcal{S}^{\leq 0}[-n],\;\;\mathcal{S}^{\geq n+1}\coloneqq \mathcal{S}^{\geq 1}[-n]\;\;\mbox{and}\;\;\mathcal{H}\coloneqq \mathcal{S}^{\leq 0} \cap \mathcal{S}^{\geq 0}.$$
Then $(\mathcal{S}^{\leq n})^{\perp}=\mathcal{S}^{\geq n+1}$ and
$^{\perp}(\mathcal{S}^{\geq n})=\mathcal{S}^{\leq n-1}$. This means that $\mathcal{S}^{\leq n}$ and $\mathcal{S}^{\geq n+1}$ are determined by each other. Moreover, up to non-canonical isomorphism, there is a unique triangle  $$ X^{\leq n-1} \lra X \lra X^{\geq n} \lra  X^{\leq n-1}[1]$$ in $\mathcal{S}$ with $X^{\leq n-1}\in \mathcal{S}^{\leq n-1}$ and $X^{\geq n}\in \mathcal{S}^{\geq n}$. The $t$-structure is said to be \emph{bounded above} (resp.\ \emph{bounded below}) if $\mathcal{S} = \bigcup_{m=1}^{\infty} \mathcal{S}^{\leq m}$ (resp.\ $\mathcal{S} = \bigcup_{m=1}^{\infty} \mathcal{S}^{\geq -m}$). It is \emph{bounded} if it is both bounded above and bounded below. Note that $\mathcal{H}$ is an abelian category and is usually called the \emph{heart} of the $t$-structure $(\mathcal{S}^{\leq 0},\mathcal{S}^{\geq 1})$ (for example, see  \cite[Theorem 1.3.6]{BBD}).

Two $t$-structures $(\mathcal{S}^{\leq 0}_i,\mathcal{S}^{\geq 1}_i)$ for $i=1,2$ on $\mathcal{S}$ are said to be \emph{equivalent} if there exists some natural number $n$ such that $\mathcal{S}^{\leq -n}_1\subseteq \mathcal{S}^{\leq 0}_2\subseteq \mathcal{S}^{\leq n}_1$, or equivalently,  $\mathcal{S}^{\geq n}_1\subseteq \mathcal{S}^{\geq 0}_2\subseteq \mathcal{S}^{\geq -n}_1$. This defines an equivalence relation on the class of all $t$-structures on $\mathcal{S}$.

The following lemma provides a canonical way to construct a $t$-structure. It first appeared in \cite[Theorem A.1]{tls03}, with alternative proofs given in \cite[Theorem 12.1]{kn13} and \cite[Theorem 3.0.1]{chns24}.

\begin{lem}{\rm \cite[Theorem A.1]{tls03}}\label{lem-p-w-gen}
Let $\mathcal{S}$ be a triangulated category with coproducts and let $G$ be a compact object in $\mathcal{S}$. Then there is a $t$-structure on $\mathcal{S}$ defined by
$$
(\mathcal{S}_{G}^{\leq 0},\mathcal{S}_{G}^{\geq 1})\coloneqq\big(\Coprod(G(-\infty,0]),G(-\infty,0]^{\perp}\big).
$$
In particular, the heart of $(\mathcal{S}_{G}^{\leq 0},\mathcal{S}_{G}^{\geq 1})$ is closed under coproducts in $\mathcal{S}$.
\end{lem}

Following \cite[Definition 1.18]{n18}, for a triangulated category $\mathcal{S}$ generated by a single compact generator $G$, the equivalence class of the $t$-structure $(\mathcal{S}_{G}^{\leq 0}, \mathcal{S}_{G}^{\geq 1})$ in Lemma \ref{lem-p-w-gen} is called the \emph{preferred equivalence class}.
Note that if $(\mathcal{S}^{\leq 0},\mathcal{S}^{\geq 1})$ is a $t$-structure on $\mathcal{S}$ in the preferred equivalence class, then $G$ is bounded if and only if $G \in \mathcal{S}^{b}$ if and only if $\mathcal{S}^c\subseteq \mathcal{S}^{b}$. This explains the reason why we introduce the notion of bounded compact generator in Definition \ref{notation}(6).

\begin{defn}{\rm \cite[1.4.3]{BBD}}
Let $\mathcal{R}$, $\mathcal{S}$ and $\mathcal{T}$ be triangulated categories. A \emph{recollement} of $\mathcal{S}$ by $\mathcal{R}$ and $\mathcal{T}$ is a diagram of six triangle functors
\begin{align}\label{diag:rec}
\xymatrixcolsep{4pc}\xymatrix{\mathcal{R} \ar[r]|{i_*=i_!} &\mathcal{S} \ar@<-2ex>[l]|{i^*} \ar@<2ex>[l]|{i^!} \ar[r]|{j^!=j^*}  &\mathcal{T}, \ar@<-2ex>[l]|{j_!} \ar@<2ex>[l]|{j_{*}}
}
\end{align}
satisfying the following conditions:

{\rm (R1)} $(i^\ast,i_\ast)$,\,$(i_!,i^!)$,\,$(j_!,j^!)$ ,\,$(j^\ast,j_\ast)$
are adjoint pairs;

{\rm (R2)}
$i_*,~j_*,~j_!$ are fully faithful$;$

{\rm (R3)}
$j^*i_*=0$ $($thus $i^*j_!=0$ and $i^! j_*=0$$);$

{\rm (R4)}
for any object $Y$ in $\mathcal{S}$, there are two triangles in $\mathcal{S}$ induced by counit and unit adjunctions$:$
$$\xymatrix@R=0.5pc{
i_*i^!(Y)\ar[r] &Y \ar[r] & j_*j^*(Y) \ar[r]& i_*i^!(Y)[1],\\
j_!j^*(Y) \ar[r] &Y \ar[r] & i_*i^*(Y) \ar[r]& j_!j^*(Y)[1].}$$
The diagram consisting of the upper two rows \begin{align*}
\xymatrixcolsep{4pc}\xymatrix{\mathcal{R} \ar[r]|{i_*} &\mathcal{S} \ar@<-2ex>[l]|{i^*}  \ar[r]|{j^*}  &\mathcal{T} \ar@<-2ex>[l]|{j_!}
}
\end{align*}
is said to be a \emph{left recollement} of $\mathcal{S}$ by $\mathcal{R}$ and $\mathcal{T}$ if the four functors $i^*$, $i_*$, $j_!$ and $j^!$ satisfy the conditions in {\rm (R1)-(R4)} involving them.  A \emph{right recollement} is defined similarly via the lower two rows.

We say that a recollement \eqref{diag:rec} \emph{extends one step downwards} if there are triangle functors $i_{\#}\colon \mathcal{R}\to\mathcal{S}$ and $j^{\#}\colon \mathcal{S}\to\mathcal{T}$ such that the diagram:
\[
\xymatrixcolsep{4pc}\xymatrix{\mathcal{T} \ar[r]|{j_{*}} &\mathcal{S} \ar@<-2ex>[l]|{ j^!=j^*} \ar@<2ex>[l]|{j^{\#}} \ar[r]|{i^{!}}  &\mathcal{R}, \ar@<-2ex>[l]|{i_*=i_!} \ar@<2ex>[l]|{i_{\#}}
}
\]
is also a recollement.
\end{defn}

The following result is well known.
\begin{lem}\label{lem-key}
For the recollement of the form \eqref{diag:rec}, assume further that $\mathcal{R}$, $\mathcal{S}$ and $\mathcal{T}$ are compactly generated triangulated categories. Then we have the following:

$(1)$ {\rm \cite[Theorem 5.1]{n96}} $j_!$ and $i^*$ preserve compact objects.

$(2)$ {\rm \cite[Chapter IV, Proposition 1.11]{br07}} The following statements are equivalent:

\;\quad {\rm (2a)} $i_*$ preserves compact objects. \;\;\; {\rm (2b)} $j^*$ preserves compact objects.

\;\quad {\rm (2c)} $i^!$ has a right adjoint. \quad\quad\quad\;\;\;   {\rm(2d)} $j_*$ has a right adjoint.

\;\quad {\rm (2e)} The recollement extends one step downwards.

If any one of the above conditions holds, then the recollement {\rm (\ref{diag:rec})} induces a left recollement:
\begin{align*}
\xymatrixcolsep{4pc}\xymatrix{\mathcal{R}^c  \ar[r]|{i_*=i_!}
&\mathcal{S}^c\ar@<-2ex>[l]|{i^*}  \ar[r]|{j^!=j^*}
&\mathcal{T}^c. \ar@<-2ex>[l]|{j_!}
}
\end{align*}
\end{lem}

A method for constructing $t$-structures by recollements is given in the following result.

\begin{prop}\label{glued $t$-str}{\rm \cite[Theorem~1.4.10]{BBD}}
In the recollement \eqref{diag:rec}, assume that $(\mathcal{R}^{\leq 0},\mathcal{R}^{\geq 1})$ on $\mathcal{R}$ and $(\mathcal{T}^{\leq 0},\mathcal{T}^{\geq 1})$ on $\mathcal{T}$ are $t$-structures. Then there is a $t$-structure $(\mathcal{S}^{\leq 0},\mathcal{S}^{\geq 1})$ on $\mathcal{S}$ defined by
\[
\mathcal{S}^{\leq 0}=\{X\in\mathcal{S}\,\mid\,i^{*}(X)\in\mathcal{R}^{\leq 0},\,j^{*}(X)\in\mathcal{T}^{\leq 0}\}\quad\text{and}\quad\mathcal{S}^{\geq 1}=\{Y\in\mathcal{S}\,\mid\,i^{!}(Y)\in\mathcal{R}^{\geq 1},\,j^{*}(Y)\in\mathcal{T}^{\geq 1}\}.
\]
This $t$-structure is called the \emph{glued $t$-structure} associated with the given $t$-structures on $\mathcal{R}$ and $\mathcal{T}$.
\end{prop}

Gluing $t$-structures along recollements has the following property.
\begin{lem}\label{lem-bounde induced bounde}
In Proposition \ref{glued $t$-str}, if both $(\mathcal{R}^{\leq 0},\mathcal{R}^{\geq 1})$ and $(\mathcal{T}^{\leq 0},\mathcal{T}^{\geq 1})$ are bounded above (resp. bounded below, bounded), then so is the glued $t$-structure $(\mathcal{S}^{\leq 0},\mathcal{S}^{\geq 1})$.
\end{lem}

\begin{proof}
We prove the assertion of Lemma \ref{lem-bounde induced bounde} for the case of bounded above $t$-structures. Other cases can be shown similarly.

Let $Y\in\mathcal{S}$. Then there is a canonical triangle $j_{!}j^{*}(Y)\to Y\to i_{*}i^{*}(Y)\to j_{!}j^{*}(Y)[1]$ in $\mathcal{S}$. Since both $(\mathcal{R}^{\leq 0},\mathcal{R}^{\geq 1})$ and $(\mathcal{T}^{\leq 0},\mathcal{T}^{\geq 1})$ are bounded above, there exist positive integers $n$ and $m$ such that $j^{*}(Y) \in \mathcal{T}^{\leq n}$ and $i^{*}(Y) \in \mathcal{R}^{\leq m}$. Note that
$i^*j_{!}j^{*}(Y)=0$, $j^*j_{!}j^{*}(Y)\simeq j^{*}(Y)$, $i^*i_{*}i^{*}(Y)\simeq i^{*}(Y)$ and $j^*i_{*}i^{*}(Y)=0$. This implies $j_{!}j^{*}(Y)\in\mathcal{S}^{\leq n}$ and $i_{*}i^{*}(Y)\in\mathcal{S}^{\leq m}$. Let $r:=\max\{n,m\}$. Since $\mathcal{S}^{\leq r}$ is closed under extensions in $\mathcal{S}$, we have $Y\in \mathcal{S}^{\leq r}$. Thus $(\mathcal{S}^{\leq 0},\mathcal{S}^{\geq 1})$ is bounded above.
\end{proof}

\subsection{Weakly approximable triangulated categories}\label{WAPP}
In this subsection, we recall the definitions of (weakly) approximable triangulated categories and their triangulated subcategories introduced by Neeman.

\begin{defn}\label{defn-app}{\rm \cite[Definition 1.25]{n18}}
Let $\mathcal{S}$ be a triangulated category with coproducts.
The category $\mathcal{S}$ is called \emph{weakly approximable} (resp. \emph{approximable}) if
it has a compact generator $G$ and a $t$-structure $(\mathcal{S}^{\le 0},\mathcal{S}^{\ge 1})$ together with an integer $A>0$ satisfying the following conditions:

{\rm (1)}
$G[A]\in \mathcal{S}^{\leq0}$ and $\Hom_{\mathcal{S}}(G[-A],\mathcal{S}^{\leq 0})=0$.

{\rm (2)}
Each object $F\in\mathcal{S}^{\leq 0}$
admits a triangle $E\ra F\ra D \ra E[1]$ in $\mathcal{S}$ with $E\in  \overline{\langle G\rangle}^{[-A,A]}$ (resp. $E\in \overline{\langle G\rangle}_A^{[-A,A]}$) and $D\in\mathcal{S}^{\le -1}$.
\end{defn}

By  \cite[Propositions 3.4 and 3.6; Lemma 3.5]{n18}, the notion of (weakly) approximable triangulated categories does not depend on the choice of compact generators and $t$-structures from the preferred equivalence class, and in particular, the $t$-structure $(\mathcal{S}^{\le 0},\mathcal{S}^{\ge 1})$ on $\mathcal{S}$ in Definition \ref{defn-app} does lie in the preferred equivalence class. Consequently, when addressing a (weakly) approximable triangulated category $\mathcal{S}$ with a compact generator $G$, one can always take the $t$-structure $(\mathcal{S}^{\leq 0},\mathcal{S}^{\geq 1})$ in Definition \ref{defn-app} to be the canonical $t$-structure $(\mathcal{S}_G^{\leq 0},\mathcal{S}_G^{\geq 1})$ associated with $G$ (see Lemma~\ref{lem-p-w-gen}).

We first present some examples of (weakly) approximable triangulated categories.

\begin{exam}\label{EXWAPP}
$(1)$ Let $X$ be a quasicompact, quasiseparated scheme and let $Z$ be a closed subset of $X$ such that $X-Z$ is quasicompact. Denote by $\mathscr{D}_{\rm qc}(X)$ the full subcategory of the unbounded derived category of $\mathscr{O}_X$-modules consisting of cochain complexes of $\mathscr{O}_X$-modules with \emph{quasicoherent} cohomology, and by $\mathscr{D}_{{\rm qc}, Z}(X)$ the full subcategory of $\mathscr{D}_{\rm qc}(X)$ consisting of complexes whose cohomology is supported on $Z$ (that is, the restriction of those complexes to $X-Z$ is acyclic). Then $\mathscr{D}_{{\rm qc},Z}(X)$ is weakly approximable (see \cite[Theorem 3.2(iv)]{n24a}).

If $X$ is a quasi-compact and separated scheme, then $\mathscr{D}_{\mathrm{qc}}(X)$ is approximable (see \cite[Example 5.6]{n18}). More generally, if $(X,\mathcal{A})$ is a separated and quasi-compact noncommutative scheme, then $\mathscr{D}_{\mathrm{qc}}(\mathcal{A})$ is approximable (see \cite[Proposition 4.1]{dlr25}).

$(2)$ If $\mathcal{S}$ is compactly generated with a compact generator $G$ such that $\Hom_{\mathcal{S}}(G, G[i]) = 0$ for $i \geq 1$, then $\mathcal{S}$ is approximable (see \cite[Remark~5.3]{n18}).
This implies that the homotopy category of the stable $\infty$-category of left $R$-module spectra over a connective $\mathbb{E}_1$-ring $R$ is approximable, where $R$ is said to be \emph{connective} if the $n$-th homotopy group $\pi_n(R)$ of $R$ vanishes for all $n<0$. Thus the derived category of a non-positive differential graded (DG) algebra is approximable. This includes the case of the (unbounded) derived category of an ordinary ring.
\end{exam}

\medskip
In the rest of this subsection, we fix a weakly approximable triangulated category $\mathcal{S}$ with a compact generator $G$ and with a $t$-structure $(\mathcal{S}^{\le 0},\mathcal{S}^{\ge 1})$ in the preferred equivalence class.

\begin{defn}\label{App-index}
For each $M\in\mathcal{S}^{\leq 0}$, we define the \emph{approximable index} of $M$ in $\mathcal{S}$ as
\[b(M)\coloneqq\inf\{n\geq0\mid \exists\;\text{a triangle }E\ra M\ra D \ra E[1] \; \mbox{in} \; \mathcal{S}\text{ with }E\in\overline{\langle G\rangle}^{[-n,n]}\text{ and }D\in\mathcal{S}^{\leq -1}\}.\]
The $G$-\emph{approximable index} of $\mathcal{S}$ is defined to be
\[a(G) \coloneqq \sup \{ b(M) \mid M \in \mathcal{S}^{\leq 0} \}.\]
Then $0\leq a(G)\leq A<+\infty$, where $A$ is given in Definition \ref{defn-app}
\end{defn}

By \cite[Remark 5.3]{n18}, if $\Hom_\mathcal{S}(G,G[n])=0$ for $n>0$, then $a(G)=0<A$. Moreover, by the proof of \cite[Corollary 3.2]{n18}, we have the following result.

\begin{lem}\label{lem-app-induction}
For each integer $m>0$ and each object $F\in \mathcal{S}^{\leq 0}$, there is a triangle $E_m\ra F\ra D_m\ra E_m[1]$ in $\mathcal{S}$ with $D_m\in \mathcal{S}^{\le -m}$ and $E_m\in \overline{\langle G\rangle}^{[1-m-a(G),a(G)]}$. If $\mathcal{S}$ is approximable, then $E_m\in\overline{\langle G\rangle}_{mA}^{[1-m-a(G), a(G)]}$.
\end{lem}

Now, we recall some full subcategories of a weakly approximable triangulated category $\mathcal{S}$ (see \cite{cns24,n18}).

\begin{itemize}
\item \emph{Bounded above objects}: $\mathcal{S}^-\coloneqq \bigcup_{m=1}^{\infty}\mathcal{S}^{\leq m}$;
\item \emph{Bounded below objects}: $\mathcal{S}^+\coloneqq \bigcup_{m=1}^{\infty}\mathcal{S}^{\geq-m}$;
\item \emph{Bounded objects}: $\mathcal{S}^b\coloneqq \mathcal{S}^-\cap\mathcal{S}^+$;
\item \emph{Compact objects}: $\mathcal{S}^c$;
\item \emph{Pseudo-compact objects}: $
\mathcal{S}^-_c\coloneqq \bigcap_{m=1}^{\infty}\big(\mathcal{S}^c*\mathcal{S}^{\leq-m}\big)$, that is, an object $F\in\mathcal{S}$ belongs to $\mathcal{S}^-_c$ if, for any integer $m>0$, there is a triangle $E\ra F\ra D\ra E[1]$ in $\mathcal{S}$ with $E\in\mathcal{S}^c$ and $D\in\mathcal{S}^{\leq-m}$;
\item \emph{Bounded pseudo-compact objects}: $\mathcal{S}^b_c\coloneqq\mathcal{S}^-_c\cap\mathcal{S}^b$.
\end{itemize}

The above subcategories are intrinsic in the sense that they are independent of the choice of the $t$-structures in the preferred equivalence class (see \cite{chns24, cns24,n18}). Moreover, by Definition \ref{defn-app} and \cite[Corollary 3.10]{bnp23}, we have the following result.

\begin{lem}\label{PR}
$(1)$ $\Hom_\mathcal{S}(G,G[i])=0$ for $i\gg 0$.

$(2)$ $
\mathcal{S}^-\big(\text{resp. } \mathcal{S}^+, \mathcal{S}^b \big)\,=\{X\in \mathcal{S}\mid \Hom_{\mathcal{S}}(G[i],X)=0 \;\text{for}\; i\ll 0\; (\text{resp. } i\gg 0, |i|\gg 0)\}.$
\end{lem}

We give an example to explain the above six subcategories of $\mathcal{S}$.

Let $\mathcal{S}:=\mathscr{D}(S\Modcat)$ for a ring $S$ and let $G:={_SS}$. Then
$$\mathcal{S}^-=\mathscr{D}^{-}(S\Modcat),\;\; \mathcal{S}^+=\mathscr{D}^{+}(S\Modcat), \;\; \mathcal{S}^b=\mathscr{D}^b(S\Modcat),$$
$$\mathcal{S}^c=\mathscr{K}^b(\pmodcat{S}), \;\; \mathcal{S}^-_c=\mathscr{K}^-(\pmodcat{S}), \;\; \mathcal{S}^b_c=\mathscr{K}^{-,b}(\pmodcat{S}).$$

The following result will be used in our later proofs.

\begin{lem}\label{lem-gluing-$t$-struc}{\rm\cite[Corollary 3.12]{bnp23}}
Let the following diagram be a recollement of weakly approximable triangulated categories
\begin{align*}
\xymatrixcolsep{4pc}\xymatrix{\mathcal{R} \ar[r]|{i_*=i_!} &\mathcal{S} \ar@<-2ex>[l]|{i^*} \ar@<2ex>[l]|{i^!} \ar[r]|{j^!=j^*}  &\mathcal{T}. \ar@<-2ex>[l]|{j_!} \ar@<2ex>[l]|{j_{*}}
}
\end{align*}
Then the glued $t$-structures on $\mathcal{S}$ associated with the $t$-structures on $\mathcal{R}$ and $\mathcal{T}$ in the preferred equivalence classes are in the preferred equivalence class.
Moreover, the recollement can be restricted to a left and right recollements, respectively:

$$
\xymatrixcolsep{4pc}\xymatrix{\mathcal{R}^- \ar[r]|{i_*=i_!} &\mathcal{S}^- \ar@<-2ex>[l]|{i^*}  \ar[r]|{j^!=j^*}  &\mathcal{T}^- \ar@<-2ex>[l]|{j_!}}
\quad\mbox{and}\quad
\xymatrixcolsep{4pc}\xymatrix{\mathcal{R}^+ \ar[r]|{i_*=i_!} &\mathcal{S}^+  \ar@<2ex>[l]|{i^!} \ar[r]|{j^!=j^*}  &\mathcal{T}^+. \ar@<2ex>[l]|{j_{*}}}
$$
In particular, $i_*$ can be restricted to $\mathcal{R}^b\ra \mathcal{S}^b$ and $j^*$ can be restricted to $\mathcal{S}^b \ra \mathcal{T}^b$.
\end{lem}

\section{Strongly bounded objects in weakly approximable triangulated categories}\label{3}

In this section, we concentrate on strongly bounded objects in a compactly generated triangulated category
and establish some basic properties of these objects. The main result of this section is Theorem \ref{prop-rec-res1}, which generalizes \cite[Proposition 4]{k91} from unbounded derived categories of rings to weakly approximable triangulated categories.

\begin{defn}\label{defn-sp}
{\rm \cite[Definition 8.5]{n25}}
Let $\mathcal{S}$ be a compactly generated triangulated category with a compact generator $G$. We define
$$
\mathcal{S}^{{\rm sb}}\coloneqq\bigcup_{i\ge 0}\overline{\langle G\rangle}^{[-i,i]}=\bigcup_{i\ge 0}\text{smd}(\text{Coprod}(G[-i,i])).
$$
The objects in $\mathcal{S}^{{\rm sb}}$ are called \emph{strongly bounded} objects in $\mathcal{S}$.
\end{defn}

Note that $\mathcal{S}^{{\rm sb}}$ does not depend on the choice of compact generators of $\mathcal{S}$. When $\mathcal{S}=\mathscr{D}(S\Modcat)$ for a ring $S$, we see that $\mathcal{S}^{{\rm sb}}$ coincides with the homotopy category $\mathscr{K}^b(\Pmodcat{S})$ of bounded complexes of projective $S$-modules. Thus $\mathcal{S}^{{\rm sb}}$ can be regarded as a natural analogue of $\mathscr{K}^b(\Pmodcat{S})$ in the framework of compactly generated triangulated categories.

\begin{lem}\label{some lems}
Let $\mathcal{S}$ be a compactly generated triangulated category with a compact generator $G$. Then:

$(1)$ The category $\mathcal{S}^{{\rm sb}}$ is a thick subcategory of $\mathcal{S}$.

$(2)$ If $\Hom_{\mathcal{S}}(G,G[i])=0$ for $i\gg 0$, then $\mathcal{S}^{{\rm sb}}\subseteq\bigcup_{m\in\mathbb{Z}}{}^{\perp}\mathcal{S}^{\leq m}$, where $(\mathcal{S}^{\leq 0},\mathcal{S}^{\geq 1})$ denotes a $t$-structure on $\mathcal{S}$ in the preferred equivalence class.

$(3)$ Let $\mathcal{R}$ be a compactly generated triangulated category with a compact generator $X$. If a triangle functor $\mathbf{F}: \mathcal{R} \to \mathcal{S}$ preserves coproducts with  $\mathbf{F}(X)\in\mathcal{S}^{{\rm sb}}$, then $\mathbf{F}$  can be restricted to a triangle functor $\mathcal{R}^{{\rm sb}}\ra \mathcal{S}^{{\rm sb}}$.
\end{lem}
\begin{proof}
(1) By definition, $\Coprod(G[-k,k])[1] = \Coprod(G[-k-1,k-1])$ for $k \geq 0$. This implies $\overline{\langle G\rangle}^{[-k,k]}[1]=\overline{\langle G\rangle}^{[-k-1,k-1]}$. Thus $\mathcal{S}^{{\rm sb}}\subseteq \mathcal{S}$ is closed under $[1]$. Similarly, $\mathcal{S}^{{\rm sb}}\subseteq \mathcal{S}$ is closed under $[-1]$. Clearly, $\mathcal{S}^{{\rm sb}}\subseteq \mathcal{S}$  is closed under direct summands and extensions. Thus $\mathcal{S}^{{\rm sb}}$ is a thick subcategory of $\mathcal{S}$.

(2) It suffices to show $(2)$ for the $t$-structure on $\mathcal{S}$ generated by $G$ (see Lemma \ref{lem-p-w-gen}). In the following, we prove that, for any integer $A>0$, there exists an integer $B>0$ such that $\overline{\langle G\rangle}^{[-A, A]} \subseteq {}^{\perp}(\mathcal{S}_{G}^{\leq -B})$.

In fact, since $\Hom_{\mathcal{S}}(G,G[i])=0$ for $i \gg 0$, there exists an integer $B>0$ such that $G(-\infty, -B] \subseteq G[-A, A]^{\perp}$. Moreover, since $G[-A,A]^{\perp}$ is closed under extensions and coproducts in $\mathcal{S}$, we have $\mathcal{S}_{G}^{\leq -B} \subseteq G[-A, A]^{\perp}$ by Lemma \ref{lem-p-w-gen}. Note that ${}^{\perp}\mathcal{S}_{G}^{\leq -B}$ is also closed under extensions and coproducts in $\mathcal{S}$. Thus $\overline{\langle G\rangle}^{[-A, A]} \subseteq {}^{\perp}\mathcal{S}_{G}^{\leq -B}$.

(3) Suppose that $\mathbf{F}$ preserves coproducts and $\mathbf{F}(X)\in \overline{\langle G\rangle}^{[\alpha,\beta]}$ for some integers $\alpha\leq \beta$. Let $M\in \mathcal{R}^{{\rm sb}}:=\bigcup_{i\ge 0}\overline{\langle X\rangle}^{[-i,i]}$.
Then $M\in \overline{\langle X\rangle}^{[-B,B]}$ for some $B\ge 0$.
We claim  $\mathbf{F}(M)\in \overline{\langle G\rangle}^{[\alpha-B,\beta+B]}.$

To show this, we define $$\mathscr{X}\coloneqq \{Y\in \mathcal{R}\mid \mathbf{F}(Y)\in \overline{\langle G\rangle}^{[\alpha-B,\beta+B]}\}.$$
By assumption, $\mathbf{F}(X)\in \overline{\langle G\rangle}^{[\alpha,\beta]}$. This implies
$\mathbf{F}(X[-i])\in \overline{\langle G\rangle}^{[\alpha+i,\beta+i]} \subseteq \overline{\langle G\rangle}^{[\alpha-B,\beta+B]}$ for  $-B\le i\le B$, and therefore $X[-i]\in \mathscr{X}$. Since $\overline{\langle G\rangle}^{[\alpha-B,\beta+B]}\subseteq\mathcal{S}$ is closed under coproducts and $\mathbf{F}$ preserves coproducts, $\mathscr{X}\subseteq\mathcal{R}$ is closed under coproducts. Moreover, since $\overline{\langle G\rangle}^{[\alpha-B,\beta+B]}\subseteq\mathcal{S}$ is closed under extensions, $\mathscr{X}\subseteq \mathcal{R}$ is closed under extensions. Thus $\overline{\langle X\rangle}^{[-B,B]}\subseteq\mathscr{X}$. In particular, $M\in \mathscr{X}$, and therefore $\mathbf{F}(M)\in \overline{\langle G\rangle}^{[\alpha-B,\beta+B]}\subseteq \mathcal{S}^{{\rm sb}}$.
This shows that $\mathbf{F}$ maps $\mathcal{R}^{{\rm sb}}$ into $\mathcal{S}^{{\rm sb}}$.
\end{proof}

\begin{lem}\label{lem-vanish}
Consider a recollement of weakly approximable triangulated categories
\begin{align*}
\xymatrixcolsep{4pc}\xymatrix{\mathcal{R} \ar[r]|{i_*=i_!} &\mathcal{S} \ar@<-2ex>[l]|{i^*} \ar@<2ex>[l]|{i^!} \ar[r]|{j^!=j^*}  &\mathcal{T} \ar@<-2ex>[l]|{j_!} \ar@<2ex>[l]|{j_{*}}
}
\end{align*}with compact generators $G_{\mathcal{R}}$, $G_{\mathcal{S}}$, and $G_{\mathcal{T}}$ for $\mathcal{R}$, $\mathcal{S}$, and $\mathcal{T}$, respectively. Suppose that $G_{\mathcal{S}}$ is bounded. Then the following statements hold.

{\rm (1)} $G_{\mathcal{T}}$ is bounded.

{\rm (2)} If $i_{*}(G_{\mathcal{R}})\in\mathcal{S}^{{\rm sb}}$, then $G_{\mathcal{R}}$ is bounded.
\end{lem}
\begin{proof}
Assume that there exists $n\in\mathbb{Z}$ such that $\Hom_{\mathcal{S}}(G_{\mathcal{S}},G_{\mathcal{S}}[i])=0$ for all $i\leq n$.

(1) By Lemma \ref{lem-key}(1), $j_!$ can restrict to $\mathcal{T}^c\ra \mathcal{S}^c$. Then there are $u\le v\in \mathbb{Z}$ such that $j_!(G_{\mathcal{T}})\in \langle G_{\mathcal{S}}\rangle^{[u,v]}$.
Then, for any $Y_1,Y_2\in \langle G_{\mathcal{S}}\rangle^{[u,v]}$, we have $\Hom_{\mathcal{S}}(Y_1,Y_2[i])=0$ for $i\le n+u-v$.
Since $j_!(G_{\mathcal{T}})\in \langle G_{\mathcal{S}}\rangle^{[u,v]}$ and $j_!$ is fully faithful, it holds that $\Hom_{\mathcal{T}}(G_{\mathcal{T}},G_{\mathcal{T}}[i])
\simeq\Hom_{\mathcal{S}}(j_!(G_{\mathcal{T}}),j_!(G_{\mathcal{T}})[i])=0$ for $i\le n+u-v$. Thus $G_{\mathcal{T}}$ is bounded.

(2) Assume $i_{*}(G_{\mathcal{R}})\in\overline{\langle G_{\mathcal{S}}\rangle}^{[-u,u]}$ for some $u\geq 0$. Let $X\in\overline{\langle G_{\mathcal{S}}\rangle}^{[-u,u]}$. Since $G_{\mathcal{S}}$ is compact, it follows from \cite[Lemma 1.8]{n21} that, for each $i\in\mathbb{Z}$, every morphism in $\Hom_{\mathcal{S}}(G_{\mathcal{S}},X[i])$ factors through an object in $\langle G_{\mathcal{S}}\rangle^{[-u,u]}[i]$. Moreover, since $\Hom_{\mathcal{S}}(G_{\mathcal{S}},G_{\mathcal{S}}[i])=0$ for all $i\le n$, we have $\langle G_{\mathcal{S}}\rangle^{[-u,u]}[i]\subseteq G_{\mathcal{S}}^{\perp}$ for $i\leq n-u$. It follows that $\Hom_{\mathcal{S}}(G_{\mathcal{S}},X[i])=0$ for $i\leq n-u$, and therefore $\overline{\langle G_{\mathcal{S}}\rangle}^{[-u,u]}[i]\subseteq G_{\mathcal{S}}^{\perp}$ for $i\leq n-u$. Thus $\overline{\langle G_{\mathcal{S}}\rangle}^{[-u,u]}[j]\subseteq G_{\mathcal{S}}[-u,u]^{\perp}$ for $j\leq n-2u$. Now, for any fixed $j\le n-2u$, let $Y\in\overline{\langle G_{\mathcal{S}}\rangle}^{[-u,u]}[j]$. Then $G_{\mathcal{S}}[-u,u]\subseteq {}^{\perp}Y$. Since ${}^{\perp}Y\subseteq \mathcal{S}$ is closed under coproducts and extensions, it is clear that $\overline{\langle G_{\mathcal{S}}\rangle}^{[-u,u]}\subseteq{}^{\perp}Y$. This forces $\Hom_{\mathcal{S}}(X,Y)=0$.
Since $i_*$ is fully faithful and $i_{*}(G_{\mathcal{R}})\in\overline{\langle G_{\mathcal{S}}\rangle}^{[-u,u]}$, we see that $\Hom_{\mathcal{R}}(G_{\mathcal{R}},G_{\mathcal{R}}[j])\simeq \Hom_{\mathcal{S}}(i_{*}(G_{\mathcal{R}}),i_{*}(G_{\mathcal{R}})[j])=0$ for all $j\le n-2u$. Thus $G_{\mathcal{R}}$ is bounded.
\end{proof}

In the following, we give a new characterization of $\mathcal{S}^{{\rm sb}}$ in terms of the vanishing of Hom-spaces.

\begin{prop}\label{lem-sp=big}
Let $\mathcal{S}$ be a weakly approximable triangulated category with a compact generator $G$ and $(\mathcal{S}^{\leq 0},\mathcal{S}^{\geq 1})$ a $t$-structure in the preferred equivalence class.
Then
$$
\mathcal{S}^{{\rm sb}}=\{X\in \mathcal{S}^-\mid \forall\; Z\in \mathcal{S}^-,\; \Hom_{\mathcal{S}}(X, Z[i])=0 \text{ for }i\gg 0 \}=\bigcup_{m\in\mathbb{Z}}\big(\mathcal{S}^-\cap {^\perp}\mathcal{S}^{\le m}\big).
$$

If $G$ is bounded, then
$$\mathcal{S}^{{\rm sb}}=\{X\in \mathcal{S}^b\mid \forall\;  Z\in \mathcal{S}^b,\;\Hom_{\mathcal{S}}(X, Z[i])=0 \text{ for }i\gg0\}=\bigcup_{m\in\mathbb{Z}}\big(\mathcal{S}^b\cap {^\perp}\mathcal{S}^{\le m}\big).\\$$
\end{prop}
\begin{proof}
Let $\mathscr{X}\coloneqq \{X\in \mathcal{S}^-\mid \forall\; Z\in \mathcal{S}^-,\; \Hom_{\mathcal{S}}(X, Z[i])=0 \text{ for }i\gg 0 \}$ and
$\mathscr{Y}\coloneqq \bigcup_{m\in\mathbb{Z}}\big(\mathcal{S}^-\cap {^\perp}\mathcal{S}^{\le m}\big)$.
Clearly, $\mathscr{Y} \subseteq \mathscr{X}$. We now show  $\mathcal{S}^{{\rm sb}}\subseteq \mathscr{Y}$.
By Lemmas \ref{PR}(1) and \ref{some lems}(2), it suffices to show $\mathcal{S}^{{\rm sb}}\subseteq\mathcal{S}^{-}$.

Let $M \in \mathcal{S}^{{\rm sb}}$. Then there exists an integer $r > 0$ such that $M \in \overline{\langle G \rangle}^{[-r, r]}$. Since $G[i]$ is compact for all $i \in \mathbb{Z}$, it follows from \cite[Lemma 1.8]{n21} that every morphism in $\Hom_{\mathcal{S}}(G[i], M)$ factors through an object in $\coprods(G[-r, r])$. As $\Hom_{\mathcal{S}}(G[i], G) = 0$ for $i \ll 0$ by Lemma \ref{PR}(1), we have $\Hom_{\mathcal{S}}(G[i], M) = 0$ for $i \ll 0$. Moreover, by Lemma~\ref{PR}(2),
$\mathcal{S}^-=\{X\in \mathcal{S}\mid \Hom_{S}(G[i],X)=0 \;\text{for}\; i\ll 0\}$. Thus $M \in \mathcal{S}^{-}$. This shows $\mathcal{S}^{{\rm sb}} \subseteq \mathcal{S}^{-}$, and therefore $\mathcal{S}^{{\rm sb}}\subseteq \mathscr{Y}$.

In the following, we prove $\mathscr{X} \subseteq \mathcal{S}^{{\rm sb}}$.

Let $X \in \mathscr{X}$. Then $X \in \mathcal{S}^{-}$. Our aim is to show $X\in \mathcal{S}^{{\rm sb}}$.
Note that $\mathcal{S}^{{\rm sb}}$ is a triangulated subcategory of $\mathcal{S}$ by Lemma \ref{some lems}(1). So, without loss of generality, assume $X \in \mathcal{S}^{\le 0}$. By Lemma~\ref{lem-app-induction}, we can construct a series of triangles in $\mathcal{S}$:
$$E_1\ra X\xrightarrow{f_1}D_1\ra E_1[1],\\ $$
$$E_2\ra X\xrightarrow{f_2} D_2\ra E_2[1],\\  $$
$$\cdots\\$$
$$E_m\ra X\xrightarrow{f_m}D_m\ra E_m[1],\\$$
$$\cdots$$
where $E_m \in \overline{\langle G\rangle}^{[1-m-a(G),a(G)]}$ and $D_m \in \mathcal{S}^{\le -m}$ for $m \ge 1$. We claim that there exists an integer $k \geq 1$ with $f_k = 0$. If this is the case, then $X$ is a direct summand of $E_k \in \mathcal{S}^{{\rm sb}}$, and thus $X \in \mathcal{S}^{{\rm sb}}$ by Lemma~\ref{some lems}(1). This shows $\mathscr{X} \subseteq \mathcal{S}^{{\rm sb}}$.

Conversely, suppose $f_k \neq 0$ for all $k \geq 1$. Let $D \coloneqq \bigoplus_{i \geq 1} D_i[-i]$. Then
$D \in \mathcal{S}^{\leq 0} \subseteq \mathcal{S}^{-}$, and for each $k \geq 1$, $D_k$ is a direct summand of $D[k]$. Since $f_k \neq 0$ for all $k\geq 1$, it follows that $\Hom_{\mathcal{S}}(X, D[k]) \neq 0$ for all $k \geq 1$, contradicting $X \in \mathscr{X}$.

Now, assume that $G$ is bounded (that is, $\Hom_{\mathcal{S}}(G[i], G) = 0$ for $i \gg 0$). Then for any $M\in\overline{\langle G\rangle}^{[-r.r]}$, we have $\Hom_{\mathcal{S}}(G[i],M)=0$ for $i\gg0$. Further, by Lemma~\ref{PR}(2), $M\in\mathcal{S}^{+}$. This implies $\mathcal{S}^{{\rm sb}}\subseteq\mathcal{S}^{+}$, and thus $\mathcal{S}^{{\rm sb}}\subseteq\mathcal{S}^{b}$. It follows that $$\mathcal{S}^{{\rm sb}}=\mathcal{S}^{{\rm sb}}\cap\mathcal{S}^{b}=\bigcup_{m\in\mathbb{Z}}\big(\mathcal{S}^-\cap {^\perp}\mathcal{S}^{\le m}\big)\cap\mathcal{S}^{b}=\bigcup_{m\in\mathbb{Z}}\big(\mathcal{S}^b\cap {^\perp}\mathcal{S}^{\le m}\big).$$
It remains to show $\mathcal{S}^{{\rm sb}}$ is the same as the following category:
$$\mathscr{Z}\coloneqq\{X\in \mathcal{S}^b\mid \forall\; Z\in\mathcal{S}^{b},\;\Hom_{\mathcal{S}}(X,Z[i])=0\text{ for }i\gg 0 \}.$$

By $\mathscr{X}=\mathcal{S}^{{\rm sb}}\subseteq\mathcal{S}^{b}$, we obtain $\mathcal{S}^{{\rm sb}}\subseteq\mathscr{Z}$. To show the converse inclusion, let $X'\in\mathscr{Z}$. Clearly, the definition of $\mathscr{Z}$ depends on $\mathcal{S}^b$, which is independent of the choice of $t$-structures up to equivalence. So, we can assume  $(\mathcal{S}^{\leq 0},\mathcal{S}^{\geq 1})=(\mathcal{S}_{G}^{\leq 0},\mathcal{S}_{G}^{\geq 1})$.  Note that both $\mathcal{S}^{{\rm sb}}$ and $\mathscr{Z}$ are triangulated subcategories of $\mathcal{S}$. So, without loss of generality, assume $X'\in\mathcal{S}_{G}^{\geq-s}\cap\mathcal{S}_{G}^{\leq 0}$ for some $s>0$. By Lemma~\ref{lem-app-induction}, for each $m\geq 1$, there is a triangle $E'_{m}\to X'\xrightarrow{f_{m}'} D'_{m}\to E'_{m}[1]$ in $\mathcal{S}$ with $E'_m \in \overline{\langle G\rangle}^{[1-m-a(G),a(G)]}$ and $D'_m \in \mathcal{S}_{G}^{\le -m}$. Since $\Hom_{\mathcal{S}}(G, G[i]) = 0$ for $i\ll 0$, there exists $t>0$ such that $\Hom_{\mathcal{S}}(G, G[i]) = 0$ for all $i\leq-t$. Combining this with Lemma \ref{lem-p-w-gen}, we have
\[G[1-m-A,A]\subseteq G(-\infty,1-m-A-t]^{\perp}=\mathcal{S}^{\geq 2-m-A-t}_{G}.\]
Since $\mathcal{S}^{\geq 2-m-A-t}_{G}\subseteq \mathcal{S}$ is closed under extensions and coproducts,
it follows that
$\overline{\langle G\rangle}^{[1-m-A,A]}\subseteq \mathcal{S}^{\geq 2-m-A-t}_{G}$. This leads to  $E'_{m}\in\mathcal{S}^{\geq 2-m-A-t}_{G}$ for all $m\geq 1$. Further, by the triangle \[E'_{m}[-m]\to X'[-m]\to D'_{m}[-m]\to E'_{m}[-m+1],\]we have $X'[-m]\in\mathcal{S}_{G}^{\geq m-s}\subseteq\mathcal{S}_{G}^{\geq -s}$ and $E'_{m}[-m+1]\in\mathcal{S}_{G}^{\geq 1-A-t}$. Hence, $D'_{m}[-m]\in\mathcal{S}_{G}^{\geq\min\{-s,1-A-t\}}\cap\mathcal{S}_{G}^{\leq 0}$. Now, let $D' \coloneqq \bigoplus_{i \geq 1} D'_i[-i]$. Then $D'\in\mathcal{S}^{b}$. Similarly, we can show that $f_{k}' = 0$ for some integer $k > 0$. Then $X'$ is a direct summand of $E'_{k}$, and further, lies in $\mathcal{S}^{{\rm sb}}$. This verifies $\mathscr{Z} \subseteq \mathcal{S}^{{\rm sb}}$, and thus $\mathscr{Z}= \mathcal{S}^{{\rm sb}}$.
\end{proof}

\begin{remark}\label{lem-app-sb}
Let $\mathcal{S}$ be an approximable triangulated category with a compact generator $G$ and let $\mathscr{W}=\bigcup_{i\geq 0}\overline{\langle G\rangle}_{i+1}^{[-i,i]}$. Using the same method as in Proposition \ref{lem-sp=big}, we can show that $\mathscr{W}=\bigcup_{m\in\mathbb{Z}}(\mathcal{S}^{-}\cap\mathcal{}^{\perp}\mathcal{S}^{\leq m})$. Thus $\mathcal{S}^{{\rm sb}}=\mathscr{W}=\bigcup_{i\geq 0}\overline{\langle G\rangle}_{i+1}^{[-i,i]}$.
\end{remark}

\begin{lem}\label{lem-r-func1}
Let $\mathcal{R}$ and $\mathcal{S}$ be weakly approximable triangulated categories
with compact generators $G_{\mathcal{R}}$ and $G_{\mathcal{S}}$, respectively.
Let $\mathbf{F}\colon \mathcal{R}\ra \mathcal{S}$ be a triangle functor with a right adjoint functor $\mathbf{H}\colon \mathcal{S}\ra \mathcal{R}$.
Suppose that $\mathbf{F}$ can be restricted to $\mathcal{R}^-\ra \mathcal{S}^-$.
The following statements are equivalent.

{\rm (1)} $\mathbf{F}(G_{\mathcal{R}})\in \mathcal{S}^{{\rm sb}}$.

{\rm (2)} $\mathbf{H}$ can be restricted to $\mathcal{S}^-\ra \mathcal{R}^-$.

{\rm (3)} $\mathbf{F}$ can be restricted to $\mathcal{R}^{{\rm sb}}\ra \mathcal{S}^{{\rm sb}}$.

Moreover, each of the above conditions implies:

{\rm (4)} $\mathbf{H}$ can be restricted to $\mathcal{S}^b\ra \mathcal{R}^b$.

\noindent If $G_{\mathcal{R}}$ and $G_{\mathcal{S}}$ are bounded and $\mathbf{F}$ can be restricted to $\mathcal{R}^{b}\ra \mathcal{S}^{b}$, then {\rm (4)} is also equivalent to {\rm (1)–(3)}.
\end{lem}
\begin{proof}
(1) $\Rightarrow$ (2) Suppose that $\mathbf{F}(G_{\mathcal{R}})\in \overline{\langle G_{\mathcal{S}}\rangle}^{[-u,u]}$ for some $u>0$. By Lemma~\ref{PR}(2),
for any $Y \in \mathcal{S}^-$, there exists an integer $n$ such that $\Hom_{\mathcal{S}}(G_{\mathcal{S}}[i],Y)=0$ for $i\le n$.
Note that $^{\perp}Y\subseteq\mathcal{S}$ is closed under coproducts and extensions.
Hence, for any $Y' \in \overline{\langle G_{\mathcal{S}}\rangle}^{[-u,u]}$ and  $i \le n - u$, we have $\Hom_{\mathcal{S}}(Y'[i], Y) = 0$. By the adjunction isomorphism,
$\Hom_{\mathcal{R}}(G_{\mathcal{R}}[i],\mathbf{H}(Y))\simeq \Hom_{\mathcal{S}}(\mathbf{F}(G_{\mathcal{R}})[i],Y)=0$ for all $i\le n-u$.
It follows from Lemma \ref{PR}(2) that $\mathbf{H}(Y) \in \mathcal{R}^-$. Thus $\mathbf{H}$ restricts to a functor $\mathcal{S}^- \to \mathcal{R}^-$.

(2) $\Rightarrow$ (3) Let $X \in \mathcal{R}^{{\rm sb}}$ and $Y' \in \mathcal{S}^{-}$. By Proposition~\ref{lem-sp=big}, $\mathcal{R}^{{\rm sb}} \subseteq \mathcal{R}^{-}$. Moreover, by the adjunction isomorphism, $\Hom_{\mathcal{S}}(\mathbf{F}(X),Y'[i])\simeq \Hom_{\mathcal{R}}(X,\mathbf{H}(Y')[i])$ for $i\in\mathbb{Z}$. By $(2)$, $\mathbf{H}(Y')\in \mathcal{R}^-$.
It follows from $X\in \mathcal{R}^{{\rm sb}}$ that $\Hom_{\mathcal{R}}(X,\mathbf{H}(Y')[i])=0$ for $i\gg0$ by Proposition~\ref{lem-sp=big}.
Then $\Hom_{\mathcal{S}}(\mathbf{F}(X),Y'[i])=0$ for $i\gg 0$. Note that $\mathbf{F}(X)\in\mathcal{S}^{-}$ by assumption on $\mathbf{F}$.
Then $\mathbf{F}(X)\in \mathcal{S}^{{\rm sb}}$ by Proposition~\ref{lem-sp=big}. Thus $\mathbf{F}$ restricts to a functor $\mathcal{R}^{{\rm sb}}\ra \mathcal{S}^{{\rm sb}}$.

Clearly, (3) $\Rightarrow$ (1).

(2) $\Rightarrow$ (4) It suffice to show that $\mathbf{H}$ can be restricted to $\mathcal{S}^{+}\to\mathcal{R}^{+}$. Let $X \in \mathcal{S}^{+}$. For each $i \in \mathbb{Z}$, there is an adjunction isomorphism $\Hom_{\mathcal{S}}(\mathbf{F}(G_{\mathcal{R}})[i],X)\simeq\Hom_{\mathcal{R}}(G_{\mathcal{R}}[i],\mathbf{H}(X))$. By assumption on $\mathbf{F}$, we have $\mathbf{F}(G_{\mathcal{R}})\in\mathcal{S}^{-}$.
Since $\Hom_{\mathcal{S}}(\mathcal{S}^{\leq 0},\mathcal{S}^{\geq 1})=0$, it follows that $\Hom_{\mathcal{S}}(\mathbf{F}(G_{\mathcal{R}})[i], X)=0$ for $i\gg0$.
This leads to $\Hom_{\mathcal{R}}(G_{\mathcal{R}}[i],\mathbf{H}(X))=0$ for $i\gg 0$. Thus $\mathbf{H}(X)\in \mathcal{R}^{+}$ by Lemma \ref{PR}(2).

Now, suppose that $G_{\mathcal{R}}$ and $G_{\mathcal{S}}$ are bounded, and that $\mathbf{F}(\mathcal{R}^{b})\subseteq \mathcal{S}^{b}$. Then $\mathcal{R}^{{\rm sb}}\subseteq \mathcal{R}^b$ and $\mathcal{S}^{{\rm sb}}\subseteq \mathcal{S}^b$ by Proposition~\ref{lem-sp=big}.

(4) $\Rightarrow$ (3) Let $X\in \mathcal{R}^{{\rm sb}}$ and $Y'\in \mathcal{S}^b$.
For each $i\in\mathbb{Z}$, there is an isomorphism
$\Hom_{\mathcal{S}}(\mathbf{F}(X),Y'[i])\simeq \Hom_{\mathcal{R}}(X,\mathbf{H}(Y')[i])$.
Moreover, $\mathbf{H}(Y') \in \mathcal{R}^b$ by $(4)$. Since $X \in \mathcal{R}^{\mathrm{sb}}$, we see from Proposition~\ref{lem-sp=big} that $\Hom_{\mathcal{R}}(X, \mathbf{H}(Y')[i]) = 0$ for $i\gg0$. Hence, $\Hom_{\mathcal{S}}(\mathbf{F}(X), Y'[i]) = 0$ for $i \gg 0$. Note that $\mathbf{F}(X)\in\mathcal{S}^{b}$ by $\mathbf{F}(\mathcal{R}^{b})\subseteq \mathcal{S}^{b}$. Thus Proposition~\ref{lem-sp=big} implies $\mathbf{F}(X)\in \mathcal{S}^{{\rm sb}}$. This shows that $\mathbf{F}$ restricts to a functor  $\mathcal{R}^{{\rm sb}}\ra \mathcal{S}^{{\rm sb}}$.
\end{proof}

The following theorem is a generalization of \cite[Proposition 4]{k91} which deals with the case of unbounded derived categories of rings.

\begin{thm}\label{prop-rec-res1}
Let the following diagram be a recollement of weakly approximable triangulated categories
\begin{align*}
\xymatrixcolsep{4pc}\xymatrix{\mathcal{R} \ar[r]|{i_*=i_!} &\mathcal{S} \ar@<-2ex>[l]|{i^*} \ar@<2ex>[l]|{i^!} \ar[r]|{j^!=j^*}  &\mathcal{T}. \ar@<-2ex>[l]|{j_!} \ar@<2ex>[l]|{j_{*}}
}
\end{align*}
Let $G_{\mathcal{R}}$ and $G_{\mathcal{S}}$ are compact generators of $\mathcal{R}$ and $\mathcal{S}$, respectively. The following statements are equivalent.

{\rm (1)} $i_{*}(G_{\mathcal{R}})\in \mathcal{S}^{{\rm sb}}$.

{\rm ($1'$)} $j^{*}(G_{\mathcal{S}})\in \mathcal{T}^{{\rm sb}}$.

{\rm (2)} The recollement induces a left recollement
\begin{align*}
\xymatrixcolsep{4pc}\xymatrix{
\mathcal{R}^{{\rm sb}} \ar[r]|{i_*=i_!}
&\mathcal{S}^{{\rm sb}} \ar@<-2ex>[l]|{i^*}   \ar[r]|{j^!=j^*}
&\mathcal{T}^{{\rm sb}}. \ar@<-2ex>[l]|{j_!}
}
\end{align*}

{\rm (3)} The recollement induces a recollement
\begin{align*}
\xymatrixcolsep{4pc}\xymatrix{\mathcal{R}^- \ar[r]|{i_*=i_!} &\mathcal{S}^- \ar@<-2ex>[l]|{i^*} \ar@<2ex>[l]|{i^!} \ar[r]|{j^!=j^*}  &\mathcal{T}^-. \ar@<-2ex>[l]|{j_!} \ar@<2ex>[l]|{j_{*}}
}
\end{align*}

Moreover, each of the above conditions implies:

{\rm (4)} The recollement induces a right recollement
\begin{align*}
\xymatrixcolsep{4pc}\xymatrix{
\mathcal{R}^b \ar[r]|{i_*=i_!}
&\mathcal{S}^b  \ar@<2ex>[l]|{i^!} \ar[r]|{j^!=j^*}
&\mathcal{T}^b.   \ar@<2ex>[l]|{j_{*}}
}
\end{align*}

\noindent Assume that $G_{\mathcal{R}}$ and $G_{\mathcal{S}}$ are bounded. Then {\rm (4)} is also equivalent to {\rm (1)–(3)}.
\end{thm}
\begin{proof}
In Theorem \ref{prop-rec-res1}, we only show the equivalence of $(1)$ and $(1')$ since other equivalences hold true by combining Lemma \ref{lem-gluing-$t$-struc} with Lemma \ref{lem-r-func1}.

(1) $\Rightarrow$ ($1'$) By Lemma \ref{lem-gluing-$t$-struc}, $j^{*}(G_{\mathcal{S}})\in\mathcal{T}^{-}$ and $j^*\colon \mathcal{S}^-\to \mathcal{T}^-$ is dense. It follows that, for any $Z' \in \mathcal{T}^{-}$, there exists an object $Y' \in \mathcal{S}^{-}$ such that $j^*(Y') \simeq Z'$.
By the adjoint pair $(j_!, j^{*})$, for each $i \in \mathbb{Z}$, there are isomorphisms:
$$\Hom_{\mathcal{T}}(j^{*}(G_{\mathcal{S}}),Z'[i])\simeq \Hom_{\mathcal{T}}(j^{*}(G_{\mathcal{S}}),j^*(Y')[i])
\simeq \Hom_{\mathcal{S}}(j_!j^{*}(G_{\mathcal{S}}),Y'[i]).$$ By Proposition~\ref{lem-sp=big},  to show $j^{*}(G_{\mathcal{S}})\in \mathcal{T}^{{\rm sb}}$, it suffices to show $j_{!}j^{*}(G_{\mathcal{S}})\in\mathcal{S}^{{\rm sb}}$. Note that there is a canonical triangle in $\mathcal{S}$:
$$j_!j^*(G_{\mathcal{S}})\ra G_{\mathcal{S}}\ra i_*i^*(G_{\mathcal{S}})\ra j_!j^*(G_{\mathcal{S}})[1].$$
By Lemma \ref{some lems}(3), the functor $i^*$ can be restricted to $\mathcal{S}^{{\rm sb}}\ra {\mathcal{R}}^{{\rm sb}}$. Moreover, since $i_{*}(G_{\mathcal{R}})\in \mathcal{S}^{{\rm sb}}$ by $(1)$, the functor $i_{*}$ can be restricted to ${\mathcal{R}}^{{\rm sb}}\ra \mathcal{S}^{{\rm sb}}$ still by Lemma \ref{some lems}(3). It follows from  $G_{\mathcal{S}}\in \mathcal{S}^{{\rm sb}}$ that $i_*i^*(G_{\mathcal{S}})\in \mathcal{S}^{{\rm sb}}$. By the above triangle, we have $j_!j^*(G_{\mathcal{S}})\in \mathcal{S}^{{\rm sb}}$. Thus $j^{*}(G_{\mathcal{S}})\in \mathcal{T}^{{\rm sb}}$.

($1'$) $\Rightarrow$ (1) By Lemma \ref{some lems}(3), the functor $j_!$ restricts to ${\mathcal{T}}^{{\rm sb}}\ra \mathcal{S}^{{\rm sb}}$. Since $ j^{*}(G_{\mathcal{S}}) \in \mathcal{T}^{{\rm sb}}$ by $(1')$, the functor $ j^{*} $ restricts to $ {\mathcal{S}}^{{\rm sb}} \to \mathcal{T}^{{\rm sb}}$ still by Lemma \ref{some lems}(3). Hence $ j_!j^*(G_{\mathcal{S}}) \in \mathcal{S}^{{\rm sb}} $. Using the triangle above and $ G_{\mathcal{S}} \in \mathcal{S}^{{\rm sb}} $, we have  $ i_*i^*(G_{\mathcal{S}}) \in \mathcal{S}^{{\rm sb}} $.
Since $i^*(G_{\mathcal{S}})$ is also a compact generator of $\mathcal{R}$, there exists an integer $u > 0$ such that $G_{\mathcal{R}} \in \langle i^*(G_{\mathcal{S}}) \rangle_{u}^{[-u, u]}$.
As $i_* i^*(G_{\mathcal{S}}) \in \mathcal{S}^{{\rm sb}}$ and $\mathcal{S}^{{\rm sb}}$ is a thick subcategory of $\mathcal{S}$ by Lemma \ref{some lems}(1), we obtain $i_*(\langle i^*(G_{\mathcal{S}})\rangle_{u}^{[-u,u]})\subseteq \mathcal{S}^{{\rm sb}}$. This implies $i_*(G_{\mathcal{R}})\in \mathcal{S}^{{\rm sb}}$.
\end{proof}

\section{Finitistic dimensions of triangulated categories}\label{4}

In this section, we consider two kinds of finitistic dimension for triangulated categories. Their finiteness is the same when restricted to weakly approximable triangulated categories (see Proposition \ref{Relation}).
Our aim is to compare (big) finitistic dimensions of triangulated categories in a recollement.

\subsection{Finitistic dimension in essentially small triangulated categories}\label{Homological dimension}

In this subsection, we give proofs of Theorem \ref{mainthm-fd-1} and Corollary \ref{cor to regular}. As a preparation, we establish the following lemma.

\begin{lem}\label{lem-fullyfaithful}
Let $\mathcal{U}$ and $\mathcal{V}$ be triangulated categories with objects $G_{\mathcal{U}} \in \mathcal{U}$ and $G_{\mathcal{V}} \in \mathcal{V}$.
Let $\mathbf{F} \colon \mathcal{U} \to \mathcal{V}$ be a triangle functor with a right adjoint $\mathbf{H} \colon \mathcal{V} \to \mathcal{U}$ such that $\mathbf{F}(G_{\mathcal{U}}) \in \langle G_{\mathcal{V}}\rangle$. If $\fd(\mathcal{U},G_{\mathcal{U}})=r<+\infty$, then $\mathbf{F}\mathbf{H}(X) \in \langle G_{\mathcal{V}}\rangle^{[0,+\infty)}[r+w_{G_{\mathcal{V}}}(\mathbf{F}(G_{\mathcal{U}}))]$ for any $X\in G_{\mathcal{V}}(-\infty,-1]^{\perp}$. In particular, if $\mathbf{H}$ is fully faithful, then $\fd(\mathcal{V},G_{\mathcal{V}})\leq\fd(\mathcal{U},G_{\mathcal{U}})+w_{G_{\mathcal{V}}}(\mathbf{F}(G_{\mathcal{U}})).$
\end{lem}
\begin{proof}
By assumption, $\mathbf{F}(G_{\mathcal{U}}) \in \langle G_{\mathcal{V}}\rangle^{[c,d]}$ for some $c \leq d$.
Let $X\in G_{\mathcal{V}}(-\infty,-1]^{\perp}$. Then $\Hom_{\mathcal{V}}(X',X)=0$ for any $X'\in \langle G_{\mathcal{V}}\rangle^{(-\infty,-1]}$.
Now, we claim that $\mathbf{H}(X)[-d]\in G_{\mathcal{U}}(-\infty,-1]^{\perp}$ in $\mathcal{U}$. Indeed, for any $i\geq 1$, \[\Hom_{\mathcal{U}}(G_{\mathcal{U}}[i],\mathbf{H}(X)[-d])\simeq \Hom_{\mathcal{U}}(G_{\mathcal{U}}[i+d],\mathbf{H}(X))\simeq \Hom_{\mathcal{V}}(\mathbf{F}(G_{\mathcal{U}})[i+d],X).\]
It follows from $\mathbf{F}(G_{\mathcal{U}})\in \langle G_{\mathcal{V}}\rangle^{[c,d]}$ that $\mathbf{F}(G_{\mathcal{U}})[i+d]\in \langle G_{\mathcal{V}}\rangle^{[c-i-d,-i]}\subseteq \langle G_{\mathcal{V}}\rangle^{(-\infty,-1]}$.
Then $$\Hom_{\mathcal{V}}(\mathbf{F}(G_{\mathcal{U}})[i+d],X)=0 \quad \text{and}\quad \Hom_{\mathcal{U}}(G_{\mathcal{U}}[i],\mathbf{H}(X)[-d])=0.$$
Thus $\mathbf{H}(X)[-d]\in G_{\mathcal{U}}(-\infty,-1]^{\perp}$ in $\mathcal{U}$.
By $\fd(\mathcal{U}, G_{\mathcal{U}})=r$, we have $\mathbf{H}(X)[-d]\in \langle G_{\mathcal{U}}\rangle^{[0,+\infty)}[r]$ and $\mathbf{H}(X)\in \langle G_{\mathcal{U}}\rangle^{[0,+\infty)}[r+d]$. It follows that $$\mathbf{F}\mathbf{H}(X)\in \langle \mathbf{F}(G_{\mathcal{U}})\rangle^{[0,+\infty)}[r+d]
\subseteq \langle G_{\mathcal{V}}\rangle^{[c,+\infty)}[r+d]=\langle G_{\mathcal{V}}\rangle^{[0,+\infty)}[r+d-c].$$
Clearly, if $\mathbf{H}$ is fully faithful, then $X \simeq \mathbf{F} \mathbf{H}(X)$ for any $X \in \mathcal{V}$. In this case, $X\in\langle G_{\mathcal{V}}\rangle^{[0,+\infty)}[r+d-c]$, and thus  $\fd(\mathcal{V},G_{\mathcal{V}})\leq\fd(\mathcal{U},G_{\mathcal{U}})+w_{G_{\mathcal{V}}}(\mathbf{F}(G_{\mathcal{U}})).$
\end{proof}

{\bf Proof of  Theorem \ref{mainthm-fd-1}.}
(1) If $\fd(\mathcal{R}, G_{\mathcal{R}}) = +\infty$ or $\fd(\mathcal{T}, G_{\mathcal{T}}) = +\infty$, then Theorem \ref{mainthm-fd-1} holds trivially. So, we assume $\fd(\mathcal{R}, G_{\mathcal{R}}) = r < +\infty$ and $\fd(\mathcal{T}, G_{\mathcal{T}}) = t < +\infty$. Recall from Definition \ref{notation}(4) that $\langle G_{\mathcal{S}} \rangle$ is the smallest full triangulated subcategory of $\mathcal{S}$ containing $G_{\mathcal{S}}$ and closed under direct summands. Now, let $i_{*}(G_{\mathcal{R}}) \in \langle G_{\mathcal{S}} \rangle^{[a, b]}$ with $w_{G_{\mathcal{S}}}(i_{*}(G_{\mathcal{R}})) = b - a$, and let $j_{!}(G_{\mathcal{T}}) \in \langle G_{\mathcal{S}} \rangle^{[c, d]}$ with $w_{G_{\mathcal{S}}}(j_{!}(G_{\mathcal{T}})) = d - c$. For each $Y\in\mathcal{S}$, there is a canonical triangle in $\mathcal{S}$:
\begin{align}\label{can-tri}
j_!j^*(Y)\ra Y\ra i_*i^*(Y)\ra j_!j^*(Y)[1].
\end{align}
Further, let $Y\in G_{\mathcal{S}}(-\infty,-1]^{\perp}$.
Then $j_!j^*(Y)\in \langle G_{\mathcal{S}}\rangle^{[0,+\infty)}[t+d-c]$ by Lemma \ref{lem-fullyfaithful}.  Define $u\coloneqq -o^{-}(G_{\mathcal{S}})-t-d+c$. Clearly, $u\leq 0$. In the following, we show $j_!j^*(Y)\in G_{\mathcal{S}}(-\infty,-1+u]^{\perp}$.

For $p\leq -1+u$ and $q=t+d-c$, we have $p+q\leq -1-o^{-}(G_{\mathcal{S}})<-o^{-}(G_{\mathcal{S}}).$
It follows from $\Hom_{\mathcal{S}}(G_{\mathcal{S}},G_{\mathcal{S}}[k])=0$ for $k<-o^{-}(G_{\mathcal{S}})$ that
$\Hom_{\mathcal{S}}(G_{\mathcal{S}}[-p],G_{\mathcal{S}}[q])
\simeq \Hom_{\mathcal{S}}(G_{\mathcal{S}},G_{\mathcal{S}}[p+q])=0.$
Thus $\langle G_{\mathcal{S}}\rangle^{[0,+\infty)}[q]\subseteq G_{\mathcal{S}}(-\infty,-1+u]^{\perp}$. Hence, \[j_!j^*(Y)\in G_{\mathcal{S}}(-\infty,-1+u]^{\perp}\text{ and }j_!j^*(Y)[1]\in G_{\mathcal{S}}(-\infty,-2+u]^{\perp}.\]
Since $Y\in G_{\mathcal{S}}(-\infty,-1]^{\perp}$, we have $Y\in G_{\mathcal{S}}(-\infty,-2+u]^{\perp}$.
Thus $i_*i^*(Y)\in G_{\mathcal{S}}(-\infty,-2+u]^{\perp}$ by (\ref{can-tri}).

Set $v\coloneqq u-b-1$. We now show $i^*(Y)[v]\in G_{\mathcal{R}}(-\infty,-1]^{\perp}$.

In fact, since $i_*$ is fully faithful, there are isomorphisms for $j\leq -1$:
$$\Hom_{\mathcal{R}}(G_{\mathcal{R}}[-j],i^*(Y)[v])
\simeq \Hom_{\mathcal{S}}(i_*(G_{\mathcal{R}})[-j],i_*i^*(Y)[v])
\simeq \Hom_{\mathcal{S}}(i_*(G_{\mathcal{R}})[-j-v],i_*i^*(Y)).$$
By $i_*(G_{\mathcal{R}})\in \langle G_{\mathcal{S}}\rangle^{[a,b]}$, we have
$$i_*(G_{\mathcal{R}})[-j-v]
\in \langle G_{\mathcal{S}}\rangle^{[a,b]}[-j-v]
= \langle G_{\mathcal{S}}\rangle^{[a+j+v,b+j+v]}
\subseteq \langle G_{\mathcal{S}}\rangle^{(-\infty,b+j+v]}
\subseteq \langle G_{\mathcal{S}}\rangle^{(-\infty,-2+u]}.$$
Note that $(\langle G_{\mathcal{S}}\rangle^{(-\infty,-2+u]})^{\perp}=G_{\mathcal{S}}(-\infty,-2+u]^{\perp}$.
Since $i_*i^*(Y)\in G_{\mathcal{S}}(-\infty,-2+u]^{\perp}$, it follows that $\Hom_{\mathcal{R}}(G_{\mathcal{R}}[-j],i^*(Y)[v])=0$ for $j\le -1$.
Thus $i^*(Y)[v]\in G_{\mathcal{R}}(-\infty,-1]^{\perp}$.

By $\fd(\mathcal{R}, G_{\mathcal{R}})=r<+\infty$, we have $i^*(Y)[v]\in \langle G_{\mathcal{R}}\rangle^{[0,+\infty)}[r]$.
It follows from $i_*(G_{\mathcal{R}})\in \langle G_{\mathcal{S}}\rangle^{[a,b]}$ that
$$i_*i^*(Y)\in \langle i_*(G_{\mathcal{R}})\rangle^{[0,+\infty)}[r-v]
\subseteq \langle G_{\mathcal{S}}\rangle^{[a,+\infty)}[r-v]=\langle G_{\mathcal{S}}\rangle^{[0,+\infty)}[r-v-a].$$
Clearly, $r-v-a=r+b+1-u-a=r+b-a+t+d-c+1+o^{-}(G_{\mathcal{S}})\geq t+d-c.$
Then \[j_!j^*(Y)\in \langle G_{\mathcal{S}}\rangle^{[0,+\infty)}[t+d-c]\subseteq \langle G_{\mathcal{S}}\rangle^{[0,+\infty)}[r+b-a+t+d-c+1+o^{-}(G_{\mathcal{S}})].\]
By (\ref{can-tri}), we have $Y\in \langle G_{\mathcal{S}}\rangle^{[0,+\infty)}[r+b-a+t+d-c+1+o^{-}(G_{\mathcal{S}})]$.
Then
\begin{align*}
\fd(\mathcal{S},G_{\mathcal{S}}) & \le r+b-a+t+d-c+1+o^{-}(G_{\mathcal{S}}) \\
& = \fd(\mathcal{R},G_{\mathcal{R}})+\fd(\mathcal{T},G_{\mathcal{T}})+b-a+d-c+1+o^{-}(G_{\mathcal{S}}).
\end{align*}
This shows Theorem \ref{mainthm-fd-1}(1).

(2) follows from Lemma \ref{lem-fullyfaithful}.
\hfill$\square$

\medskip
Applying Theorem \ref{mainthm-fd-1} to a left recollement of  triangulated categories in which the middle category has a classical generator, we obtain the following relationship among the finiteness of finitistic dimensions of those triangulated categories.

\begin{cor}
Let $\mathcal{R}$, $\mathcal{S}$, and $\mathcal{T}$ be triangulated categories with a left recollement
\begin{align*}
\xymatrixcolsep{4pc}\xymatrix{\mathcal{R} \ar[r]|{i_*} &\mathcal{S} \ar@<-2ex>[l]|{i^*}  \ar[r]|{j^*}  &\mathcal{T}. \ar@<-2ex>[l]|{j_!}
}
\end{align*}
Suppose that $\mathcal{S} = \langle G \rangle$ for some object $G\in\mathcal{S}$ with $\Hom_{\mathcal{S}}(G, G[k]) = 0$ for all $k \ll 0$. Then

{\rm (1)} If $\fd(\mathcal{R}) < +\infty$ and $\fd(\mathcal{T}) < +\infty$, then $\fd(\mathcal{S}) < +\infty$.

{\rm (2)} If $\fd(\mathcal{S}) < +\infty$, then $\fd(\mathcal{R}) < +\infty$.

\end{cor}
\begin{proof}
It is clear that $\mathcal{R}=\langle i^{*}(G) \rangle$. Then $(1)$ and $(2)$ follow from Theorem~\ref{mainthm-fd-1}.
\end{proof}

{\bf Proof of Corollary \ref{cor to regular}.}
By taking the opposite category along the recollement of  $\mathcal{S}$ by $\mathcal{R}$ and $\mathcal{T}$, we have a recollement of $\mathcal{S}^{\mathrm{op}}$ by $\mathcal{R}^{\mathrm{op}}$ and $\mathcal{T}^{\mathrm{op}}$:
\[
(*)\quad\quad
\xymatrixcolsep{4pc}\xymatrix{
\mathcal{R}^{\mathrm{op}} \ar[r]|{i_*=i_!} &
\mathcal{S}^{\mathrm{op}} \ar@<-2ex>[l]|{i^!} \ar@<2ex>[l]|{i^*} \ar[r]|{j^!=j^*} &
\mathcal{T}^{\mathrm{op}} \ar@<-2ex>[l]|{j_{*}} \ar@<2ex>[l]|{j_!}.
}
\]
Further, by Lemma~\ref{lem-bounde induced bounde}, both $\mathcal{S}$ and $\mathcal{S}^{\mathrm{op}}$ have bounded $t$-structures.

Recall that $\fd(\mathcal{R}^{\mathrm{op}}) < +\infty$ means that there exists an object $G_{\mathcal{R}}^{\mathrm{op}} \in \mathcal{R}^{\mathrm{op}}$ such that $\fd(\mathcal{R}^{\mathrm{op}}, G_{\mathcal{R}}^{\mathrm{op}}) < +\infty$. Similarly, $\fd(\mathcal{T}^{\mathrm{op}}, G_{\mathcal{T}}^{\mathrm{op}}) < +\infty$ for some $G_{\mathcal{T}}^{\mathrm{op}} \in \mathcal{T}^{\mathrm{op}}$. Set $G_{\mathcal{S}}:= i_{*}(G_{\mathcal{R}}) \oplus j_{*}(G_{\mathcal{T}})$. Since $\mathcal{S}^{\mathrm{op}}$ has a bounded $t$-structure, we have $\Hom_{\mathcal{S}^{\mathrm{op}}}(G_{\mathcal{S}}^{\mathrm{op}}, G_{\mathcal{S}}^{\mathrm{op}}[k]) = 0$ for $k \ll 0$. Now, we apply Theorem~\ref{mainthm-fd-1} to the above two lines (that form a left recollement) of the recollement $(\ast)$ and obtain $\fd(\mathcal{S}^{\mathrm{op}}, G_{\mathcal{S}}^{\mathrm{op}}) < +\infty$. Since $\mathcal{S}$ has a bounded $t$-structure, it follows from \cite[Theorem 1.5(1)]{bcrpz24} that $\mathcal{S}$ is equivalent to $\mathfrak{S}_{G_{\mathcal{S}}}(\mathcal{S})$ as triangulated categories, where $\mathfrak{S}_{G_{\mathcal{S}}}(\mathcal{S})$ denotes the completion of the triangulated category $\mathcal{S}$ (in the sense of Neeman) with respect to the good metric defined by $G_{\mathcal{S}}$. Thus $\mathcal{S}$ is regular. \hfill$\square$

\subsection{Finitistic dimensions in weakly approximable triangulated categories}\label{FDWAPP}

In this subsection, we discuss (big) finitistic dimensions of weakly approximable triangulated categories connected by recollements. In particular, we establish Theorems \ref{thm-fd-small} and \ref{thm-fd-big} and also give a proof of Corollary \ref{mainthm-cx17}.

Let $\mathcal{S}$ be a compactly generated triangulated category with a compact generator $G$, and let $(\mathcal{S}^{\le 0},\mathcal{S}^{\ge 1})\coloneqq(\mathcal{S}_G^{\le 0},\mathcal{S}_G^{\ge 1})$ be the $t$-structure generated by $G$ (see Lemma~\ref{lem-p-w-gen}). Moreover, in Definition \ref{Ffpd}, we have defined the projective dimension $\pd_G(X)$ for each object $X\in\mathcal{S}$, and also the big finitistic dimension ${\rm Fpd}(\mathcal{S}, G)$ and finitistic dimension ${\rm fpd}(\mathcal{S}, G)$ of $\mathcal{S}$ with respect to $G$.
Some basic properties of those homological dimensions are collected as follows.

\begin{remark}\label{some rems}
$(1)$ Note that $\pd_G(X[n])=\pd_G(X)+n$ for each $X\in\mathcal{S}$ and $n\in\mathbb{Z}$. For any triangle $X_1\to X_2\to X_3\to X_1[1]$ in $\mathcal{S}$, we have $\pd_G(X_2)\leq \max\{\pd_G(X_1), \pd_G(X_3)\}$.
Thus the full subcategory of $\mathcal{S}$ consisting of all objects $X$ with
$\pd_G(X)<+\infty$ is a triangulated subcategory of $\mathcal{S}$ closed under direct summands. Moreover, the finiteness of ${\rm Fpd}(\mathcal{S}, -)$ and ${\rm fpd}(\mathcal{S}, -)$ are independent of the choice of different compact generators of $\mathcal{S}$.

$(2)$ Suppose that $\Hom_\mathcal{S}(G, G[i])=0$ for $|i|\gg 0$. By the first paragraph of the proof of \cite[Proposition B.3]{bcrpz24}, we have $\Fpd(\mathcal{S},G)\geq\fpd(\mathcal{S}, G)\ge -o^-(G)$. This means that $\{\Fpd(\mathcal{S}, G),\fpd(\mathcal{S}, G)\}\subseteq\mathbb{Z}\cup\{+\infty\}$.

$(3)$ Let $S$ be a ring. The \emph{big finitistic dimension} of $S$, denoted by ${\rm Findim}(S)$, is defined as
$$
\Fd(S) \coloneqq \sup\{\pd_S(M) \mid M \in S\text{-Mod},\; \pd_S(M)<+\infty\};
$$
the \emph{finitistic dimension} of $S$, denoted by ${\rm findim}(S)$, is defined as
$$
\fd(S) \coloneqq \sup\{\pd_S(M) \mid M \in S\text{-Mod}\cap\mathscr{D}^{c}(S\Modcat)\}.
$$
Then $\Fd(S)={\rm Fpd}(\mathscr{D}(S\Modcat), S)$ and $\fd(S)={\rm fpd}(\mathscr{D}(S\Modcat), S).$

$(4)$ In \cite[Definition~7.1]{bssw25}, the big finitistic dimension for a Noetherian non-positive differential graded algebra $R$ was defined, and is equal to $\operatorname{Fpd}(\mathscr{D}(R),R)$ by \cite[Appendix~B]{bcrpz24}.
\end{remark}

\begin{lem}\label{range}
Let $\mathcal{S}$ be a compactly generated triangulated category with a compact generator $G$.
If $M\in\overline{\langle G\rangle}^{[-n,+\infty)}$ for some $n\in\mathbb{Z}$, then $\pd_{G}(M)\leq o^{+}(G)+n$. In particular, if $o^{+}(G)<+\infty$, then $\pd_G(M)<+\infty$ for any $M\in\mathcal{S}^{\rm {sb}}$.
\end{lem}

\begin{proof}
If $o^{+}(G)=+\infty$, then Lemma \ref{range} holds trivially. Assume that $o^{+}(G)<+\infty$. By Definition~\ref{index}, $\Hom_{\mathcal{S}}(G,G[i])=0$ for $i>o^{+}(G)$.
Then $G(-\infty,-o^{+}(G)]\subseteq G[1,+\infty)^{\perp}$. Since every object in $G[1,+\infty)$ is compact, the category $G[1,+\infty)^{\perp}$ is closed under coproducts in $\mathcal{S}$. Clearly, $G[1,+\infty)^{\perp}$ is closed under extensions in $\mathcal{S}$. Thus $\mathcal{S}_{G}^{\leq -o^{+}(G)}\subseteq G[1,+\infty)^{\perp}$ by Lemma~\ref{lem-p-w-gen}. This also means that $G[1,+\infty)\subseteq{}^{\perp}\mathcal{S}_{G}^{\leq -o^{+}(G)}$.
Since ${}^{\perp}\mathcal{S}_{G}^{\leq -o^{+}(G)}\subseteq \mathcal{S}$ is closed under coproducts and extensions, we have $\overline{\langle G\rangle}^{[1,+\infty)}\subseteq{}^{\perp}\mathcal{S}_{G}^{\leq -o^{+}(G)}$. It follows that
\[
\overline{\langle G\rangle}^{[-n,+\infty)}=\overline{\langle G\rangle}^{[1,+\infty)}[n+1]\subseteq({}^{\perp}\mathcal{S}_{G}^{\leq -o^{+}(G)})[n+1]={}^{\perp}\mathcal{S}_{G}^{\leq -o^{+}(G)-n-1}.
\] If  $M\in \overline{\langle G\rangle}^{[-n,+\infty)}$, then
$\Hom_\mathcal{S}(M, \mathcal{S}_{G}^{\leq 0}[o^+(G)+n+1])=0$, and thus $\pd_{G}(M)\leq o^+(G)+n$ by Definition \ref{Ffpd}.
\end{proof}

From now on, let $\mathcal{S}$ be a {\bf weakly approximable} triangulated category
with a compact generator $G$ and with the $t$-structure $(\mathcal{S}^{\le 0}, \mathcal{S}^{\ge 1}) $ generated by $G$  (see Lemma~\ref{lem-p-w-gen}).

The following result provides a converse of Lemma \ref{range} in the case of weakly approximable triangulated categories.
\begin{lem}\label{range-1}
Suppose that $G$ is bounded. If $M\in\mathcal{S}^b$ and $\pd_G(M)<+\infty$, then there exists an integer $n$ (depending on $M$) such that  $M\in \overline{\langle G\rangle}^{[-\pd_G(M)-a(G),\, n]}$.
If, in addition, $M\in\mathcal{S}^c$, then $M\in \langle G\rangle ^{[-\pd_G(M)-a(G),\, n]}$.
\end{lem}
\begin{proof}
Let $M\in\mathcal{S}^b$ and $m\coloneqq\pd_G(M)<+\infty$.
By Definition \ref{Ffpd}(1), $\Hom_\mathcal{S}(M, N[i])=0$ for $i>m$ and $N\in \mathcal{S}^{b}\cap\mathcal{S}^{\le 0}$. Since $M\in\mathcal{S}^b\subseteq\mathcal{S}^{-}$, we have $M[n]\in\mathcal{S}^b\cap\mathcal{S}^{\le 0}$ for some integer $n\geq 0$. If $m=-\infty$, then $M=0$ by the vanishing of $\Hom_\mathcal{S}(M, M[n][-n])$. In this case, Lemma \ref{range-1} holds trivially. Now, we assume $m>-\infty$, that is, $m$ is finite.  Since $\mathcal{S}$ is weakly approximable, it follows from Lemma \ref{lem-app-induction} that there is a triangle in $\mathcal{S}$:
$$E\lra M[n]\lraf{h} D\lra E[1]$$ with $E\in \overline{\langle G\rangle}^{[-m-n-a(G),a(G)]}$ and $D\in \mathcal{S}^{\le -m-n-1}$. Since $G$ is bounded, $E\in\mathcal{S}^{{\rm sb}}\subseteq\mathcal{S}^{b}$ by Proposition~\ref{lem-sp=big}. Combining this with $M\in\mathcal{S}^b$, we have $D[-m-n-1]\in\mathcal{S}^b\cap\mathcal{S}^{\le 0}$. Consequently,
$$\Hom_{\mathcal{S}}(M[n], D)\simeq\Hom_{\mathcal{S}}(M, D[-m-n-1][m+1])=0.$$
Thus $h=0$, and further, $M[n]$ is a direct summand of $E$. This implies
$M\in\overline{\langle G\rangle}^{[-m-a(G),a(G)+n]}.$
If $M\in\mathcal{S}^c$, then it follows from Lemma \ref{Closure} that
$M\in\mathcal{S}^c\cap\overline{\langle G\rangle}^{[-m-a(G),a(G)+n]} = \langle G\rangle^{[-m-a(G),a(G)+n]}.$
\end{proof}

In the following result, $(1)$ gives a new characterization of the category $\mathcal{S}^{{\rm sb}}$ (see Definition \ref{defn-sp}), while $(2)$ establishes the equivalences of the finiteness of two kinds of finitistic dimensions, which generalizes \cite[Proposition B.3(3)]{bcrpz24}.

\begin{prop}\label{Relation}
Suppose that $G\in\mathcal{S}$ is bounded. The following statements are true.

$(1)$ Let $M\in\mathcal{S}^b$. Then $\pd_G(M)<+\infty$ if and only if $M\in \mathcal{S}^{{\rm sb}}$. Moreover, $\pd_G(M)=-\infty$ if and only if $M=0$.

$(2)$ $\fd(\mathcal{S}^c, G)\leq \fpd(\mathcal{S},G)+a(G)$ and
$\fpd(\mathcal{S},G)\leq \fd(\mathcal{S}^c, G)+o^+(G)$. In particular, $\fd(\mathcal{S}^c, G)$ is finite if and only if so is $\fpd(\mathcal{S},G)$.
\end{prop}

\begin{proof}
$(1)$ follows from Lemmas \ref{range} and \ref{range-1}. This can be also concluded from Proposition \ref{lem-sp=big}.

$(2)$ Since $\mathcal{S}$ is weakly approximable, $o^+(G)<+\infty$ by Lemma \ref{PR}(1).
By Lemma \ref{range}, $\pd_G(M)<+\infty$ for any $M\in\mathcal{S}^c$. As $G$ is bounded, the inclusion $\mathcal{S}^c\subseteq\mathcal{S}^b$ holds. Moreover,
$G(-\infty, -1]^{\perp}\cap\mathcal{S}=\mathcal{S}^{\ge 0}$ by Lemma \ref{lem-p-w-gen}. Consequently, by Definitions \ref{Intro-defn-fd} and \ref{Ffpd}, there are equalities:
$$
\fd(\mathcal{S}^c, G)=\inf\big\{n\in\mathbb{N}\mid\mathcal{S}^c\cap\mathcal{S}^{\ge 0}\subseteq\langle G\rangle^{[-n,+\infty)}\big\}\;\;\mbox{and}\;\;
{\rm fpd}(\mathcal{S}, G)=\sup\{\pd_G(M)\mid M\in \mathcal{S}^c\cap\mathcal{S}^{\ge 0}\}.
$$
Since $\langle G\rangle^{[-n,+\infty)}\subseteq \overline{\langle G\rangle}^{[-n,+\infty)}$, we see from
Lemma \ref{range} that $\pd_{G}(M)\leq n+o^{+}(G)$ for any $M\in \langle G\rangle^{[-n,+\infty)}$.
This implies $\fpd(\mathcal{S},G)\leq \fd(\mathcal{S}^c, G)+o^+(G)$. By Lemma \ref{range-1},
if $M\in \mathcal{S}^c\cap\mathcal{S}^{\ge 0}$ (and automatically $\pd_G(M)<+\infty)$, then $M\in \langle G\rangle ^{[-\pd_G(M)-a(G), +\infty)}$. This implies $\fd(\mathcal{S}^c, G)\leq \fpd(\mathcal{S},G)+a(G)$.
\end{proof}

Motivated by Definition \ref{Intro-defn-fd} and Proposition \ref{Relation}(2), we introduce the following definition which is of independent interest.

\begin{defn}
Let $\mathcal{U}$ be a compactly generated triangulated category with a single compact generator. The \emph{big finitistic dimension} of $\mathcal{U}$ at an object $U$ is defined as
$$\Fd(\mathcal{U}, U)\coloneqq \inf\big\{n\in\mathbb{N}\mid U(-\infty, -1]^{\perp}
\cap\mathcal{U}^{\rm sb}\subseteq\bigcup_{m\in\mathbb{N}}\overline{\langle U\rangle}^{[0,\, m]}[n]\,\big\}.$$
\end{defn}

By Lemmas \ref{range} and \ref{range-1}, we can show the following result that is related to two kinds of big finitistic dimensions.

\begin{prop}
Suppose that $G\in\mathcal{S}$ is bounded. Then $\Fd(\mathcal{S}, G)\leq \Fpd(\mathcal{S},G)+a(G)$ and
$\Fpd(\mathcal{S},G)\leq \Fd(\mathcal{S}, G)+o^+(G)$. In particular, $\Fd(\mathcal{S}, G)$ is finite if and only if so is $\Fpd(\mathcal{S},G)$.
\end{prop}

In the rest of the section, let $\mathcal{R}$, $\mathcal{S}$, and $\mathcal{T}$ be weakly approximable triangulated categories with compact generators $G_{\mathcal{R}}$, $G_{\mathcal{S}}$, and $G_{\mathcal{T}}$, respectively. Unless otherwise stated, we always equip them with the $t$-structures generated by their  compact generators (see Lemma~\ref{lem-p-w-gen}), that is,
\[
(\mathcal{R}^{\le 0}, \mathcal{R}^{\ge 1}) \coloneqq (\mathcal{R}_{G_{\mathcal{R}}}^{\le 0}, \mathcal{R}_{G_{\mathcal{R}}}^{\ge 1}), \quad
(\mathcal{S}^{\le 0}, \mathcal{S}^{\ge 1}) \coloneqq (\mathcal{S}_{G_{\mathcal{S}}}^{\le 0}, \mathcal{S}_{G_{\mathcal{S}}}^{\ge 1}) \quad\mbox{and}\quad
(\mathcal{T}^{\le 0}, \mathcal{T}^{\ge 1}) \coloneqq (\mathcal{T}_{G_{\mathcal{T}}}^{\le 0}, \mathcal{T}_{G_{\mathcal{T}}}^{\ge 1}).
\]
For simplicity of notation, we write $\pd_\mathcal{S}(X)$ for $\pd_{G_\mathcal{S}}(X)$ for any $X\in\mathcal{S}$.

Now, we consider the following recollement of $\mathcal{S}$ by $\mathcal{R}$ and $\mathcal{T}$:
\begin{align*}
\xymatrixcolsep{4pc}\xymatrix{\mathcal{R} \ar[r]|{i_*=i_!} &\mathcal{S} \ar@<-2ex>[l]|{i^*} \ar@<2ex>[l]|{i^!} \ar[r]|{j^!=j^*}  &\mathcal{T}. \ar@<-2ex>[l]|{j_!} \ar@<2ex>[l]|{j_{*}}
}\qquad (\star)
\end{align*}
By Lemma~\ref{lem-key}(1), $i^*$ and $j_!$ can be restricted to $\mathcal{S}^c\ra \mathcal{R}^c$ and $\mathcal{T}^c\ra \mathcal{S}^c$, respectively. In particular, $i^{*}(G_{\mathcal{S}})\in \mathcal{R}^c$ and $ j_{!}(G_{\mathcal{T}})\in \mathcal{S}^c$. Note that $\mathcal{R}^c=\langle G_{\mathcal{R}}\rangle$  and $\mathcal{S}^c=\langle G_{\mathcal{S}}\rangle$. Thus there are two integers $c\leq d$ such that $$j_!(G_{\mathcal{T}}) \in \langle G_{\mathcal{S}} \rangle^{[c, d]}\subseteq \langle G_{\mathcal{S}} \rangle^{(-\infty, d]}.$$ In later discussions, \textbf{we always fix $\bm{c}$ and $\bm{d}$}. Then $0\leq w_{G_{\mathcal{S}}}(j_{!}(G_{\mathcal{T}}))\leq d-c<+\infty$ (see Definition \ref{index}). Similarly, $w_{G_{\mathcal{R}}}(i^{*}(G_{\mathcal{S}}))$ is a nonnegative integer.

\begin{lem}\label{lem-fd-contain}
Suppose that $j_!$ can be restricted to a functor $\mathcal{T}^b\ra \mathcal{S}^b$.
Then $\mathcal{T}^{\le 0}\cap \mathcal{T}^b \subseteq  j^*(\mathcal{S}^{\le d}\cap \mathcal{S}^b)$.
\end{lem}
\begin{proof}
We first show $j^*(\mathcal{S}^{\ge d}) \subseteq \mathcal{T}^{\ge 0}$.

Since $(j_!, j^*)$ is an adjoint pair, $\Hom_{\mathcal{T}}(G_{\mathcal{T}}[p], j^*(Y)) \simeq \Hom_{\mathcal{S}}(j_!(G_{\mathcal{T}})[p], Y)$ for $p \in \mathbb{Z}$ and $Y \in \mathcal{S}$.
Recall that $j_!(G_{\mathcal{T}}) \in \langle G_{\mathcal{S}} \rangle^{[c, d]}$ and $G_{\mathcal{S}}\in \mathcal{S}^{\le 0}$. This forces $j_!(G_{\mathcal{T}})[p] \in \langle G_{\mathcal{S}} \rangle^{[c-p, d-p]}\subseteq \mathcal{S}^{\le d-p}$.
Let $Y \in \mathcal{S}^{\ge d}$ and $p\geq 1$. Since $\Hom_\mathcal{S}( \mathcal{S}^{\le d-p}, \mathcal{S}^{\ge d})=0$ by the $t$-structure $(\mathcal{S}^{\le 0}, \mathcal{S}^{\ge 1})$ on $\mathcal{S}$, we have $\Hom_{\mathcal{S}}(j_!(G_{\mathcal{T}})[p], Y) = 0$. Thus $\Hom_{\mathcal{T}}(G_{\mathcal{T}}[p], j^*(Y)) = 0$. Then the inclusion $j^*(\mathcal{S}^{\ge d}) \subseteq \mathcal{T}^{\ge 0}$ follows from the characterization of $\mathcal{T}^{\ge 0}$ in Lemma \ref{lem-p-w-gen}.

For checking $\mathcal{T}^{\le 0}\cap \mathcal{T}^b \subseteq  j^*(\mathcal{S}^{\le d}\cap \mathcal{S}^b)$, let $Z\in \mathcal{T}^{\le 0}\cap \mathcal{T}^b$. Since $j_!$ restricts to a functor $\mathcal{T}^b\ra \mathcal{S}^b$, we have $j_!(Z)\in \mathcal{S}^b$.
As the $t$-structure $(\mathcal{S}^{\le 0}, \mathcal{S}^{\ge 1})$ on $\mathcal{S}$ can be restricted to a $t$-structure on $\mathcal{S}^b$, there is a canonical triangle in $\mathcal{S}$:
\begin{align*}
(j_!(Z))^{\le d}\lra j_!(Z)\lraf{h} (j_!(Z))^{\ge d+1}\lra (j_!(Z))^{\le d}[1]
\end{align*}
where $(j_!(Z))^{\le d}\in \mathcal{S}^{\le d}\cap \mathcal{S}^b$ and $(j_!(Z))^{\ge d+1}\in \mathcal{S}^{\ge d+1}\cap \mathcal{S}^b$. Since $j^*(\mathcal{S}^{\ge d})\subseteq \mathcal{T}^{\ge 0}$, we have $j^*(\mathcal{S}^{\ge d+1})\subseteq \mathcal{T}^{\ge 1}$. In particular, $j^*((j_!(Z))^{\ge d+1})\in \mathcal{T}^{\ge 1}$.
Since $\Hom_\mathcal{T}( \mathcal{T}^{\le 0}, \mathcal{T}^{\ge 1})=0$ and $Z\in \mathcal{T}^{\le 0}$, it follows that $\Hom_{\mathcal{S}}(j_!(Z), (j_!(Z))^{\ge d+1})
\simeq \Hom_{\mathcal{T}}(Z, j^*((j_!(Z))^{\ge d+1}))=0$. This implies $h=0$, and therefore $j_!(Z)$ is a direct summand of $j_!(Z))^{\le d}$. Clearly, $\mathcal{S}^{\le d}\cap \mathcal{S}^b$ is closed under direct summands in $\mathcal{S}$. Thus $j_!(Z)\in \mathcal{S}^{\le d}\cap \mathcal{S}^b$.
Since $j_!$ is fully faithful, $Z\simeq j^*j_!(Z)$. This forces $Z\in j^*(\mathcal{S}^{\le d}\cap \mathcal{S}^b)$.
\end{proof}

We now give the main result of this subsection.

\begin{thm}\label{thm-fd-small}
{\rm (1)}
Suppose that the compact generator $G_{\mathcal{S}}$ of $\mathcal{S}$ is bounded and $i_*(G_{\mathcal{R}})\in\mathcal{S}^{c}$.
Then:

{\rm (i)}
$\fpd(\mathcal{S},G_{\mathcal{S}})\le \fpd(\mathcal{R},G_{\mathcal{R}})+w_{G_{\mathcal{S}}}(i_{*}(G_{\mathcal{R}}))+a(G_{\mathcal{R}})+\fpd(\mathcal{T},G_{\mathcal{T}})
+w_{G_{\mathcal{S}}}(j_{!}(G_{\mathcal{T}}))+a(G_{\mathcal{T}})+o^{-}(G_{\mathcal{S}})+o^{+}(G_{\mathcal{S}})+1.$

{\rm (ii)} $\fpd(\mathcal{R},G_{\mathcal{R}})\le \fpd(\mathcal{S},G_{\mathcal{S}})+w_{G_{\mathcal{R}}}(i^{*}(G_{\mathcal{S}}))
+a(G_{\mathcal{S}})+o^{+}(G_{\mathcal{R}})$.

{\rm (2)}  Suppose that $j_!(\mathcal{T}^{\ge 0})\subseteq \mathcal{S}^{\ge -e}$ for some integer $e$.
Then $\fpd(\mathcal{T},G_{\mathcal{T}})\le\fpd(\mathcal{S},G_{\mathcal{S}})+d+e$.
\end{thm}
\begin{proof}
(1) We apply Theorem \ref{mainthm-fd-1} to show (1).

Recall that $j_!$ and $i^*$ can be restricted to  $\mathcal{T}^c\to\mathcal{S}^c$ and $\mathcal{S}^c\to\mathcal{R}^c$, respectively. Since $i_*(G_{\mathcal{R}}) \in\mathcal{S}^{c}$, the functor $i_*$ restricts to $\mathcal{R}^c \to \mathcal{S}^c$. By Lemma~\ref{lem-key}(2), the functor $j^*$ also restricts to $\mathcal{S}^c \to \mathcal{T}^c$. It follows from the recollement $(\star)$ that
there is a left recollement of triangulated categories consisting of compact objects:
\begin{align*}
\xymatrixcolsep{4pc}\xymatrix{\mathcal{R}^c \ar[r]|{i_*} &\mathcal{S}^c \ar@<-2ex>[l]|{i^*}  \ar[r]|{j^*}  &\mathcal{T}^c. \ar@<-2ex>[l]|{j_!}
}
\end{align*}
Note that $i_*(G_{\mathcal{R}}), j_!(G_{\mathcal{T}})\in \mathcal{S}^c=\langle G_{\mathcal{S}} \rangle$. Since $G_{\mathcal{S}}$ is bounded, $\Hom_{\mathcal{S}}(G_{\mathcal{S}}[n],G_{\mathcal{S}})=0$ for $n\gg 0$. Now, we apply Theorem \ref{mainthm-fd-1}(1) to the above left recollement and obtain the following inequality:
$$(\sharp)\quad \fd(\mathcal{S}^c,G_{\mathcal{S}})\le \fd(\mathcal{R}^c,G_{\mathcal{R}})+\fd(\mathcal{T}^c,G_{\mathcal{T}})
+w_{G_{\mathcal{S}}}(i_{*}(G_{\mathcal{R}}))+w_{G_{\mathcal{S}}}(j_{!}(G_{\mathcal{T}}))
+o^{-}(G_{\mathcal{S}})+1,$$

By Lemma \ref{lem-vanish}, $G_{\mathcal{R}}$ and $G_{\mathcal{T}}$ are bounded. So we can apply Proposition \ref{Relation}(2) to the categories $\mathcal{R}$ and $\mathcal{T}$. Our aim is to replace $\fd(\mathcal{S}^c,G_{\mathcal{S}})$, $\fd(\mathcal{R}^c,G_{\mathcal{R}})$ and $\fd(\mathcal{T}^c,G_{\mathcal{T}})$ with $\fpd(\mathcal{S},G_{\mathcal{S}})$, $\fpd(\mathcal{R},G_{\mathcal{R}})$ and $\fpd(\mathcal{T},G_{\mathcal{T}})$, respectively.

Since $\mathcal{R}$, $\mathcal{S}$ and $\mathcal{T}$ are weakly approximable with bounded compact generators, we see from the inequality $(\sharp)$ and Proposition \ref{Relation}(2) that
$
\fpd(\mathcal{S},G_\mathcal{S})\leq \fd(\mathcal{S}^c, G_\mathcal{S})+o^+(G_\mathcal{S})
\leq\fd(\mathcal{R}^c,G_{\mathcal{R}})
+\fd(\mathcal{T}^c,G_{\mathcal{T}})
+w_{G_{\mathcal{S}}}(i_{*}(G_{\mathcal{R}}))+w_{G_{\mathcal{S}}}(j_{!}(G_{\mathcal{T}}))
+o^{-}(G_{\mathcal{S}})+1+o^+(G_\mathcal{S})
\leq \fpd(\mathcal{R},G_\mathcal{R})+a(G_\mathcal{R})+ \fpd(\mathcal{T},G_\mathcal{T})+a(G_\mathcal{T})+w_{G_{\mathcal{S}}}(i_{*}(G_{\mathcal{R}}))
+w_{G_{\mathcal{S}}}(j_{!}(G_{\mathcal{T}}))
+o^{-}(G_{\mathcal{S}})+1+o^+(G_\mathcal{S}).$ This shows (i).

Next, we show (ii).

Indeed, since $i^*(G_{\mathcal{S}})\in \mathcal{R}^c=\langle G_{\mathcal{R}} \rangle$, we can apply Theorem \ref{mainthm-fd-1}(2) to the above left recollement and obtain
$\fd(\mathcal{R}^c,G_{\mathcal{R}})\le \fd(\mathcal{S}^c,G_{\mathcal{S}})+w_{G_{\mathcal{R}}}(i^*(G_{\mathcal{S}})).$
It follows from Proposition \ref{Relation}(2) that
$
\fpd(\mathcal{R},G_\mathcal{R})\leq \fd(\mathcal{R}^c, G_\mathcal{R})+o^+(G_\mathcal{R})\leq \fd(\mathcal{S}^c,G_{\mathcal{S}})+w_{G_{\mathcal{R}}}(i^*(G_{\mathcal{S}}))+o^+(G_\mathcal{R})\leq
\fpd(\mathcal{S},G_\mathcal{S})+a(G_\mathcal{S})+w_{G_{\mathcal{R}}}(i^*(G_{\mathcal{S}}))+o^+(G_\mathcal{R}). $  This shows (ii).

(2) If $\fpd(\mathcal{S},G_{\mathcal{S}})=+\infty$, then (2) holds automatically.
So, we assume $\fpd(\mathcal{S},G_{\mathcal{S}})<+\infty$.

By Lemma \ref{range}, for any weakly approximable triangulated category $\mathcal{U}$ with a compact generator $U$, we see that $\pd_U(X)<+\infty$ for any $X\in \mathcal{U}^c$. This implies
$${\rm fpd}(\mathcal{T}, G_\mathcal{T})=\sup\{\pd_\mathcal{T}(Z)\mid Z\in \mathcal{T}^c \cap \mathcal{T}^b \cap \mathcal{T}^{\ge 0}\}.$$
Let $Z \in \mathcal{T}^c \cap \mathcal{T}^b \cap \mathcal{T}^{\ge 0}=\mathcal{T}^c \cap \mathcal{T}^{\ge 0}$. Since $j_!$ preserves compact objects and $\mathcal{S}$ is weakly approximable, we have $j_!(Z)\in\mathcal{S}^c$ and further $\pd_{\mathcal{S}}(j_!(Z))<+\infty$.

Now, assume that $j_!(\mathcal{T}^{\ge 0})\subseteq \mathcal{S}^{\ge -e}$ for some integer $e$.
Then $j_!(\mathcal{T}^+)\subseteq \mathcal{S}^+$, and $j_!(Z)\in \mathcal{S}^{\ge -e}$ due to $Z\in\mathcal{T}^{\geq 0}$. Note that $j_!(\mathcal{T}^-)\subseteq \mathcal{S}^-$ by Lemma \ref{lem-gluing-$t$-struc}. Thus $j_!(\mathcal{T}^b)=j_!(\mathcal{T}^-\cap\mathcal{T}^+)\subseteq \mathcal{S}^-\cap\mathcal{S}^+=\mathcal{S}^b$. Since $j_!$ preserves compact objects and $Z\in\mathcal{T}^c\cap\mathcal{T}^b$, we have $j_!(Z)\in\mathcal{S}^c\cap\mathcal{S}^b$. By Remark \ref{some rems}(1), it follows that
$\pd_{\mathcal{S}}(j_!(Z))=\pd_{\mathcal{S}}(j_!(Z)[-e])+e\leq\fpd(\mathcal{S},G_{\mathcal{S}})+e$. This means that, for any $Y'\in \mathcal{S}^{\le 0}\cap \mathcal{S}^b$,
\begin{align}\label{33}
\Hom_{\mathcal{S}}(j_!(Z),Y'[p])=0 \;\text{for}\; p>e+\fpd(\mathcal{S},G_{\mathcal{S}}).
\end{align}
Moreover, by Lemma \ref{lem-fd-contain}, $\mathcal{T}^{\le 0}\cap \mathcal{T}^b \subseteq  j^*(\mathcal{S}^{\le d}\cap \mathcal{S}^b)$.
Now, we take an object $Z_1\in \mathcal{T}^{\le 0}\cap \mathcal{T}^b$. Then there exists an object $Y_1\in \mathcal{S}^{\le d}\cap \mathcal{S}^b$ such that $Z_1\simeq j^*(Y_1)$.
Clearly, $Y_1[d]\in \mathcal{S}^{\le 0}\cap \mathcal{S}^b$.
By $(\ref{33})$ and the adjoint pair $(j_!, j^*)$, there is a series of isomorphisms:
$$\Hom_{\mathcal{T}}(Z,Z_1[p])\simeq
\Hom_{\mathcal{T}}(Z,j^*(Y_1)[p])\simeq \Hom_{\mathcal{S}}(j_!(Z),Y_1[p])
\simeq \Hom_{\mathcal{S}}(j_!(Z),Y_1[d][p-d])=0$$
for $p>d+e+\fpd(\mathcal{S},G_{\mathcal{S}})$.
This implies $\pd_{\mathcal{T}}(Z)\le d+e+\fpd(\mathcal{S},G_{\mathcal{S}})$. Thus $\fpd(\mathcal{T},G_{\mathcal{T}})\le d+e+\fpd(\mathcal{S},G_{\mathcal{S}})$.
\end{proof}

Related to big finitistic dimensions of weakly approximable triangulated categories, we have the following theorem in which the condition (1) is  weaker than the one in Theorem \ref{thm-fd-small}(1).

\begin{thm}\label{thm-fd-big}
{\rm (1)} Suppose that the compact generator $G_{\mathcal{S}}$ of $\mathcal{S}$ is bounded and $i_*(G_{\mathcal{R}})\in\mathcal{S}^{{\rm sb}}$. Then

{\rm (i)}
$\Fpd(\mathcal{S},G_{\mathcal{S}})\le \Fpd(\mathcal{R},G_{\mathcal{R}})+\overline{w_{G_{\mathcal{S}}}}(i_{*}(G_{\mathcal{R}}))+a(G_{\mathcal{R}})+\Fpd(\mathcal{T},G_{\mathcal{T}})+w_{G_{\mathcal{S}}}(j_{!}(G_{\mathcal{T}}))+a(G_{\mathcal{T}})+o^{-}(G_{\mathcal{S}})+o^{+}(G_{\mathcal{S}})+1.$

{\rm (ii)} $\Fpd(\mathcal{R},G_{\mathcal{R}})\le \Fpd(\mathcal{S},G_{\mathcal{S}})+w_{G_{\mathcal{R}}}(i^{*}(G_{\mathcal{S}}))+a(G_{\mathcal{S}})+o^{+}(G_{\mathcal{R}})$.

{\rm (2)} Suppose that $j_!(\mathcal{T}^{\ge 0})\subseteq \mathcal{S}^{\ge -e}$ for some integer $e$. Then $\Fpd(\mathcal{T},G_{\mathcal{T}})\le\Fpd(\mathcal{S},G_{\mathcal{S}})+d+e$.
\end{thm}

\begin{proof}
(1) By Lemma \ref{lem-vanish}, $G_{\mathcal{R}}$ and $G_{\mathcal{T}}$ are bounded. It follows from Proposition \ref{lem-sp=big} that $\mathcal{R}^{{\rm sb}}\subseteq\mathcal{R}^{b}$, $\mathcal{S}^{{\rm sb}}\subseteq\mathcal{S}^{b}$ and $\mathcal{T}^{{\rm sb}}\subseteq\mathcal{T}^{b}$.
Further, since $i_*(G_{\mathcal{R}}) \in\mathcal{S}^{{\rm sb}}$, by Theorem \ref{prop-rec-res1}, the functor $i_*$ restricts to $\mathcal{R}^{{\rm sb}} \to \mathcal{S}^{{\rm sb}}$. So, we can make the following assumptions:
$$i_{*}(G_{\mathcal{R}})=\overline{\langle G_{\mathcal{S}}\rangle}^{[\alpha,\beta]},\; \overline{w_{G_{\mathcal{S}}}}(i_{*}(G_{\mathcal{R}}))=\beta-\alpha; \; \;i^{*}(G_{\mathcal{S}})=\langle G_{\mathcal{R}}\rangle^{[a,b]}, \;w_{G_{\mathcal{R}}}(i^{*}(G_{\mathcal{S}}))=b-a;$$
$$j_{!}(G_{\mathcal{T}})=\langle G_{\mathcal{S}}\rangle^{[c,d]},\; w_{G_{\mathcal{S}}}(j_{!}(G_{\mathcal{T}}))=d-c.$$  Note that $a(G_{\mathcal{R}})$, $a(G_{\mathcal{S}})$ and $a(G_{\mathcal{T}})$ are finite since $\mathcal{R}$, $\mathcal{S}$ and $\mathcal{T}$ are weakly approximable. Moreover, $o^{+}(G_{\mathcal{S}})$ and $o^{+}(G_{\mathcal{R}})$ are finite by Lemma \ref{PR}(1), while $o^{-}(G_{\mathcal{S}})$ is finite by the boundedness of $G_{\mathcal{S}}$.

(i) When either $\Fpd(\mathcal{R},G_{\mathcal{R}})=+\infty$ or $\Fpd(\mathcal{T},G_{\mathcal{T}})=+\infty$, (i) holds trivially. So, we assume
$$\sigma:=\Fpd(\mathcal{R},G_{\mathcal{R}})<+\infty\quad\mbox{and}\quad
\delta:=\Fpd(\mathcal{T},G_{\mathcal{T}})<+\infty.$$
Let $Y\in\mathcal{S}^{b}\cap\mathcal{S}^{\geq 0}$ with $\pd_{\mathcal{S}}(Y)<+\infty$.
Then there is a canonical triangle in $\mathcal{S}$
\begin{align}\label{f-canonical-triangle}
j_!j^*(Y)\lra Y\lra i_*i^*(Y)\lra j_!j^*(Y)[1].
\end{align}
Our strategy for showing (i) is to first determine the range of $j_!j^*(Y)$ and $i_*i^*(Y)$ by $G_\mathcal{S}$, and then to apply Lemma \ref{range} to give a upper bound for $\pd_{\mathcal{S}}(Y)$.
The proof is divided into four steps.

{\bf Claim 1}: $j_!j^*(Y)\in \overline{\langle G_{\mathcal{S}}\rangle}^{[-u,+\infty)}$, where $u\coloneqq \delta+d-c+a(G_{\mathcal{T}})$.

Since $i_*(G_{\mathcal{R}}) \in\mathcal{S}^{{\rm sb}}$, it follows from Theorem \ref{prop-rec-res1} that there is a left recollement:
\begin{align*}(\dag)\quad
\xymatrixcolsep{4pc}\xymatrix{
\mathcal{R}^{{\rm sb}} \ar[r]|{i_*=i_!}
&\mathcal{S}^{{\rm sb}} \ar@<-2ex>[l]|{i^*}   \ar[r]|{j^!=j^*}
&\mathcal{T}^{{\rm sb}}. \ar@<-2ex>[l]|{j_!}
}
\end{align*}
As $Y\in\mathcal{S}^{b}$ and $\pd_{\mathcal{S}}(Y)<+\infty$, we have $Y\in\mathcal{S}^{{\rm sb}}$ by Lemma~\ref{range-1}. Then $j^{*}(Y)\in\mathcal{T}^{{\rm sb}}\subseteq\mathcal{T}^{b}$, and therefore $\pd_{\mathcal{T}}(j^{*}(Y))<+\infty$ by Lemma \ref{range}. Now we show $j^*(Y)\in\mathcal{T}^b\cap\mathcal{T}^{\geq -d}$. Since $j_!(G_{\mathcal{T}})\in \langle G_{\mathcal{S}}\rangle^{[c,d]}$ and $G_{\mathcal{S}}\in\mathcal{S}^{\le 0}$, we have $j_!(G_{\mathcal{T}})[p]\in \mathcal{S}^{\le -1}$ for $p> d$.
It follows from $Y\in \mathcal{S}^{\ge 0}$ that $\Hom_{\mathcal{S}}(j_!(G_{\mathcal{T}})[p],Y)=0$ for $p>d$. Combining this with the adjoint pair $(j_!, j^*)$ leads to \[\Hom_{\mathcal{T}}(G_{\mathcal{T}}[p],j^*(Y))\simeq\Hom_{\mathcal{S}}(j_{!}(G_{\mathcal{T}})[p],Y)=0\text{ for }p>d.\]
This implies $j^*(Y)\in\mathcal{T}^{b}\cap\mathcal{T}^{\ge -d}$. By Remark \ref{some rems}(1), we obtain $\pd_{\mathcal{T}}(j^*(Y))=\pd_{\mathcal{T}}(j^*(Y)[-d])+d\leq\delta+d.$
Further, $j^*(Y)\in\overline{\langle G_{\mathcal{T}}\rangle}^{[-\delta-d-a(G_{\mathcal{T}}),+\infty)}$  by Lemma \ref{range-1}. Let $u_{1}:=\delta+d+a(G_{\mathcal{T}})$ and \[\mathscr{X}:=\{X\in\mathcal{T}\mid j_{!}(X)\in\overline{\langle G_{\mathcal{S}}\rangle}^{[c-u_{1},+\infty)}\}.\]
It follows from $j_{!}(G_{\mathcal{T}})\in\langle G_{\mathcal{S}}\rangle^{[c,d]}$ that $G_{\mathcal{T}}[-u_{1},+\infty)\subseteq\mathscr{X}$. Since $j_{!}$ preserves coproducts, $\mathscr{X}$ is closed under coproducts, direct summands and extensions in $\mathcal{T}$. This implies $\overline{\langle G_{\mathcal{T}}\rangle}^{[-u_{1},+\infty)}\subseteq\mathscr{X}$ and therefore  $j_{!}j^{*}(Y)\in\overline{\langle G_{\mathcal{S}}\rangle}^{[c-u_{1},+\infty)}=\overline{\langle G_{\mathcal{S}}\rangle}^{[-u,+\infty)}$.

{\bf Claim 2:} $i_*i^*(Y)\in \overline{\langle G_{\mathcal{S}}\rangle}^{[-v,+\infty)}$, where $v\coloneqq u+\sigma+\beta-\alpha+a(G_{\mathcal{R}})+o^{-}(G_{\mathcal{S}})+1$

Since $\Hom_{\mathcal{S}}(G_{\mathcal{S}},G_{\mathcal{S}}[i])=0$ for $i<- o^{-}(G_{\mathcal{S}})$, we have $G_{\mathcal{S}}\in \mathcal{S}^{\ge - o^{-}(G_{\mathcal{S}})}$ by Lemma \ref{lem-p-w-gen}. It follows from
$j_!j^*(Y)\subseteq\langle G_{\mathcal{S}}\rangle^{[-u,+\infty)}$ that $j_!j^*(Y)\in \mathcal{S}^{\ge -u-o^{-}(G_{\mathcal{S}}) }$. By $Y \in \mathcal{S}^{\ge 0}$, we see from the triangle~\eqref{f-canonical-triangle} that  $i_* i^*(Y) \in \mathcal{S}^{\ge -u - o^{-}(G_{\mathcal{S}})-1}$.

Let $v_1\coloneqq u+o^{-}(G_{\mathcal{S}})+\beta+1$. Since $i_*(G_{\mathcal{R}})\in \langle G_{\mathcal{S}}\rangle^{[\alpha,\beta]}$ and $i_* i^*(Y) \in \mathcal{S}^{\ge -u - o^{-}(G_{\mathcal{S}}) - 1}$, it follows from the characterization of $\mathcal{S}^{\geq 1}$ in Lemma \ref{lem-p-w-gen} that $\Hom_{\mathcal{S}}(i_*(G_{\mathcal{R}})[p],i_*i^*(Y))=0$ for $p>v_1$. Using the full faithfulness of $i_*$, we obtain $\Hom_{\mathcal{R}}(G_{\mathcal{R}}[p],i^*(Y))=0$ for $p>v_1$. Still by Lemma \ref{lem-p-w-gen}, $i^*(Y) \in \mathcal{R}^{\ge -v_1}$. By the recollement $(\dag)$, $i^*(Y)[-v_1]\in \mathcal{R}^{{\rm sb}}\cap\mathcal{R}^{\ge 0}\subseteq\mathcal{R}^b\cap\mathcal{R}^{\ge 0}.$ Hence $\pd_{\mathcal{R}}(i^*(Y)[-v_{1}])<+\infty$ by Lemma \ref{range}. By Remark~\ref{some rems}(1), we obtain $\pd_{\mathcal{R}}(i^*(Y))=\pd_{\mathcal{R}}(i^*(Y)[-v_{1}])+v_{1}\leq\sigma+v_{1}$. By Lemma \ref{range-1}, $i^*(Y)\in\overline{\langle G_{\mathcal{R}}\rangle}^{[-\sigma-v_1-a(G_{\mathcal{R}}),+\infty)}$. Note that $i_{*}$ preserves coproducts and $i_{*}(G_{\mathcal{R}})=\overline{\langle G_{\mathcal{S}}\rangle}^{[\alpha,\beta]}$. Similarly, we can show  $i_*i^*(Y)\in\overline{\langle G_{\mathcal{S}}\rangle}^{[\alpha-\sigma-v_{1}-a(G_{\mathcal{R}}),+\infty)}=\overline{\langle G_{\mathcal{S}}\rangle}^{[-v,+\infty)}$.

{\bf Claim 3:} $u\le v$.

It suffices to show that $\sigma\ge -(\beta-\alpha+o^{-}(G_{\mathcal{S}}))$. By Remark~\ref{some rems}(2), we only need to show that $\Hom_{\mathcal{R}}(G_{\mathcal{R}}[p],G_{\mathcal{R}})=0$ for $p>\beta-\alpha+o^{-}(G_{\mathcal{S}})$.

Indeed, since $i_*$ is fully faithful, we have $\Hom_{\mathcal{R}}(G_{\mathcal{R}}[p],G_{\mathcal{R}})\simeq
\Hom_{\mathcal{S}}(i_*(G_{\mathcal{R}})[p],i_*(G_{\mathcal{R}})).$
Recall that $i_*(G_{\mathcal{R}})\in \langle G_{\mathcal{S}}\rangle^{[\alpha,\beta]}$ and $\Hom_{\mathcal{S}}(G_{\mathcal{S}}[i],G_{\mathcal{S}})=0$ for $i> o^{-}(G_{\mathcal{S}})$. Consequently,  for any $Y_3,Y_4\in \langle G_{\mathcal{S}}\rangle^{[\alpha,\beta]}$, we have $\Hom_{\mathcal{S}}(Y_3[p],Y_4)=0$ for $p>\beta-\alpha+o^{-}(G_{\mathcal{S}})$.
In particular, taking $Y_3 = Y_4 = i_*(G_{\mathcal{R}})$ yields $\Hom_{\mathcal{S}}(i_*(G_{\mathcal{R}})[p], i_*(G_{\mathcal{R}}))=0$ for $p>\beta-\alpha+o^{-}(G_{\mathcal{S}})$. Thus $\Hom_{\mathcal{R}}(G_{\mathcal{R}}[p], G_{\mathcal{R}})=0$ for $p>\beta-\alpha+o^{-}(G_{\mathcal{S}})$.

{\bf Claim 4:} $\Fpd(\mathcal{S},G_{\mathcal{S}})\le o^{+}(G_{\mathcal{S}})+v$. Thus $(1)$ holds.

In fact, combining the triangle (\ref{f-canonical-triangle}) with the claims $1$-$3$, we obtain $Y\in \langle G_{\mathcal{S}} \rangle^{[-v,+\infty)}$. By Lemma~\ref{range}, $\pd_{\mathcal{S}}(Y)\le o^{+}(G_{\mathcal{S}})+v$. Thus $\Fpd(\mathcal{S},G_{\mathcal{S}})\le o^{+}(G_{\mathcal{S}})+v$.

\smallskip
(ii) If $\Fpd(\mathcal{S},G_{\mathcal{S}})=+\infty$, then (ii) holds automatically.
So, we assume $\Fpd(\mathcal{S},G_{\mathcal{S}})<+\infty$.

Let $X\in\mathcal{R}^b\cap\mathcal{R}^{\ge 0}$ with $\pd_{\mathcal{R}}(X)<+\infty$. Define $w\coloneqq \Fpd(\mathcal{S},G_{\mathcal{S}})+b<+\infty$. In the following, we show  $\pd_{\mathcal{S}}(i_*(X))\leq w$.

By Lemma \ref{range-1}, $X\in\mathcal{R}^{{\rm sb}}$. It follows from the recollement $(\dag)$ that $i_*(X)\in\mathcal{S}^{{\rm sb}}$. Hence $\pd_{\mathcal{S}}(i_*(X))<\infty$ by Lemma \ref{range}. Since $i^*(G_{\mathcal{S}})\in \langle G_{\mathcal{R}}\rangle^{[a,b]}$ and $X\in \mathcal{R}^{\ge 0}$, we have $\Hom_{\mathcal{R}}(i^*(G_{\mathcal{S}})[p],X)=0$ for $p>b$.
Then $$\Hom_{\mathcal{S}}(G_{\mathcal{S}}[p],i_*(X))\simeq \Hom_{\mathcal{R}}(i^*(G_{\mathcal{S}})[p],X)=0 \;\text{for}\; p>b.$$
This implies $i_*(X)\in \mathcal{S}^{{\rm sb}}\cap\mathcal{S}^{\ge -b}\subseteq\mathcal{S}^b\cap\mathcal{S}^{\ge -b}$. By Remark \ref{some rems}(1), we obtain $$\pd_{\mathcal{S}}(i_*(X))=\pd_{\mathcal{S}}(i_*(X)[-b])+b\leq\Fpd(\mathcal{S},G_{\mathcal{S}})+b=w.$$ Further, by Lemma \ref{range-1}, $i_{*}(X)\in\overline{\langle G_{\mathcal{S}}\rangle}^{[-w-a(G_{\mathcal{S}}),+\infty)}$. Note that $i^*$ preserves coproducts and $i^*(G_{\mathcal{S}})=\langle G_{\mathcal{R}}\rangle^{[a,b]}$. This forces $i^*i_*(X)\in\overline{\langle G_{\mathcal{R}}\rangle}^{[a-w-a(G_{\mathcal{S}}),+\infty)}$. Since $i_{*}$ is fully faithful, there is an isomorphism $X\simeq i^*i_*(X)$. Now, by Lemma \ref{range}, $\pd_{\mathcal{R}}(X)\leq -a+w+a(G_{\mathcal{S}})+o^{+}(G_{\mathcal{R}})=\Fpd(\mathcal{S},G_{\mathcal{S}})+b-a+a(G_{\mathcal{S}})+o^{+}(G_{\mathcal{R}}).$
Thus
$\Fpd(\mathcal{R},G_{\mathcal{R}})\le\Fpd(\mathcal{S},G_{\mathcal{S}})+b-a+a(G_{\mathcal{S}})+o^{+}(G_{\mathcal{R}})$.

(2) If $\Fpd(\mathcal{S},G_{\mathcal{S}})=+\infty$, then (2) holds automatically.
So, we assume $\Fpd(\mathcal{S},G_{\mathcal{S}})<+\infty$.

Let $Z \in\mathcal{T}^b \cap \mathcal{T}^{\ge 0}$ with $\pd_{\mathcal{T}}(Z) < +\infty$. We first show $\pd_{\mathcal{S}}(j_!(Z))<+\infty$.

For this aim, we consider the glued $t$-structure $(\widetilde{\mathcal{S}}^{\leq 0},\widetilde{\mathcal{S}}^{\geq 1})$ associated with the $t$-structures $(\mathcal{R}^{\leq 0},\mathcal{R}^{\leq 1})$ and $(\mathcal{T}^{\leq 0},\mathcal{T}^{\geq 1})$ via the recollement $(\star)$, where\[\widetilde{\mathcal{S}}^{\leq0}:=\{X\in\mathcal{S}\,\mid\,i^{*}(X)\in\mathcal{R}^{\leq 0},\,j^{*}(X)\in\mathcal{T}^{\leq 0}\}\quad {\rm by \; Proposition}\; \ref{glued $t$-str}. \]
Since $(\widetilde{\mathcal{S}}^{\leq 0},\widetilde{\mathcal{S}}^{\geq 1})$ is in the preferred equivalence class by Lemma \ref{lem-gluing-$t$-struc}, we have  $\mathcal{S}^{\leq -m}\subseteq\widetilde{\mathcal{S}}^{\leq 0}\subseteq\mathcal{S}^{\leq m}$ for some $m\in\mathbb{N}$. Note that $j^*(\mathcal{S}^b)\subseteq\mathcal{T}^b$ by Lemma \ref{lem-gluing-$t$-struc}, and that $j^*(\mathcal{S}^{\leq 0})[m]=j^*(\mathcal{S}^{\leq -m})\subseteq j^*(\widetilde{\mathcal{S}}^{\leq 0})\subseteq \mathcal{T}^{\leq 0}$. This implies $j^*(\mathcal{S}^b\cap\mathcal{S}^{\leq 0})\subseteq \mathcal{T}^b\cap \mathcal{T}^{\leq m}$. By the adjoint isomorphism \[\Hom_{\mathcal{S}}(j_!(Z),Y')\simeq \Hom_{\mathcal{T}}(Z,j^*(Y'))\text{ for each }Y'\in\mathcal{S}^b\cap\mathcal{S}^{\leq 0},\] we have $\pd_{\mathcal{S}}(j_!(Z))\leq \pd_{\mathcal{T}}(Z)+m< +\infty$.

Assume that $j_!(\mathcal{T}^{\ge 0})\subseteq \mathcal{S}^{\ge -e}$ for some integer $e$.
Then $j_!(\mathcal{T}^+)\subseteq \mathcal{S}^+$, and $j_!(Z)\in \mathcal{S}^{\ge -e}$ by $Z\in\mathcal{T}^{\geq 0}$.  Observe that $j_!(\mathcal{T}^-)\subseteq \mathcal{S}^-$ by Lemma \ref{lem-gluing-$t$-struc}. Thus $j_!(\mathcal{T}^b)=j_!(\mathcal{T}^-\cap\mathcal{T}^+)\subseteq \mathcal{S}^-\cap\mathcal{S}^+=\mathcal{S}^b$. This forces $j_!(Z)\in\mathcal{S}^b$ since $Z\in\mathcal{T}^b$. By Remark \ref{some rems}(1), we have
$\pd_{\mathcal{S}}(j_!(Z))=\pd_{\mathcal{S}}(j_!(Z)[-e])+e\leq\Fpd(\mathcal{S},G_{\mathcal{S}})+e$. This means that, for any $Y'\in \mathcal{S}^{\le 0}\cap \mathcal{S}^b$,
\begin{align}\label{f-iso-ix-y-2}
\Hom_{\mathcal{S}}(j_!(Z),Y'[p])=0 \;\text{for}\; p>e+\Fpd(\mathcal{S},G_{\mathcal{S}}).
\end{align}
Moreover, by Lemma \ref{lem-fd-contain}, $\mathcal{T}^{\le 0}\cap \mathcal{T}^b \subseteq  j^*(\mathcal{S}^{\le d}\cap \mathcal{S}^b)$.
Now, we take an object $Z_1\in \mathcal{T}^{\le 0}\cap \mathcal{T}^b$. Then there exists an object $Y_1\in \mathcal{S}^{\le d}\cap \mathcal{S}^b$ such that $Z_1\simeq j^*(Y_1)$.
Clearly, $Y_1[d]\in \mathcal{S}^{\le 0}\cap \mathcal{S}^b$.
By (\ref{f-iso-ix-y-2}) and the adjoint pair $(j_!, j^*)$, for $p>d+e+\Fpd(\mathcal{S},G_{\mathcal{S}})$, there is a series of isomorphisms:
$$\Hom_{\mathcal{T}}(Z,Z_1[p])\simeq
\Hom_{\mathcal{T}}(Z,j^*(Y_1)[p])\simeq \Hom_{\mathcal{S}}(j_!(Z),Y_1[p])
\simeq \Hom_{\mathcal{S}}(j_!(Z),Y_1[d][p-d])=0.$$
This implies $\pd_{\mathcal{T}}(Z)\le d+e+\Fpd(\mathcal{S},G_{\mathcal{S}})$. Thus $\Fpd(\mathcal{T},G_{\mathcal{T}})\le d+e+\Fpd(\mathcal{S},G_{\mathcal{S}})$.
\end{proof}

As a preparation for showing Corollary \ref{mainthm-cx17}, we recall the following definition.

\begin{defn}\label{cowidth}\cite[Section~3.1]{cx17}
Let $R$ be a ring and $\cpx{Y}\in\D{R\Modcat}$.
Suppose that $\cpx{Y}$ is isomorphic in $\Db{R\Modcat}$ to a bounded complex
$\cpx{J}\in\Cb{\Imodcat{R}}$, where $\Imodcat{R}$ denotes the category of left injective $R$-modules.
The \textit{homological cowidth} of $\cpx{Y}$ is defined as
$$cw(\cpx{Y})\coloneqq \min\left\{
\alpha_{\cpx{I}}-\beta_{\cpx{I}}\,\geq 0\, \bigg |
\begin{array}{ll}
\cpx{I}\simeq \cpx{Y} \;\;\mbox{in}\;\;\Db{R\Modcat}\;\;\mbox{for}\;\;
\cpx{I}\in\Cb{\Imodcat{R}} \\
\;\;\mbox{with}\;\; I^i=0\;\;\mbox{for}\;\; i<\beta_{\cpx{I}}\;\;\mbox{or}\;\; i>\alpha_{\cpx{I}}
\end{array}
\right\}.$$
\end{defn}
\noindent Note that if $\cpx{Y}\in \overline{\langle\Hom_{\mathbb{Z}}(R,\mathbb{Q}/\mathbb{Z})\rangle}^{[u,v]}$, then $cw(\cpx{Y})\le v-u$.

\medskip
{\bf Proof of Corollary \ref{mainthm-cx17}.}
Let $\mathcal{R}\coloneqq \D{R\Modcat}$, $\mathcal{S}\coloneqq \D{S\Modcat}$ and $\mathcal{T}\coloneqq \D{T\Modcat}$.
By Example \ref{EXWAPP}(2), $\mathcal{R}$, $\mathcal{S}$ and $\mathcal{T}$ are approximable triangulated categories. Moreover, their compact generators can be chosen to be ${_R}R$, ${_S}S$ and ${_T}T$ with $$a(R)=o^{-}(R)=o^{+}(R)=a(S)=o^{-}(S)=o^{+}(S)=a(T)=o^{-}(T)=o^{+}(T)=0.$$
By Remark \ref{some rems}(3), $\fd(S)=\fpd(\mathcal{S}, S)$. Similar statements hold also for $R$ and $T$.
Recall that $\mathcal{S}^{c}\simeq\Kb{S\text{-}{\rm proj}}$ and $\mathcal{S}^{{\rm sb}}\simeq\Kb{S\text{-}{\rm Proj}}$. Thus Corollary \ref{mainthm-cx17}(1) follows by combining Remark~\ref{kuandu} with Theorems \ref{thm-fd-small}(1) and \ref{thm-fd-big}(1).

(2) We prove the case of finitistic dimension; the case of big finitistic dimension can be proved  analogously. Clearly, $\mathcal{T}^{b}=\Db{T\Modcat}$ and $\mathcal{S}^{b}=\Db{S\Modcat}$. Suppose that $j_{!}$ restricts to $\mathcal{T}^{b}\ra\mathcal{S}^{b}$. We now show that $j_{!}(\mathcal{T}^{\geq 0})\subseteq\mathcal{S}^{\geq -e}$ for some $e\in\mathbb{Z}$.

We consider the exact functor $(-)^\vee \coloneqq  \Hom_{\mathbb{Z}}(-, \mathbb{Q}/\mathbb{Z}) \colon \mathbb{Z}\Modcat \to \mathbb{Z}\Modcat.$ Since $\mathbb{Q}/\mathbb{Z}$ is an injective cogenerator for $\mathbb{Z}\Modcat$, a $\mathbb{Z}$-module $U$ is zero if and only if so is $U^\vee$. Let $\cpx{I}\coloneqq j^{!}(S^{\vee})$. By \cite[Lemma 3.3(2)]{cx17}, $\cpx{I}$ is isomorphic in $\mathcal{T}$ to a bounded complex of injective $T$-module.
We may assume that $\cpx{I}$ is of the form
$$0\lra I^{f}\lra I^{f+1}\lra \cdots \lra I^{e-1}\lra I^{e}\lra 0$$
where all $I^i$ are injective and $cw(\cpx{I}) = e - f$.
Now, let $\cpx{Z}\in \mathcal{T}^{\ge 0}$. The proof of \cite[Lemma 3.3(2)]{cx17} yields the following isomorphisms:
$$
(\ast)\quad [H^n(j_!(\cpx{Z}))]^{\vee}\simeq \Hom_{\mathcal{S}}(j_!(\cpx{Z})[n],S^{\vee})
\simeq \Hom_{\mathcal{T}}(\cpx{Z}[n],j^{!}(S^{\vee}))
=\Hom_{\mathcal{T}}(\cpx{Z}[n],\cpx{I}).$$
Note that $\Hom_{\mathcal{T}}(\cpx{Z}[n],\cpx{I}) \simeq \Hom_{\K{T\text{-}{\rm Inj}}}(\cpx{Z}[n],\cpx{I})$  by $\cpx{I}\in \Kb{\Imodcat{T}}$.
Since $\cpx{Z} \in \mathcal{T}^{\ge 0}$, we can write $\cpx{Z}$ as
$$0\lra Z^{0}\lra Z^{1}\lra Z^{2}\lra \cdots.$$ Then $\Hom_{\K{T\text{-}{\rm Inj}}}(\cpx{Z}[n],\cpx{I})=0$ for $n<-e$. This implies that
$[H^n(j_!(\cpx{Z}))]^{\vee} \simeq \Hom_{\K{T\text{-}\mathrm{Inj}}}(\cpx{Z}[n], \cpx{I})= 0$ for $n <-e$. Thus
$H^n(j_!(\cpx{Z})) = 0$ for $n < -e$; in other words, $j_!(\cpx{Z}) \in \mathcal{S}^{\ge -e}$.

Thanks to Theorem \ref{thm-fd-small}(2), $\fd(T)\le e+d+\fd(S)$, where $j_!(T)\in \langle S\rangle^{[c,d]}$ for some integers $c$ and $d$. Since $[H^n(j_!(T))]^{\vee}\simeq H^{-n}(\cpx{I})$ by $(\ast)$ and $H^{-n}(\cpx{I})=0$ for $n>-f$,  we have $[H^n(j_!(T))]^{\vee}=0$ for $n>-f$.
Thus $H^n(j_!(T))=0$ for $n>-f$. Note that $j_!(T)$ is isomorphic to a bounded complex of finitely generated projective $S$-modules. So, we can choose $d\leq -f$. It follows that
$$\fd(T)\le e+d+\fd(S)\le e-f+\fd(S) =\fd(S)+cw(\cpx{I})
=\fd(S)+cw(j^!(\Hom_{\mathbb{Z}}(S,\mathbb{Q}/\mathbb{Z}))).$$
This shows the inequality in Corollary \ref{mainthm-cx17}(2). \hfill$\square$

\section{Strongly compactly generated triangulated categories}\label{5}
This section is devoted to giving the reductions of both global dimension and strong compact generation for (weakly) approximable triangulated categories in recollements. Also, a relationship between global dimension and strong compact generation for compactly generated triangulated categories is established.

\subsection{Global dimension of triangulated categories}

In this subsection, we consider the finiteness of global dimension for weakly approximable triangulated categories and give a reduction of the finiteness of global dimension by recollements.

\begin{defn}\label{global dimension}
Let $\mathcal{S}$ be a compactly generated triangulated category with a compact generator $G$ and let $(\mathcal{S}_{G}^{\leq 0},\mathcal{S}_{G}^{\geq 1})$ be the $t$-structure generated by $G$ (see Lemma \ref{lem-p-w-gen}).
The \emph{global dimension} of $\mathcal{S}$ with respect to $G$ is defined as
\[\gpd(\mathcal{S}, G)\coloneqq \sup\{\pd_G(X)\mid X\in \mathcal{S}_{G}^{b}\cap\mathcal{S}_{G}^{\ge 0}\}.\]
\end{defn}
Note that if $\gpd(\mathcal{S}, G) < +\infty$, then $\gpd(\mathcal{S}, G) = \operatorname{Fpd}(\mathcal{S}, G)$. The following result provides several equivalent characterizations of global dimension.

\begin{lem}\label{lem-global}
Let $\mathcal{H}$ be the heart of the $t$-structure $(\mathcal{S}_{G}^{\leq 0},\mathcal{S}_{G}^{\geq 1})$ on $\mathcal{S}$. Then:

$(1)$
$$
\gpd(\mathcal{S},G)= \sup\{\pd_{G}(X) \mid X\in\mathcal{H}\}=
\inf\{d\in\mathbb{Z} \mid \Hom_{\mathcal{S}}(X,Y[i])=0,\;\forall\, X,Y\in\mathcal{H},\,i>d\}$$
$$
=\inf\{d\ge 0 \mid \Hom_{\mathcal{S}}(\mathcal{S}_G^{\geq 0}\cap\mathcal{S}^{b},\mathcal{S}_G^{\leq-(d+1)}\cap\mathcal{S}^{b})=0\}.\quad\quad\quad\quad\quad\quad $$
$(2)$ Suppose that $\mathcal{S}$ is weakly approximable and $G$ is bounded. Then   $\mathcal{S}^b=\mathcal{S}^{{\rm sb}}$ if and only if ${\rm gpd}(\mathcal{S},G)<+\infty$.
\end{lem}

\begin{proof}
$(1)$ Observe that $\mathcal{S}_{G}^{b}\cap\mathcal{S}_{G}^{\ge 0}=\bigcup_{n\in\mathbb{N}}\mathcal{H}[-n]$ and that $\pd_G(X[k])=\pd_G(X)+k$ for any $X\in\mathcal{S}^b$ and $k\in\mathbb{Z}$ by Remark \ref{some rems}(1). Thus $(1)$ holds.

$(2)$  Since $G$ is bounded, the inclusion $\mathcal{S}^{{\rm sb}}\subseteq\mathcal{S}^{b}$ holds by Proposition~\ref{lem-sp=big}.

Suppose $d\coloneqq {\rm gpd}(\mathcal{S},G)<+\infty$. Let $M\in\mathcal{S}^{b}$. Then there exists $m\geq 0$ such that $M[-m]\in\mathcal{S}^{b}\cap\mathcal{S}^{\geq 0}$. Hence $\pd_{G}(M)=\pd_G(M[-m])+m\leq d+m<+\infty$. Since $\mathcal{S}$ is weakly approximable and $G$ is bounded, it follows from Proposition \ref{Relation}(1) that $M\in\mathcal{S}^{{\rm sb}}$. Thus $\mathcal{S}^b=\mathcal{S}^{{\rm sb}}$.

Conversely, suppose $\mathcal{S}^{{\rm sb}}=\mathcal{S}^b$. Still by Proposition \ref{Relation}(1), every object in $\mathcal{S}^b$ has finite projective dimension. Assume $\gpd(\mathcal{S},G)=+\infty$. By $(1)$, for each $i\in\mathbb{N}$, there exists an object $X_{i}\in\mathcal{H}$ such that $\pd_{G}(X_{i})\geq i$. Now, let $X:=\bigoplus_{i\in\mathbb{N}}X_i\in \mathcal{S}$. Since $\mathcal{H}$ is closed under coproducts in $\mathcal{S}$, we have $X\in\mathcal{H}$ and further $\pd_G(X)<\infty$. However, since $X_i$ is a direct summand of $X$, it is clear that $\pd_G(X)\geq \pd_G(X_i)\geq i$. Thus $\pd_{G}(X)=+\infty$, which leads to a contradiction.
\end{proof}

\begin{remark}\label{rem-global}
(1) Let $S$ be a ring. The \emph{global dimension} of $S$, denoted by ${\rm gldim}(S)$, is defined as
${\rm gldim}(S) \coloneqq \sup\{\pd_S(M) \mid M \in S\text{-Mod}\}.$
Then ${\rm gldim}(S)={\rm gpd}(\mathscr{D}(S\Modcat), S)$.

(2) By Lemma \ref{lem-global}(1), our definition ${\rm gpd}(\mathcal{S}, G)$ coincides with the global dimension of the bounded $t$-structure $(\mathcal{S}_G^{\leq 0}\cap\mathcal{S}^b,\mathcal{S}_G^{\geq 1}\cap\mathcal{S}^{b})$ on $\mathcal{S}^b$ defined in \cite[Definition 3.3]{clz23}.
\end{remark}

The main result of this subsection is the following.
\begin{thm}\label{global dim}
Let $\mathcal{R}$, $\mathcal{S}$, and $\mathcal{T}$ be \emph{weakly approximable} triangulated categories with respective bounded compact generators $G_{\mathcal{R}}$, $G_{\mathcal{S}}$, and $G_{\mathcal{T}}$. Suppose that there is a recollement:
\begin{align*}
\xymatrixcolsep{4pc}\xymatrix{\mathcal{R} \ar[r]|{i_*=i_!} &\mathcal{S} \ar@<-2ex>[l]|{i^*} \ar@<2ex>[l]|{i^!} \ar[r]|{j^!=j^*}  &\mathcal{T}. \ar@<-2ex>[l]|{j_!} \ar@<2ex>[l]|{j_{*}}
}
\end{align*}
Then
${\rm gpd}(\mathcal{S},G_{\mathcal{S}})<+\infty$ if and only if ${\rm gpd}(\mathcal{R},G_{\mathcal{R}})<+\infty$ and ${\rm gpd}(\mathcal{T},G_{\mathcal{T}})<+\infty$.
\end{thm}
\begin{proof}
 We first show the necessity of Theorem \ref{global dim}.

Suppose ${\rm gpd}(\mathcal{S},G_{\mathcal{S}})<+\infty$. Then $\mathcal{S}^b=\mathcal{S}^{{\rm sb}}$ by Lemma ~\ref{lem-global}(2). Further, by Lemma~\ref{lem-gluing-$t$-struc}, $i_*$ restricts to  $\mathcal{R}^{b} \to \mathcal{S}^{{b}}$. It follows that $i_*$ restricts to $\mathcal{R}^{b} \to \mathcal{S}^{{\rm sb}}$. This implies $i_*(G_\mathcal{R})\in\mathcal{S}^{{\rm sb}}$, due to $G_\mathcal{R}\in\mathcal{R}^b$. Now, by the equivalences of $(1)$ and $(2)$ in Theorem~\ref{prop-rec-res1}, $i^*$ restricts to $\mathcal{S}^{{\rm sb}}\to \mathcal{R}^{{\rm sb}}$. Note that $i_{*}$ is fully faithful. Consequently, $X\simeq i^{*}i_{*}(X)\in\mathcal{R}^{{\rm sb}}$ for any $X\in\mathcal{R}^{b}$.
This shows $\mathcal{R}^{b}\subseteq \mathcal{R}^{{\rm sb}}$. Since $\mathcal{R}$ is weakly approximable with a bounded compact generator, $\mathcal{R}^{{\rm sb}}\subseteq \mathcal{R}^{b}$ by Proposition \ref{lem-sp=big}. Thus $\mathcal{R}^{{\rm sb}}= \mathcal{R}^{b}$, and further, ${\rm gpd}(\mathcal{R},G_{\mathcal{R}})<+\infty$ by Lemma \ref{lem-global}(2).

Since $i_{*}(G_{\mathcal{R}})\in \mathcal{S}^{{\rm sb}}$,  we see from Theorem~\ref{prop-rec-res1} that $j^*$ restricts to $\mathcal{S}^{{\rm sb}}\to \mathcal{T}^{{\rm sb}}$ and $j_{*}$ restricts to $\mathcal{T}^{b}\to\mathcal{S}^{b}$. Now, let $Y\in\mathcal{T}^{b}$.
Since $j_{*}(Y)\in\mathcal{S}^{b}=\mathcal{S}^{{\rm sb}}$ and $j_*$ is fully faithful, we have
$Y\simeq j^{*}j_{*}(Y)\in\mathcal{T}^{{\rm sb}}$.
This shows $\mathcal{T}^{b}\subseteq\mathcal{T}^{{\rm sb}}$. Similarly, we can show ${\rm gpd}(\mathcal{T},G_{\mathcal{T}})<+\infty$ by Lemma \ref{lem-global}(2).

Next, we show the sufficiency of Theorem \ref{global dim}.

Suppose $\mathrm{gpd}(\mathcal{R}, G_{\mathcal{R}}) < +\infty$ and $\mathrm{gpd}(\mathcal{T}, G_{\mathcal{T}}) < +\infty$. By Lemma \ref{lem-global}(2), $\mathcal{R}^b = \mathcal{R}^{\mathrm{sb}}$ and $\mathcal{T}^b = \mathcal{T}^{\mathrm{sb}}$. Since $\mathcal{S}$ is weakly approximable with a bounded compact generator, $\mathcal{S}^{\rm sb}\subseteq \mathcal{S}^b$ by Proposition \ref{lem-sp=big}. Moreover, by Lemma~\ref{lem-gluing-$t$-struc}, $j^{*}$ restricts to $\mathcal{S}^b \to \mathcal{T}^b=\mathcal{T}^{\mathrm{sb}}$. Since $G_\mathcal{S}\in\mathcal{S}^b$, we have $j^*(G_\mathcal{S})\in \mathcal{T}^{\mathrm{sb}}$. It follows from Theorem~\ref{prop-rec-res1} that there exists a left recollement
\begin{align*}
\xymatrixcolsep{4pc}\xymatrix{
\mathcal{R}^b=\mathcal{R}^{\mathrm{sb}} \ar[r]|{i_*=i_!}
&\mathcal{S}^{\mathrm{sb}} \ar@<-2ex>[l]|{i^*}   \ar[r]|{j^!=j^*}
&\mathcal{T}^{\mathrm{sb}}=\mathcal{T}^b. \ar@<-2ex>[l]|{j_!}
}
\end{align*}
and a right recollement
\begin{align*}
\xymatrixcolsep{4pc}\xymatrix{
\mathcal{R}^{\mathrm{sb}} =\mathcal{R}^b \ar[r]|{i_*=i_!}
&\mathcal{S}^b  \ar@<2ex>[l]|{i^!} \ar[r]|{j^!=j^*}
&\mathcal{T}^b=\mathcal{T}^{\mathrm{sb}}.   \ar@<2ex>[l]|{j_{*}}
}
\end{align*}
Consequently, the same functor $j^*$ induces two triangle equivalences $\mathcal{S}^{\rm sb}/i_*(\mathcal{R}^b)\to \mathcal{T}^b$ and $\mathcal{S}^b/i_*(\mathcal{R}^b)\to \mathcal{T}^b$. This forces $\mathcal{S}^b=\mathcal{S}^{{\rm sb}}$.
Thus ${\rm gpd}(\mathcal{S},G_{\mathcal{S}})<+\infty$ by Lemma \ref{lem-global}(2).
\end{proof}

The following result generalizes \cite[Proposition~4]{k91} which is concentrated on unbounded derived categories of rings.
\begin{prop}\label{prop-global dim-rec}
Let $\mathcal{R}$, $\mathcal{S}$, and $\mathcal{T}$ be weakly approximable triangulated categories with respective bounded compact generators $G_{\mathcal{R}}$, $G_{\mathcal{S}}$, and $G_{\mathcal{T}}$. Suppose that there is a recollement:
\begin{align*}
\xymatrixcolsep{4pc}\xymatrix{\mathcal{R} \ar[r]|{i_*=i_!} &\mathcal{S} \ar@<-2ex>[l]|{i^*} \ar@<2ex>[l]|{i^!} \ar[r]|{j^!=j^*}  &\mathcal{T}. \ar@<-2ex>[l]|{j_!} \ar@<2ex>[l]|{j_{*}}
}
\end{align*}
If either
${\rm gpd}(\mathcal{S},G_{\mathcal{S}})<+\infty$ or ${\rm gpd}(\mathcal{T},G_{\mathcal{T}})<+\infty$, then the given recollement restricts to the following recollement:
\begin{align*}
\xymatrixcolsep{4pc}\xymatrix{\mathcal{R}^b \ar[r]|{i_*=i_!} &\mathcal{S}^b \ar@<-2ex>[l]|{i^*} \ar@<2ex>[l]|{i^!} \ar[r]|{j^!=j^*}  &\mathcal{T}^b. \ar@<-2ex>[l]|{j_!} \ar@<2ex>[l]|{j_{*}}
}
\end{align*}
\end{prop}
\begin{proof}
The case where ${\rm gpd}(\mathcal{S},G_{\mathcal{S}})<+\infty$ follows directly from the argument in the proof of Theorem~\ref{global dim}.

We now assume that ${\rm gpd}(\mathcal{T},G_{\mathcal{T}})<+\infty$. Then $\mathcal{T}^b=\mathcal{T}^{{\rm sb}}$ by Lemma ~\ref{lem-global}(2).
Moreover, Lemma~\ref{lem-gluing-$t$-struc} implies that $j^*$ restricts to  $\mathcal{S}^{b} \to \mathcal{T}^{{b}}$.
It follows that $j^*$ restricts to $\mathcal{S}^{b} \to \mathcal{T}^{{\rm sb}}$. This implies $j^*(G_\mathcal{S})\in\mathcal{T}^{{\rm sb}}$, due to $G_\mathcal{S}\in\mathcal{S}^b$.
By the equivalences of the conditions $(1')$ and $(4)$ in Theorem~\ref{prop-rec-res1}, the given recollement induces a right recollement:
\begin{align*}
\xymatrixcolsep{4pc}\xymatrix{
\mathcal{R}^b \ar[r]|{i_*=i_!}
&\mathcal{S}^b  \ar@<2ex>[l]|{i^!} \ar[r]|{j^!=j^*}
&\mathcal{T}^b=\mathcal{T}^{{\rm sb}},   \ar@<2ex>[l]|{j_{*}}
}
\end{align*}
and $j_!$ restricts to  $\mathcal{T}^{{\rm sb}}\to \mathcal{S}^{{\rm sb}}$. Since $\mathcal{S}$ is weakly approximable with a bounded compact generator, $\mathcal{S}^{{\rm sb}}\subseteq \mathcal{S}^{b}$ by Proposition \ref{lem-sp=big}. Consequently $j_!$ restricts to  $\mathcal{T}^{b}\to \mathcal{S}^{b}$ by Theorem~\ref{prop-rec-res1}.

To complete the proof, it remains to show that $i^*$ restricts to  $\mathcal{S}^{b}\to \mathcal{R}^{b}$. For any $Y\in \mathcal{S}^b$, we prove that $i^*(Y)\in \mathcal{R}^b$.
Consider the canonical triangle in $\mathcal{S}$:
$$j_!j^*(Y)\ra Y\ra i_*i^*(Y)\ra j_!j^*(Y)[1].$$
Since both  $Y$ and $j_!j^*(Y)[1]$ lie in $\mathcal{S}^b$, we obtain $i_*i^*(Y)\in \mathcal{S}^b$. By Lemma \ref{PR}(2), $\Hom_{\mathcal{S}}(G_{\mathcal{S}}[p],i_*i^*(Y))=0$ for $|p|\gg 0$.
Using the adjoint pair $(i^*, i_*)$, we deduce
$$\Hom_{\mathcal{R}}(i^*(G_{\mathcal{S}})[p],i^*(Y))\simeq \Hom_{\mathcal{S}}(G_{\mathcal{S}}[p],i_*i^*(Y))=0 \text{ for }|p|\gg 0.$$
Since $i^*(G_{\mathcal{S}})$ is a compact generator of $\mathcal{R}$, there exist integers $u \leq v$ such that  $G_{\mathcal{R}}\in \langle i^*(G_{\mathcal{S}})\rangle^{[u,v]}$. This implies $\Hom_{\mathcal{R}}(G_{\mathcal{R}}[p],i^*(Y))=0 \text{ for }|p|\gg 0.$
Applying Lemma \ref{PR}(2) again, we conclude $i^*(Y)\in \mathcal{R}^b$.
\end{proof}

\begin{remark}
For a compactly generated triangulated category $\mathcal{K}$, Kostas, Psaroudakis and Vitória recently introduced the notion of far-away orthogonality and defined in \cite[Definition 3.1]{kpv25} the triangulated subcategories $\mathcal{K}^{b}$ and $\mathcal{K}_{p}^{b}$ of $\mathcal{K}$ consisting of all \emph{bounded} and \emph{bounded projective} objects, respectively. When $\mathcal{K}^{b}=\mathcal{K}_{p}^{b}$, $\mathcal{K}$ is said to be of \emph{finite global dimension} (see \cite[Definition 4.2]{kpv25}). Moreover, a reduction of the finiteness of global dimension under (bounded) recollements of compactly generated triangulated categories such that bounded projective objects are bounded was done in \cite[Proposition 6.10]{kpv25}.

By Proposition \ref{lem-sp=big} and \cite[Lemma~3.12 and Remark~3.15]{kpv25}, if $\mathcal{K}$ is weakly approximable with a bounded compact generator $G_{\mathcal{K}}$, then $\mathcal{K}^{b}$ and $\mathcal{K}_{p}^{b}$ coincide with the subcategories $\mathcal{K}^{b}$ and $\mathcal{K}^{{\rm sb}}$ defined in our paper, respectively. So, in this case, the category $\mathcal{K}$ is of finite global dimension if and only if
$\gpd(\mathcal{K}, G_{\mathcal{K}}) < +\infty$. Thus the sufficiency of Theorem \ref{global dim} also follows from Proposition \ref{lem-sp=big} and \cite[Proposition 6.10(2)]{kpv25}. In the case of the existence of a recollement among $\mathcal{R}^{b}$, $\mathcal{S}^{b}$ and $\mathcal{T}^{b}$, Theorem \ref{global dim} is also a consequence of Proposition \ref{lem-sp=big} and \cite[Proposition 6.10(2)]{kpv25}.
\end{remark}

Applying Theorem \ref{global dim} to unbounded derived categories of rings, we obtain the following result.
\begin{cor}{\rm \cite[Theorem 3.17]{cx17}}
Let $R$, $S$ and $T$ be three rings.
Suppose that there is a recollement among the derived categories
$\D{R\Modcat}$, $\D{S\Modcat}$ and $\D{T\Modcat}$$:$
\begin{align}\label{cor-rec-gd}
\xymatrixcolsep{4pc}\xymatrix{
\D{R\Modcat} \ar[r]|{i_*=i_!} &\D{S\Modcat} \ar@<-2ex>[l]|{i^*} \ar@<2ex>[l]|{i^!} \ar[r]|{j^!=j^*}  &\D{T\Modcat}. \ar@<-2ex>[l]|{j_!} \ar@<2ex>[l]|{j_{*}}
}
\end{align}
Then ${\rm gldim}(S)<+\infty$ if and only if ${\rm gldim}(R)<+\infty$ and ${\rm gldim}(T)<+\infty$.
Moreover,

{\rm (1)} ${\rm gldim}(R)\le {\rm gldim}(S)+w(i^*(S))$ and ${\rm gldim}(T)\le {\rm gldim}(S)+cw(j^!(\Hom_{\mathbb{Z}}(S,\mathbb{Q}/\mathbb{Z})))$.

{\rm (2)} ${\rm gldim}(S)\le {\rm gldim}(R)+{\rm gldim}(T)+w(i_*(R))+w(j_!(T))+1$.
\end{cor}

Finally, we consider special triangulated categories. Let $R$ be a commutative Noetherian ring.
An $R$-linear weakly approximable triangulated category $\mathcal{S}$ is said to be \emph{locally Hom-finite} if $\Hom_{\mathcal{S}}(X,Y)\in R\modcat$ for any $X$, $Y\in\mathcal{S}^{c}$, where $R\modcat$ denotes the category of \emph{finitely generated} $R$-modules.

\begin{prop}\label{App-case}
Let the following diagram be a recollement of $R$-linear weakly approximable triangulated categories with bounded compact generators $G_{\mathcal{R}}$, $G_{\mathcal{S}}$ and $G_{\mathcal{T}}$, respectively:
\begin{align*}
\xymatrixcolsep{4pc}\xymatrix{\mathcal{R} \ar[r]|{i_*=i_!} &\mathcal{S} \ar@<-2ex>[l]|{i^*} \ar@<2ex>[l]|{i^!} \ar[r]|{j^!=j^*}  &\mathcal{T}. \ar@<-2ex>[l]|{j_!} \ar@<2ex>[l]|{j_{*}}
}
\end{align*}
Assume that $\mathcal{S}$ and $\mathcal{T}$ are locally Hom-finite, and that $\mathcal{T}$ is approximable with ${\rm gpd}(\mathcal{T},G_{\mathcal{T}})<+\infty$. Then $\fpd(\mathcal{S},G_{\mathcal{S}}) < +\infty \;\bigl(\text{resp. }\Fpd(\mathcal{S},G_{\mathcal{S}}) < +\infty\bigr)$ if and only if $\fpd(\mathcal{R},G_{\mathcal{R}}) < +\infty \;\bigl(\text{resp. }\Fpd(\mathcal{R},G_{\mathcal{R}}) < +\infty\bigr)$.
\end{prop}
\begin{proof}
Since ${\rm gpd}(\mathcal{T},G_{\mathcal{T}})<+\infty$, we have $\fpd(\mathcal{T},G_{\mathcal{T}}) \leq \Fpd(\mathcal{T},G_{\mathcal{T}})={\rm gpd}(\mathcal{T},G_{\mathcal{T}})$. Then Proposition \ref{App-case} follows from Theorems~\ref{thm-fd-small}(1) and \ref{thm-fd-big}(1) by showing $i_{*}(G_{\mathcal{R}}) \in \mathcal{S}^{c}$. By Lemma~\ref{lem-key}(2), it suffices to show $j^{*}(G_{\mathcal{S}}) \in \mathcal{T}^c$.

For checking this, we first apply a Brown representability theorem for approximable triangulated categories (namely, \cite[Theorem~1.4(i)]{n18}) to show $j^{*}(G_{\mathcal{S}})\in\mathcal{T}_{c}^b$.

Let $F:=\Hom_{\mathcal{T}}(-, j^{*}(G_{\mathcal{S}})): (\mathcal{T}^c)\opp\to R\Modcat$. Then $F$ is a cohomological functor. By \cite[Remark 1.2]{n18}, it suffices to show that $F$ is $G_\mathcal{T}$-finite, that is, $F(G_\mathcal{T}[i])\in R\modcat$ for each $i\in\mathbb{Z}$ and $F(G_\mathcal{T}[i])=0$ for $|i|\gg 0$.

In fact, since $j_{!}$ preserves compact objects by Lemma~\ref{lem-key}(1) and $\mathcal{S}$ is locally Hom-finite, it follows from the adjoint pair $(j_!, j^*)$ that $$F(X)=\Hom_{\mathcal{T}}(X,j^{*}(G_{\mathcal{S}}))\simeq\Hom_{\mathcal{S}}(j_{!}(X),G_{\mathcal{S}})\in R\text{-}{\rm mod}$$ for each $X\in\mathcal{T}^{c}=\langle G_\mathcal{T}\rangle$.
In particular, $F(G_\mathcal{T}[i]))\in R\modcat$.
Note that $j^{*}(G_{\mathcal{S}})\in\mathcal{T}^{b}$ by Lemma~\ref{lem-gluing-$t$-struc}. By the characterization of $\mathcal{T}^b$ in Lemma \ref{PR}(2) for the weakly approximable triangulated category $\mathcal{T}$, we have $F(G_{\mathcal{T}}[i])=\Hom_{\mathcal{T}}(G_{\mathcal{T}}[i],j^{*}(G_{\mathcal{S}}))=0$ for $|i|\gg 0$. Thus $F$ is a $G_\mathcal{T}$-finite cohomological functor.

Now, since $\mathcal{T}$ is locally Hom-finite and \emph{approximable} and $F$ is $G_\mathcal{T}$-finite, we see from \cite[Theorem~1.4(i)]{n18} that there exists an object $Z\in\mathcal{T}_{c}^b$ and a natural isomorphism of cohomological functors:
$$
\alpha:\Hom_{\mathcal{T}}(-,Z)|_{\mathcal{T}^{c}}
\simeq\Hom_{\mathcal{T}}(-,j^{*}(G_{\mathcal{S}}))|_{\mathcal{T}^{c}}: (\mathcal{T}^{c})^{{\rm op}}\to R\Modcat.
$$
Moreover, by \cite[Lemma~9.5(ii) and Lemma~7.8]{n18}, there is a morphism $f:Z\to j^{*}(G_{\mathcal{S}})$ in $\mathcal{T}$ such that $\alpha=\Hom_{\mathcal{T}}(-,f)|_{\mathcal{T}^{c}}$. This implies $\Hom_\mathcal{T}(\mathcal{T}^c, \Cone(f))=0$, where $Z\to j^{*}(G_{\mathcal{S}})\to \Cone(f)\to Z[1]$ is a triangle in $\mathcal{T}$. Since $\mathcal{T}^c$ generates $\mathcal{T}$, the morphism $f$ is an isomorphism. Thus $j^{*}(G_{\mathcal{S}})\simeq Z\in\mathcal{T}_{c}^b$.

Since $\mathcal{T}$ is approximable with ${\rm gpd}(\mathcal{T},G_{\mathcal{T}})<+\infty$, we see from Lemma \ref{lem-global}(2) that $\mathcal{T}^{\mathrm{sb}} = \mathcal{T}^{b}$. This gives $j^{*}(G_{\mathcal{S}}) \in \mathcal{T}^{\mathrm{sb}}\cap\mathcal{T}_{c}^b$.
It remains to show $\mathcal{T}^{\mathrm{sb}}\cap\mathcal{T}_{c}^b=\mathcal{T}^{c}$.

Since $\mathcal{T}^{c}=\langle G_{\mathcal{T}}\rangle$ and $G_{\mathcal{T}}$ are bounded, we have $\mathcal{T}^{c}\subseteq\mathcal{T}^{{\rm sb}}\cap\mathcal{T}_{c}^b$. To show the converse, let $X\in\mathcal{T}^{\mathrm{sb}}\cap\mathcal{T}_{c}^b$. Since $\mathcal{T}$ is weakly approximable, it follows from Proposition \ref{lem-sp=big} that there is a positive integer $m$ such that $X\in{}^{\perp}(\mathcal{T}_{G_\mathcal{T}}^{\leq -m})$. Further, by the definition of $\mathcal{T}_{c}^{b}$, there is a triangle $C_{m}\to X\to D_{m}\to C_{m}[1]$ with $C_{m}\in\mathcal{T}^{c}$ and $D_{m}\in\mathcal{T}_{G_\mathcal{T}}^{\leq -m}$. Consequently, $\Hom_\mathcal{T}(X, D_m)=0$, and further  $X$ is a direct summand of $C_m$. This forces $X\in\mathcal{T}^{c}$, and therefore $\mathcal{T}^{\mathrm{sb}}\cap\mathcal{T}_{c}^b=\mathcal{T}^{c}$. Thus $j^{*}(G_{\mathcal{S}}) \in \mathcal{T}^{\mathrm{sb}}\cap\mathcal{T}_{c}^b=\mathcal{T}^c$.
\end{proof}

\subsection{Strong compact generation and finite global dimension}
In this subsection, we discuss connections between strong compact generation and finite global dimension for compactly generated triangulated categories. It turns out that a categorical version of Kelly’s theorem on rings is obtained (see \cite{k65}).

We now fix a compactly generated triangulated category $\mathcal{S}$ with a compact generator
$G\neq 0$. Let $(\mathcal{S}^{\leq 0}, \mathcal{S}^{\geq 1}):= (\mathcal{S}_{G}^{\leq 0}, \mathcal{S}_{G}^{\geq 1})$ be the $t$-structure on $\mathcal{S}$ generated by $G$ (see Lemma~\ref{lem-p-w-gen}) with its heart and the associated cohomological functor given by
$$\mathcal{H}:=\mathcal{S}_{G}^{\leq 0}\cap \mathcal{S}_{G}^{\geq 0}\quad \mbox{and}\quad H: \mathcal{S} \to \mathcal{H}.$$
Then $\mathcal{H}$ is an abelian category. Further, we define $H^i:=H\circ[i]$ for any $i\neq 0$, and $H^0:=H$.

We say that $\mathcal{S}$ is \textit{strongly compactly generated} if there exists an integer $n\geq 0$ such that $\mathcal{S} = \overline{\langle G\rangle}_{n+1}^{(-\infty, +\infty)}$. In this case, $G$ is called a \textit{strong compact generator} for $\mathcal{S}$.

We state two important classes of strongly compactly generated triangulated categories: $(a)$ the derived category of a smooth dg algebra (see \cite[Lemma~3.6(a)]{l10}); $(b)$ the derived category $\mathscr{D}_{\mathrm{qc}}(X)$ for a separated, quasicompact scheme $X$ covered by open affine subschemes $\mathrm{Spec}(R_{i})$ with each $R_{i}$ of finite global dimension (see \cite[Theorem~2.1]{n21})

The main result of this subsection is the following.

\begin{thm}\label{main-thm-scg-gd}
Let $\mathcal{S}$ be a triangulated category with coproducts and $0\neq G \in \mathcal{S}$ a compact generator such that $\Hom(G, G[i]) = 0$ for $|i| \gg 0$. For each $n\in\mathbb{N}$, we consider the following conditions:

{\rm (1)} $\mathcal{S} = \overline{\langle G\rangle}_{n+1}^{(-\infty, +\infty)}$.
In particular, $\mathcal{S}$ is strongly compactly generated.

{\rm (2)} $o^+(G)-o^-(G)\leq {\rm gpd}(\mathcal{S},G)\leq o^+(G)+n(1+o^-(G))<+\infty$.

\noindent Then $(1)$ implies $(2)$. If $\Hom_\mathcal{S}(G,G[i])=0$ for all $i\neq 0$, then $(2)$ implies $(1)$.
\end{thm}

Before giving a proof of Theorem \ref{main-thm-scg-gd}, we first recall the following definition.

\begin{defn}{\rm \cite[Definition 4.1]{pv18}}
Let $\mathcal{K}$ be a triangulated category with coproducts. An object $M\in\mathcal{K}$ is \emph{silting} if the pair $(M[1,+\infty)^{\perp},M(-\infty,0]^{\perp})$ forms a $t$-structure on $\mathcal{K}$. If $M$ is silting and $\Add(M)\subseteq M(-\infty,-1]^{\perp}\cup M[1,+\infty)^{\perp}$, then $M$ is said to be \emph{tilting}.
\end{defn}

Silting objects of triangulated categories admit many nice properties such as

\begin{lem}{\rm \cite[Proposition 4.3]{pv18}}\label{prop-sil}
Let $\mathcal{K}$ be a triangulated category and $M$ a silting object of $\mathcal{K}$. Then
$M(-\infty, +\infty)^{\perp}=0$ and $\mathbbm{t}:=(M[1,+\infty)^{\perp},M(-\infty,0]^{\perp})$ is a nondegenerate $t$-structure. If $M$ is tilting, then $M$ is a projective generator in the heart of the $t$-structure $\mathbbm{t}$.
\end{lem}

\begin{remark}\label{rmk7}
If $G\in\mathcal{S}$ is silting, then it is easy to see that $\Hom_{\mathcal{S}}(G,G[i])=0$ for $i>0$. Hence, $\mathcal{S}$ is approximable (see \cite[Remark~5.3]{n18}) and the $t$-structures $(G[1,+\infty)^{\perp},G(-\infty,0]^{\perp})$ and $(\mathcal{S}^{\leq 0},\mathcal{S}^{\geq 1})$ on $\mathcal{S}$ are the same.
\end{remark}

\begin{lem}\label{proj obj}
{\rm (1)} Let $P \in \mathcal{H}$. If $\pd_G(P) = 0$, then $P$ is projective in $\mathcal{H}$. Conversely, if $G$ is tilting and $P$ is projective in $\mathcal{H}$, then $\pd_G(P) = 0$.

{\rm (2)} Let  $0\to L\to M\to N\to 0$ be an exact sequence in $\mathcal{H}$ with $\pd_{G}(M)<\pd_{G}(N)<+\infty$. Then $\pd_{G}(N)=\pd_G(L)+1$.

$(3)$ Suppose that $G\in\mathcal{S}$ is tilting. Then:

\quad\; $(a)$ If $P\in\mathcal{H}$ is projective and $Y \in \mathcal{S}$, then the canonical morphism
\[
H^i(-):\;\Hom_{\mathcal{S}}\big(P[-i], Y\big) \lra \Hom_{\mathcal{H}}\big(P, H^i(Y)\big)
\]
is an isomorphism for any $i \in \mathbb{Z}$.

\quad\; $(b)$ Let $n \in \mathbb{N}$ and $X \in \mathcal{S}$ with $\pd_{G}(H^{i}(X))\leq n$ for all $i\in\mathbb{Z}$. Then $X\in\overline{\langle G\rangle}_{n+1}$.
\end{lem}
\begin{proof}
(1) Let $f\colon  M \to N$ be an epimorphism in the abelian category $\mathcal{H}$. This induces a triangle
$L\to M \xrightarrow{f} N \to L[1] $ in $\mathcal{S}$ with $L \in \mathcal{H}\subseteq \mathcal{S}^{b} \cap \mathcal{S}^{\leq 0}$. Since $\pd_{G}(P) = 0$, we have $\Hom_{\mathcal{S}}(P, L[1]) = 0$.
This implies that $\Hom_\mathcal{S}(P, f)$ is surjective. Thus $P$ is projective in $\mathcal{H}$.

Suppose that $G\in\mathcal{S}$ is tilting. Since $\mathcal{S}^{\leq -1}=G[0,+\infty)^{\perp}$, we have the inclusion $(\ast): G \subseteq {}^{\perp}(\mathcal{S}^{\leq -1})$. It follows from  $0\neq \Hom_\mathcal{S}(G,G)$ and $G\in\mathcal{S}^b\cap\mathcal{S}^{\leq 0}$ that $\pd_G(G)=0$. Moreover, by Lemma~\ref{prop-sil}, $G$ is a projective generator in $\mathcal{H}$. If $P\in\mathcal{H}$ is projective, then it is a direct summand of the coproducts of $G$, and thus $\pd_G(P) = 0$.

(2) Clearly, there is a triangle $N[-1]\to L\to M\to N$ in $\mathcal{S}$. For any $Y\in\mathcal{S}^{b}\cap\mathcal{S}^{\leq0}$ and $i\in\mathbb{Z}$, we have an exact sequence of abelian groups:
\[\Hom_{\mathcal{S}}(N,Y[i])\to\Hom_{\mathcal{S}}(M,Y[i])\to\Hom_{\mathcal{S}}(L,Y[i])\to\Hom_{\mathcal{S}}(N,Y[i+1])\to \Hom_{\mathcal{S}}(M,Y[i+1]).\]
Let $\pd_{G}(N)=n+1<+\infty$. Then $\pd_G(M)\leq n$. It follows that
$\Hom_{\mathcal{S}}(M,Y[i])=0=\Hom_{\mathcal{S}}(N,Y[i+1])$ for $i>n$, and that there exists a surjection $\Hom_{\mathcal{S}}(L,Y[n])\to\Hom_{\mathcal{S}}(N,Y[n+1])$. This implies that
$\Hom_{\mathcal{S}}(L,Y[i])=0$ for $i>n$. Moreover, since $\pd_{G}(N)=n+1$, there exists at least one object  $Z\in\mathcal{S}^{b}\cap\mathcal{S}^{\leq0}$ such that $\Hom_{\mathcal{S}}(N,Z[n+1])\neq 0$. Consequently, $\Hom_{\mathcal{S}}(L,Z[n])\neq 0$. Thus $\pd_{G}(L)=n$.

(3) Without loss of generality, it suffices to consider $i=0$. By Lemma~\ref{prop-sil}, $G$ is a projective generator in $\mathcal{H}$. So each projective object in $\mathcal{H}$ is a direct summand of the coproducts of $G$. Clearly, $H(-)$ respects coproducts (for example, see \cite[Lemma 2.4]{n18}). So, to show $(a)$, it suffices to show that  the morphism $\alpha_0:\Hom_{\mathcal{S}}(G, Y) \to \Hom_{\mathcal{H}}(G, H^{0}(Y))$, given by taking $H^0$, is an isomorphism.

By the $t$-structure $(\mathcal{S}^{\leq 0}, \mathcal{S}^{\geq 1})$,  there are canonical triangles in $\mathcal{S}$ (see Definition \ref{T-structure}):
\[Y^{\leq 0}\xrightarrow{f_1}Y \xrightarrow{g_1} Y^{\geq 1}\to Y^{\leq 0}[1]\quad\text{and}\quad Y^{\leq -1} \xrightarrow{f_2} Y^{\leq 0} \xrightarrow{g_2} H^{0}(Y)\to Y^{\leq -1}[1].\]
Let $h\colon  G \to H^0(Y)$. By $(\ast)$, we have $\Hom_{\mathcal{S}}(G, Y^{\leq -1}[1]) = 0$. Thus there exists a morphism $s\colon  G \to Y^{\leq 0}$ such that $h = g_2 s$. Then $H^0(f_1 s)= h$, proving the surjection of $\alpha_0$.

Now suppose that $f\colon  G \to Y$ is a morphism in $\mathcal{S}$ with $\alpha_0(f) =H^0(f)= 0$. Since
$G\in\mathcal{S}^{\leq 0}$, there exists a unique morphism $f^{\leq 0}: G\to Y^{\leq 0}$ in $\mathcal{S}$ such that $f=f_1f^{\leq 0}$. By $H^0=((-)^{\leq 0})^{\geq 0}$ and $G\in\mathcal{H}$, we have commutative diagrams:\[\begin{tikzcd}
G \arrow[d, "f^{\leq 0}"'] \arrow[r, no head, shift left] \arrow[r, no head] & G \arrow[d, "f"] &                                & G \arrow[d, "f^{\leq 0}"'] \arrow[r, no head, shift left] \arrow[r, no head] & G \arrow[d, "H^0(f)=0"] \\
Y^{\leq 0} \arrow[r, "f_{1}"']                                               & {Y,}             & Y^{\leq-1} \arrow[r, "f_{2}"'] & Y^{\leq 0} \arrow[r, "g_{2}"']                                               & H^{0}(Y).
\end{tikzcd}\]This forces $g_{2}f^{\leq 0}=0$, and therefore there exists a morphism $g: G\to Y^{\leq -1}$ such that $f^{\leq 0}=f_{2}g$. Since $\Hom_\mathcal{S}(G,  Y^{\leq-1}) =0$ by $(\ast)$, we have $g=0$. Thus $ f =f_1f_2g= 0 $, proving the injection of $ \alpha_{0} $.

Note that $\pd_G(M)\geq 0$ for any $0\neq M\in\mathcal{H}$ since $\Hom_\mathcal{S}(M,M)\neq 0$. Now, we proceed by induction on $n$ to show $(b)$.

In fact, for $n=0$,  we see from $(1)$ that $H^{i}(X)\in\mathcal{H}$ is projective. Since $G$ is tilting, it follows from Lemma~\ref{prop-sil} that $H^i(X)\in \overline{\langle G\rangle}_{1}$ for all $i\in\mathbb{Z}$. Moreover, by $(a)$, the identity morphism $\mathrm{Id}_{H^{i}(X)}\colon  H^{i}(X) \to H^{i}(X)$ admits a unique lift $f_{i}\colon  H^{i}(X)[-i] \to X$ such that $H^i(f_i)=\mathrm{Id}_{H^{i}(X)}$. Let $f\colon \bigoplus_{j\in\mathbb{Z}}H^{j}(X)[-j]\to X$ be the morphism in $\mathcal{S}$ induced by $\{f_{j}\}_{i\in\mathbb{Z}}$. Since the functor $H(-):\mathcal{S}\to\mathcal{H}$ preserves coproducts, each $H^i(f)$ is an isomorphism. By Lemma \ref{prop-sil} and \cite[Proposition 1.3.7]{BBD}, $f$ is an isomorphism. Thus $X\in\overline{\langle G\rangle}_{1}$.

Assume that $(b)$ holds for $n=k\geq 0$. We now consider $n=k+1$. Since $G$ is a projective generator in $\mathcal{H}$ by Lemma~\ref{prop-sil}, for each $i\in\mathbb{Z}$, there exists a set $I_i$ and an epimorphism $h_i:G^{(I_i)}\to H^i(X)$ in $\mathcal{H}$. By $(a)$,  $h_i$ admits a unique lift $g_i\colon  G^{(I_i)}[-i] \to X$ such that $H^i(g_i)=h_i$. Now, we consider the morphism $g \colon  Q:=\bigoplus_{j\in\mathbb{Z}}G^{(I_j)}[-j]\to X$ induced by $\{g_{j}\}_{j\in\mathbb{Z}}$. Then there is a triangle
$(\diamondsuit):\; Q \xrightarrow{g} X \to Y \to Q[1]$ in $\mathcal{S}$ that satisfies $H^i(g)=H^i(g_i)=h_i$ for each $i\in\mathbb{Z}$. Since all $h_i$ are epimorphisms, taking $H^i(-)$ to the triangle $(\diamondsuit)$ yields a series of exact sequences in $\mathcal{H}$:
\[0 \to H^{i-1}(Y) \to H^i(Q) \to H^i(X) \to 0.\]
Note that $H^i(Q)=G^{(I_i)}$ is projective in $\mathcal{H}$. By $(1)$, we have $\pd_G(H^i(Q))=0$.
Since $\pd_G(H^i(X))\leq k+1$ by assumption, it follows from $(2)$ that $\pd_G(H^{i-1}(Y)) \leq k$ for each $i\in\mathbb{Z}$. By induction, $Y \in \overline{\langle G \rangle}_{k+1}$. Clearly, $Q \in \overline{\langle G \rangle}_1$. Thus the triangle $(\diamondsuit)$ implies $X \in \overline{\langle G \rangle}_{k+2}$. This shows $(b)$.
\end{proof}

\begin{lem}\label{SG}
$(1)$ Suppose that $\Hom_\mathcal{S}(G[i], G) = 0$ for $i\gg 0$. Then, for any $j\in\mathbb{N}$,
$$\mathcal{S}^{\geq 0}\cap \overline{\langle G\rangle}_{j+1}^{(-\infty, +\infty)}\subseteq \overline{\langle G\rangle}_{j+1}^{[-j(o^{-}(G)+1), +\infty)}.$$

$(2)$ Suppose that $\mathcal{S} = \overline{\langle G\rangle}_{n+1}^{(-\infty, +\infty)}$ for some integer $n\geq 0$ and  $\Hom_{\mathcal{S}}(G, G[i]) = 0$ for $|i| \gg 0$. Then
$\Hom_\mathcal{S}(\mathcal{S}^{\geq 0}, \mathcal{S}^{\leq 0}[m+1])=0,$
where $m:=o^+(G)+n(1+o^-(G))<+\infty$. Moreover, $o^+(G)-o^-(G)\leq \gpd(\mathcal{S}, G)\leq m$.
\end{lem}

\begin{proof}
$(1)$ We borrow the idea used in the proof of \cite[Proposition 4.11]{bcrpz24} to prove Lemma \ref{SG}.

For each $X\in\mathcal{S}$, we define
$\Lambda_X:=\{i\in\mathbb{Z}\,|\,\Hom_{\mathcal{S}}(G[i],X)\not=0\}.$
Let $X_0\in\mathcal{S}^{\geq  0}$. Since $\mathcal{S}^{\geq  0}=G(-\infty,-1]^{\perp}$ by Lemma \ref{lem-p-w-gen}, $\Lambda_{X_0}$ consists of nonpositive integers. For each $i\in\Lambda_{X_0}$, let $I_i$ be a generating set of $\Hom_{\mathcal{S}}(G[i],X)$ as a left module over the ring $R:=\End_\mathcal{S}(G)^{{\rm op}}$. Then there exists an object
$$G_{i,X_0}=\bigoplus_{i\in \Lambda_{X_0}}G[i]^{(I_i)}\in {\overline{\langle G[i] \rangle}}_1^{\{0\}}$$ and a morphism $f_{i,X_0}: G_{i,X_0}\to X_0$ in $\mathcal{S}$ such that
$$\Hom_{\mathcal{S}}(G[i],f_{i,X_0}):{\rm Hom}_{\mathcal{S}}(G[i],G_{i,X_0})\lra
\Hom_{\mathcal{S}}(G[i],X_0)$$ is surjective homomorphism of $R$-modules.
Now, we define $X'_0:=\bigoplus_{i\in\Lambda_{X_0}}G_{i,X_0}$ and let $f:X'_0\rightarrow X_0$ be the morphism induced by the family $\{f_{i,X_0}\}_{i\in\Lambda_{X_0}}$. Then $f$ is extended to a triangle in $\mathcal{S}$:
$$
\begin{tikzcd}
(\dag)\quad X'_0 \arrow[r, "f"] & X_0 \arrow[r, "g_0"] & X_1 \arrow[r] & {X'_0[1]}
\end{tikzcd}
$$
The construction of $f$ implies that $X'_0\in\overline{\langle G\rangle}_1^{[0,+\infty)}$ and $\Hom_\mathcal{S}(G[j], f)$ is surjective for any $j\in\mathbb{Z}$. It follows that
$g_0$ belongs to  $$\mathscr{S}:=\{h\in\Hom_{\mathcal{S}}(X,Y)\mid X,Y\in\mathcal{S}\;\;\text{and}\;\;\Hom_{\mathcal{S}}(G[j],h)=0,\;\forall\; j\in\mathbb{Z}\}.$$

Set $d:=o^-(G)$. By Definition \ref{index}(1), $d=\sup\big\{j\in\mathbb{N}\mid \Hom_\mathcal{S}(G[j],G)\not=0\big\}$. Clearly, the assumption on $G$ implies that $0\leq d<+\infty$ and  $\Hom_\mathcal{S}(G[j],G)=0$ for $j>d$. Consequently, $G[-d]\in\mathcal{S}^{\geq  0}$.
Since $\mathcal{S}^{\geq  0}\subseteq\mathcal{S}$ is closed under extensions, negative shifts, direct summands and coproducts, we have $\overline{\langle G\rangle}^{[d,+\infty)} \in \mathcal{S}^{\geq  0}$. Thus $X'_0\in\overline{\langle G\rangle}_1^{[0,+\infty)}\subseteq\overline{\langle G\rangle}^{[0,+\infty)}\subseteq\mathcal{S}^{\geq  -d}$ and $X_1\in\mathcal{S}^{\geq  -d-1}$. The above procedure can be carried out analogously for the object $X_1$, yielding a distinguished triangle
$X_1' \to  X_1 \lraf{g_1}X_2 \to X_1'[1],$
where $X_1'\in\overline{\langle G\rangle}_1^{[-d-1,+\infty)}\subseteq\mathcal{S}^{\geq -2d-1}$, $g_1\in\mathscr{S}$ and $X_2\in\mathcal{S}^{\geq -2d-2}$. Therefore the octahedron axiom of triangulated category produces
a triangle
$G_2 \to X_0 \lraf{g_1g_0}X_2\to G_2[1]$,
where $G_2\in X_0'\ast X_1'\subseteq\overline{\langle G\rangle}_2^{[-d-1,+\infty)}$. More generally, for any $j\in\mathbb{N}^{+}$, we obtain a triangle
$$\begin{tikzcd}
G_{j+1} \arrow[r] & X_0 \arrow[r, "g_j\cdots g_1g_0"] & X_{j+1} \arrow[r] & {G_{j+1}[1]}
\end{tikzcd}$$
where $G_{j+1}\in\overline{\langle G\rangle}_{j+1}^{[-j(d+1),+\infty)}$ and $g_t\in\mathscr{S}$ for $0\leq t\leq j$.
Note that $\mathscr{S}$ has the following nice property $(\ast)$: If $X\in\overline{\langle G\rangle}_{n+1}^{(-\infty, +\infty)}$ for some $n\in\mathbb{N}$ and $f:X\to Y$ is the composition of $n+1$ morphisms in $\mathscr{S}$, then $f=0$.
This can be shown by induction on $n$. For example, the case $n=1$ follows from the fact that
$$
\mathscr{S}= \{h\in\Hom_{\mathcal{S}}(X,Y)\mid X,Y\in\mathcal{S}\;\text{and}\;\Hom_{\mathcal{S}}(Z,h)=0,\;\forall\; Z\in\overline{\langle G\rangle}_1^{(-\infty, +\infty)}\}.
$$
Now, let $X_0\in \mathcal{S}^{\geq 0}\cap \overline{\langle G\rangle}_{j+1}^{(-\infty, +\infty)}$.
Since $g_t\in\mathscr{S}$ for $0\leq t\leq j$, we have $g_j\cdots g_1g_0=0$ by $(\ast)$.
This implies that $X_0$ is a direct summand of $G_{j+1}$, and therefore $X_0\in\overline{\langle G\rangle}_{j+1}^{[-j(d+1),+\infty)}$.

$(2)$ By Lemma \ref{lem-p-w-gen}, $\mathcal{S}^{\leq 0}=\Coprod(G(-\infty,0])$. Since $\mathcal{S}= \overline{\langle G\rangle}_{n+1}^{(-\infty, +\infty)}$, it follows from $(1)$ and Definition \ref{notation}(5) that
$$\mathcal{S}^{\geq 0}\subseteq \overline{\langle G\rangle}_{n+1}^{[-n(d+1), +\infty)}=\text{smd}\big(\text{Coprod}_{n+1}(G[-n(d+1), +\infty))\big),$$where $d=o^{-}(G)$. To show $\Hom_\mathcal{S}(\mathcal{S}^{\geq 0}, \mathcal{S}^{\leq 0}[m+1])=0$, it suffices to show $$(\diamondsuit):\;\Hom_\mathcal{S}(\text{Coprod}_{n+1}(G[-n(d+1), +\infty), \Coprod(G(-\infty,0])[m+1])=0.$$
By $G\in\mathcal{S}^c$, we see that $(\diamondsuit)$ holds if and only if  $\Hom_\mathcal{S}(G[-n(d+1), +\infty), G(-\infty, -m-1])=0$. The latter is equivalent to saying that $\Hom_\mathcal{S}(G[-r], G[-s])=0$
for all $r\geq -n(d+1)$ and $s\leq -m-1$. Clearly, $\Hom_\mathcal{S}(G[-r], G[-s])\simeq \Hom_\mathcal{S}(G, G[r-s])$ and $r-s\geq -n(d+1)+(m+1)=o^+(G)+1$. By the definition of $o^+(G)$, we have $\Hom_\mathcal{S}(G, G[p])=0$ for $p\geq o^+(G)+1$. Thus $(\diamondsuit)$ holds.

By Lemma \ref{lem-global}(1),  $\gpd(\mathcal{S}, G)\leq m<+\infty$. It remains to show $o^+(G)-o^-(G)\leq \gpd(\mathcal{S}, G)$.

Let $c:=\gpd(\mathcal{S}, G)<+\infty$. Then $\Hom_{\mathcal{S}}(\mathcal{S}^{\geq 0}\cap\mathcal{S}^{b},\mathcal{S}^{\leq-(c+1)}\cap\mathcal{S}^{b})=0$. Since $G$ is bounded, we have $G\in \mathcal{S}^c\subseteq \mathcal{S}^b$. Then $\Hom_{\mathcal{S}}(\mathcal{S}^{\geq 0}\cap\mathcal{S}^{c}, (\mathcal{S}^{\leq 0}\cap\mathcal{S}^c)[c+1])=0$.
On the one hand, $G[-d]\in\mathcal{S}^{\geq  0}$ by the proof of $(1)$. On the other hand, $\mathcal{S}^{\leq 0}$ is closed under positive shifts in $\mathcal{S}$ and contains $G$. It follows that $\Hom_{\mathcal{S}}(G, G[j])=0$ for $j\geq c+d+1$. Further, by the definition of $o^+(G)$, we have $c+d\geq o^+(G)$. Thus $c\geq o^+(G)-d=o^+(G)-o^-(G)$.
\end{proof}

{\bf Proof of Theorem \ref{main-thm-scg-gd}.}
The implication of $(1)$ to $(2)$ follows from Lemma \ref{SG}(2).

Suppose $\Hom_\mathcal{S}(G,G[i])=0$ for all $i\neq 0$. Then $G\in\mathcal{S}$ is tilting.
In this case, $o^+(G)=0=o^-(G)$. Now, assume that $(2)$ holds. Then $0\leq \gpd(\mathcal{S},G)\leq n<+\infty$. In particular, $\pd_G(M)\leq n$ for any $M\in\mathcal{H}\subseteq\mathcal{S}^{\geq 0}\cap\mathcal{S}^b$. Thus the implication of $(2)$ to $(1)$ follows from Lemma \ref{proj obj}(3)(b).
\hfill$\square$

\smallskip

A special case of Theorem \ref{main-thm-scg-gd} is Kelly's theorem on unbounded derived categories.

\begin{cor}{\rm \cite{k65}}\label{cor-kellythm}
Let $R$ be a ring with identity and let $G$ be a compact, tilting object of $\D{R\Modcat}$.
Then \[\gd(\D{R\Modcat},G)=\min\{n\in\mathbb{N}\mid\D{R\Modcat}=\overline{\langle G\rangle}_{n+1}\}.\]In particular, ${\rm gldim}(R)=\min\{n\in\mathbb{N}\mid\D{R\Modcat}=\overline{\langle R\rangle}_{n+1}\}.$
\end{cor}

\begin{proof}
Since $G$ is a tilting object of $\D{R\Modcat}$, we have $o^+(G)=0=o^-(G)$. By Lemmas \ref{prop-sil} and \ref{proj obj}(1), $G\in\mathcal{H}$ with $\pd_G(G)=0$. This implies $0\leq {\rm gpd}(\D{R\Modcat},G)$.
Now, it follows from Theorem \ref{main-thm-scg-gd} that, for each $n\in\mathbb{N}$, $\D{R\Modcat} = \overline{\langle G\rangle}_{n+1}^{(-\infty, +\infty)}$ if and only if $0\leq {\rm gpd}(\D{R\Modcat},G)\leq n$. Thus ${\rm gpd}(\D{R\Modcat},G)=\min\{n\in\mathbb{N}\mid \D{R\Modcat}=\overline{\langle G\rangle}_{n+1}^{(-\infty, +\infty)}\}$.
\end{proof}

\subsection{Reduction of strong compact generation by recollements}
In this subsection, we discuss the behavior of strong compact generation under recollements of triangulated categories.

We begin by establishing the following result.

\begin{lem}\label{lem-scg}
Let $\mathcal{R}$, $\mathcal{S}$ and $\mathcal{T}$ be triangulated categories with coproducts and with compact generators $G_{\mathcal{R}}$, $G_{\mathcal{S}}$ and $G_{\mathcal{T}}$, respectively. Suppose that there is a recollement
\[
\xymatrixcolsep{4pc}\xymatrix{\mathcal{R} \ar[r]|{i_*=i_!} &\mathcal{S} \ar@<-2ex>[l]|{i^*} \ar@<2ex>[l]|{i^!} \ar[r]|{j^!=j^*}  &\mathcal{T}. \ar@<-2ex>[l]|{j_!} \ar@<2ex>[l]|{j_{*}}
}
\]

{\rm (1)} If $\mathcal{S}$ is strongly compactly generated, then so is $\mathcal{R}$. If, in addition, $j^*(G_{\mathcal{S}})\in \overline{\langle G_{\mathcal{T}}\rangle}^{[-u,u]}_u$ for some $u>0$, then $\mathcal{T}$ is strongly compactly generated.

{\rm (2)} Suppose that $i_*(G_{\mathcal{R}})\in \overline{\langle G_{\mathcal{S}}\rangle}^{[-v,v]}_v$ for some $v>0$. If $\mathcal{R}$ and $\mathcal{T}$ are strongly compactly generated, then so is $\mathcal{S}$.
\end{lem}
\begin{proof}
(1) Let $\mathcal{S}=\overline{\langle G_{\mathcal{S}}\rangle}_{s+1}^{(-\infty,+\infty)}$ for some $s\geq 0$. Since $i_*$ is fully faithful, $X\simeq i^*i_*(X)$ for any $X\in \mathcal{R}$.
As $i_*(X)\in\mathcal{S}$ and $i^*$ preserves coproducts, it follows that
$$
X\simeq i^*(i_*(X))\in i^*(\overline{\langle G_{\mathcal{S}}\rangle}_{s+1}^{(-\infty,+\infty)})
\subseteq \overline{\langle i^*(G_{\mathcal{S}})\rangle}_{s+1}^{(-\infty,+\infty)}.
$$ This forces $\mathcal{R}=\overline{\langle i^*(G_{\mathcal{S}})\rangle}_{s+1}^{(-\infty,+\infty)}$. By Lemma~\ref{lem-key}(1), $i^*(G_{\mathcal{S}})\in\mathcal{R}^c$. Thus $\mathcal{R}$ is strongly compactly generated.

Let $Z\in \mathcal{T}$. Since $j_!$ is fully faithful, $Z\simeq j^*j_!(Z)$.
As $j_!(Z)\in\mathcal{S}$ and $j^*$ preserves coproducts, we obtain $$
Z\simeq j^*j_!(Z)\in j^*(\overline{\langle G_{\mathcal{S}}\rangle}_{s+1}^{(-\infty,+\infty)})
\subseteq \overline{\langle j^*(G_{\mathcal{S}})\rangle}_{s+1}^{(-\infty,+\infty)}.
$$Hence $\mathcal{T}=\overline{\langle j^*(G_{\mathcal{S}})\rangle}_{s+1}^{(-\infty,+\infty)}$. Now, we assume $j^*(G_{\mathcal{S}})\in \overline{\langle G_{\mathcal{T}}\rangle}^{[-u,u]}_u$. Then $\overline{\langle j^*(G_{\mathcal{S}})\rangle}_{s+1}^{(-\infty,+\infty)}\subseteq\overline{\langle G_{\mathcal{T}}\rangle}_{u(s+1)}^{(-\infty,+\infty)}$ and thus $\mathcal{T}$ is strongly compactly generated.

(2) Let $\mathcal{R}=\overline{\langle G_{\mathcal{R}}\rangle}_{r+1}^{(-\infty,+\infty)}$ and $\mathcal{T}=\overline{\langle G_{\mathcal{T}}\rangle}_{t+1}^{(-\infty,+\infty)}$ for some $r, t\geq 0$.
For each $Y\in \mathcal{S}$, we have a canonical triangle $j_!j^*(Y)\ra Y \ra i_*i^*(Y)\ra j_!j^*(Y)[1]$ in $\mathcal{S}$. Since $j^*(Y)\in \mathcal{T}$ and $j_!$ preserves coproducts,
$$j_!j^*(Y)\in j_!(\overline{\langle G_{\mathcal{T}}\rangle}_{t+1}^{(-\infty,+\infty)})
\subseteq \overline{\langle j_!(G_{\mathcal{T}})\rangle_{t+1}}^{(-\infty,+\infty)}.$$
Similarly, $i_*i^*(Y)\in \overline{\langle i_{*}(G_{\mathcal{R}})\rangle}_{r+1}^{(-\infty,+\infty)}$.
It follows from the triangle that $Y\in \overline{\langle j_{!}(G_{\mathcal{T}})\oplus i_{*}(G_{\mathcal{R}})\rangle}_{r+t+2}^{(-\infty,+\infty)}$. This forces $\mathcal{S}=\overline{\langle j_{!}(G_{\mathcal{T}})\oplus i_{*}(G_{\mathcal{R}})\rangle}_{r+t+2}^{(-\infty,+\infty)}$. By Lemma~\ref{lem-key}, $j_!(G_{\mathcal{T}})\in\mathcal{S}^{c}=\langle G_{\mathcal{S}}\rangle$. Since $i_*(G_{\mathcal{R}})\in \overline{\langle G_{\mathcal{S}}\rangle}^{[-v,v]}_v$, we have $j_!(G_{\mathcal{T}})\oplus i_*(G_{\mathcal{R}})\in\overline{\langle G_{\mathcal{S}}\rangle}_{l}^{[-l,l]}$ for some $l> 0$. Then $\overline{\langle j_{!}(G_{\mathcal{T}})\oplus i_{*}(G_{\mathcal{R}})\rangle}_{r+t+2}^{(-\infty,+\infty)}\subseteq\overline{\langle G_{\mathcal{S}}\rangle}_{l(r+t+2)}$. Thus $\mathcal{S}$ is strongly compactly generated.
\end{proof}

Now, we state the main result of this subsection.

\begin{thm}\label{thm-rec-scg}
Let $\mathcal{R}$, $\mathcal{S}$ and $\mathcal{T}$ be approximable triangulated categories with bounded compact generators.
Suppose that there is a recollement
\begin{align*}%\label{thm-rec-st-thm}
\xymatrixcolsep{4pc}\xymatrix{\mathcal{R} \ar[r]|{i_*=i_!} &\mathcal{S} \ar@<-2ex>[l]|{i^*} \ar@<2ex>[l]|{i^!} \ar[r]|{j^!=j^*}  &\mathcal{T}. \ar@<-2ex>[l]|{j_!} \ar@<2ex>[l]|{j_{*}}
}
\end{align*}
Then $\mathcal{S}$ is strongly compactly generated if and only if so are $\mathcal{R}$ and $\mathcal{T}$.
\end{thm}
\begin{proof}
Suppose that $\mathcal{S}$ is strongly compactly generated by a bounded compact generator $G_\mathcal{S}$. Then $\mathcal{S}=\overline{\langle G_{\mathcal{S}}\rangle}_{s+1}^{(-\infty,+\infty)}$ for some $s\ge 0$. By Lemma \ref{lem-scg}(1), $\mathcal{R}$ is strongly compactly generated.
As $\mathcal{S}$ is weakly approximable and $G_{\mathcal{S}}$ is bounded, we see from Lemma \ref{PR}(1) that $\Hom_{\mathcal{S}}(G_\mathcal{S},G_\mathcal{S}[i])=0$ for $|i|\gg 0$. Since $\mathcal{S}$ is strongly compactly generated, ${\rm gpd}(\mathcal{S},G_\mathcal{S})<+\infty$ by Theorem \ref{main-thm-scg-gd}. It follows from Theorem \ref{global dim} that ${\rm gpd}(\mathcal{T},G_\mathcal{T})<+\infty$.
By Lemma \ref{lem-global}(2),  $\mathcal{T}^b=\mathcal{T}^{{\rm sb}}$. Further, by Lemma \ref{lem-gluing-$t$-struc}, $j^*(\mathcal{S}^b)\subseteq\mathcal{T}^b$. Since $G_{\mathcal{S}}\in\mathcal{S}^c\subseteq\mathcal{S}^b$, we obtain $j^{*}(G_{\mathcal{S}})\in \mathcal{T}^{{\rm sb}}$. Now, let $G_{\mathcal{T}}$ be a bounded compact generator of $\mathcal{T}$. Since $\mathcal{T}$ is approximable, we have $\mathcal{T}^{{\rm sb}}=\bigcup_{i\geq 0}\overline{\langle G_{\mathcal{T}}\rangle}_{i+1}^{[-i,i]}$ by Remark \ref{lem-app-sb}. Thus there exists an integer $u>0$ such that $j^{*}(G_{\mathcal{S}})\in \overline{\langle G_{\mathcal{T}}\rangle}_u^{[-u,u]}$. By Lemma \ref{lem-scg}(1), $\mathcal{T}$ is strongly compactly generated.

Conversely, suppose that $\mathcal{R}$ and $\mathcal{T}$ are strongly compactly generated by bounded compact generators $G_\mathcal{R}$ and $G_\mathcal{T}$, respectively. Then
${\rm gpd}(\mathcal{T},G_\mathcal{T})<+\infty$ by Theorem \ref{main-thm-scg-gd}. Clearly, by Lemma \ref{lem-global}(2), $\mathcal{T}^b=\mathcal{T}^{{\rm sb}}$.
Note that $j^*$ can be restricted to $\mathcal{S}^b\to \mathcal{T}^b$ by Lemma \ref{lem-gluing-$t$-struc}.  Since $\mathcal{S}$ has a bounded compact generator $G_{\mathcal{S}}$, we have $\mathcal{S}^{{\rm sb}}\subseteq \mathcal{S}^b$. Thus $j^*$ can be restricted to $\mathcal{S}^{{\rm sb}}\to \mathcal{T}^{{\rm sb}}$. In particular, $j^*(G_{\mathcal{S}})\in\mathcal{T}^{{\rm sb}}$. This implies $i_{*}(G_{\mathcal{R}})\in \mathcal{S}^{{\rm sb}}$, due to the equivalences of $(1)$ and $(1')$ in Theorem \ref{prop-rec-res1}. Similarly, by Remark \ref{lem-app-sb}, since $\mathcal{S}$ is approximable, there exists an integer $v>0$ such that $i_{*}(G_{\mathcal{R}})\in\overline{\langle G_{\mathcal{S}}\rangle}_v^{[-v,v]}$. Thus $\mathcal{S}$ is strongly compactly generated by Lemma \ref{lem-scg}(2).
\end{proof}

\section{Application to recollements induced by idempotents of rings}\label{6}
In this section, we apply our main results to recollements of  derived categories induced from idempotents of rings.

Let $k$ be a commutative ring, $A$ a flat $k$-algebra and $e^2=e \in A$ an idempotent.

The following result supplies a class of recollements of approximable triangulated categories.

\begin{prop}\label{construction}{\rm \cite[Proposition~2.10]{ky16}} There exists a differential graded (DG) $k$-algebra $\tilde{A}$, a homomorphism of dg $k$-algebras $f: A\to\tilde{A}$ and a recollement of derived categories:
\begin{align*}
\xymatrixcolsep{4pc}\xymatrix{\mathscr{D}(\tilde{A}) \ar[r]|{i_*=i_!} &\D{A\Modcat} \ar@<-2ex>[l]|{i^*} \ar@<2ex>[l]|{i^!} \ar[r]|{j^!=j^*}  &\D{eAe\Modcat} \ar@<-2ex>[l]|{j_!} \ar@<2ex>[l]|{j_{*}}
}
\end{align*} satisfying the following conditions:

{\rm (1)} $j_{!}=Ae\otimesL_{eAe}-$, $j^{*}=eA\otimesL_{A}-$, $i^{*}={_{\tilde{A}}}\tilde{A}\otimesL_{A}-$ and $i_{*}$ is the (derived) restriction functor induced from $f$;

{\rm (2)} $\tilde{A}^n=0$ for $n>0$, $H^{0}(\tilde{A})\simeq A/AeA$ and $H^{-n}(\tilde{A})\simeq\Tor^{eAe}_{n-1}(Ae,eA)$ for all $n\geq 2$.
\end{prop}

Applying our results to the recollement in Proposition \ref{construction},  we immediately obtain the following result. This provides a method for constructing non-positive dg algebras with finite (big) finitistic dimension.

\begin{cor}\label{Idempotent}
Suppose that the $eAe$-module $eA$ has a finite projective resolution by finitely generated (respectively,  arbitrary) projective $eAe$-modules.
Then:

{\rm (1)}
$\fd(A)\le \fpd(\mathscr{D}(\tilde{A}),\tilde{A})+w_{A}(i_{*}(\tilde{A}))+\fd(eAe)+1$
(respectively,  $\Fd(A)\le \Fpd(\mathscr{D}(\tilde{A}),\tilde{A})+w_{A}(i_{*}(\tilde{A}))+\Fd(eAe)+1$).

{\rm (2)} $\fpd(\mathscr{D}(\tilde{A}),\tilde{A})\le \fd(A)$
(respectively,  $\Fpd(\mathscr{D}(\tilde{A}),\tilde{A})\le \Fd(A)$).
\end{cor}

\begin{proof}
We prove the case of finitistic dimension; the case of big finitistic dimension can be proved  analogously. Observe that $o^{+}(\tilde{A})=o^{+}(A)=o^{-}(A)=0$. By \cite[Remark 5.3]{n18}, $a(\tilde{A})=a(A)=a(eAe)=0$. Moreover, by Proposition~\ref{construction}, $j_{!}(eAe) \simeq Ae$ and $i^{*}(A) \simeq \tilde{A}$. Thus $w_{A}(j_{!}(eAe))=0=w_{\tilde{A}}(i^{*}(A))$.

By Proposition \ref{construction}, $j^{*}(A)\simeq eA$. The assumption on the $eAe$-module $eA$ implies that $j^{*}(A)\in\Dc{eAe\Modcat}$. By Lemma \ref{lem-key}(1), $i_{*}(\tilde{A})\in\Dc{A\Modcat}$. Now, $(1)$ and $(2)$ follow from Theorems~\ref{thm-fd-small}(1).
\end{proof}

\begin{cor}\label{Idempotent-app}
Let $A$ be a finite-dimensional $k$-algebra over a field $k$.

$(1)$ Suppose that $\Tor_n^{eAe}(Ae,eA)=0$ for $n\gg 0$. If ${\rm gldim}(A)<+\infty$, then $\gpd(\mathscr{D}(\tilde{A}),\tilde{A})<+\infty$ and ${\rm gldim}(eAe)<+\infty$.

$(2)$ Suppose that ${\rm gldim}(eAe)<+\infty$. Then $\fd(A) < +\infty$ if and only if $\fpd(\mathscr{D}(\tilde{A}),\tilde{A})< +\infty$.

$(3)$ Suppose that $\Tor_n^{eAe}(Ae,eA)=0$ for $n\gg 0$ and that every simple left $(A/AeA)$-module has finite projective dimension as an $A$-module. Then $\gpd(\mathscr{D}(\tilde{A}),\tilde{A})<+\infty$. If, in addition, $\fd(eAe)<+\infty$, then $\fd(A)<+\infty$.

%\noindent 
Analogous conclusions of $(2)$ and $(3)$ hold true when finitistic dimension is replaced by big finitistic dimension or global dimension.
\end{cor}
\begin{proof}
(1) Since $\Tor_n^{eAe}(Ae,eA)=0$ for $n\gg 0$, the object $\tilde{A}\in\mathscr{D}(\tilde{A})$ is bounded by Proposition \ref{construction}. Suppose ${\rm gldim}(A)<+\infty$. By Theorem \ref{global dim}, $\gpd(\mathscr{D}(\tilde{A}),\tilde{A})<+\infty$ and ${\rm gldim}(eAe)<+\infty$.

(2) Since ${\rm gldim}(eAe)<+\infty$ and $A$ is finite-dimensional, the $eAe$-module $eA$ has a finite projective resolution by finitely generated projective $eAe$-modules. Note that $w_{\tilde{A}}(i^{*}(A))$ and $w_{A}(i_{*}(\tilde{A}))$ are finite (see the proof of Corollary \ref{Idempotent}). Moreover, $\fd(eAe)={\rm gldim}(eAe)$. Thus (2) follows from Corollary \ref{Idempotent}.

(3) Let $\{S_1,\dots,S_m\}$ be a complete set of pairwise non-isomorphic simple modules over $A/AeA$, and set $S:=\bigoplus_{i=1}^{m} S_{i}$. By \cite[Lemma 5.2]{g26}, $\mathscr{D}(\tilde{A})=\langle\overline{S}\rangle_{N}$ for some $N>0$. By assumption, each simple left $(A/AeA)$-module has finite projective dimension as an $A$-module, which implies $i_*(S)\in\Dc{A\Modcat}$. Since $i_*$ is fully-faithful and $i^*$ preserves compact objects, we obtain $S\simeq i^*i_*(S)\in\mathscr{D}(\tilde{A})^{c}$. Hence $\mathscr{D}(\tilde{A})$ is strongly compactly generated. Clearly, $\tilde{A}\simeq i^{*}(A)$ by Proposition \ref{construction}. Recall from the proof of $(1)$ that $\tilde{A}$ is a bounded compact generator of $\mathscr{D}(\tilde{A})$. Therefore, Theorem~\ref{main-thm-scg-gd} yields $\gpd(\mathscr{D}(\tilde{A}),\tilde{A})<+\infty$.

We now claim $i_*i^*(A)\in\Dc{A\Modcat}$. By assumption, every finitely generated $(A/AeA)$-module has finite projective dimension as an $A$-module. By Proposition \ref{construction}, we obtain $H^n(i_{*}i^{*}(A))=H^n(i^{*}(A))\simeq H^n(\tilde{A})$ for $n\in\mathbb{Z}$. It follows from $\Tor_n^{eAe}(Ae,eA)=0$ for $n\gg 0$ that $\bigoplus_{n\in\mathbb{Z}}H^{n}(i_{*}i^{*}(A))\in A\text{-}{\rm mod}$. In particular, $i_{*}i^{*}(A)$ is a bounded complex. Since
\[i_{*}i^{*}(A)\in\ker(j^{*})=\{X\in\D{A\Modcat}\mid H^n(X)\in (A/AeA)\text{-}{\rm Mod}\text{ for all }n\in\mathbb{Z}\},\]
all cohomologies $H^n(i_{*}i^{*}(A))$ of $i_{*}i^{*}(A)$ have finite projective resolutions by finitely generated projective $A$-modules. This implies $H^n(i_{*}i^{*}(A))\in\Dc{A\Modcat}$. Thus $i_*(\tilde{A})\simeq i_{*}i^{*}(A)\in\Dc{A\Modcat}$. Clearly, $w_{A}(j_!(eAe))=w_{A}(Ae)=0$. Thus, the desire result follows from Theorem \ref{thm-fd-small}(1).
\end{proof}

\begin{cor}\label{Extension}
Let $A$ be a finite-dimensional algebra over a field $k$ with $D:=\Hom_k(-,k)$ and $M\in A\modcat$. If $\Ext_{A}^{i}\big(M, D\Hom_A(M,A)\big)=0$ for $i\gg0$ and ${\rm gldim}\big(\End_{A}(A\oplus M)\big)<+\infty$, then ${\rm gldim}(A)<+\infty$.
\end{cor}
\begin{proof}
Let $X:=A\oplus M$,  $B:=\End_{A}(X)^{{\rm op}}$ and $e^2=e\in B$ the idempotent element corresponding to the direct summand $A$ of the $A$-module $X$. Then $eBe\simeq A$ as algebras. Note that $\Tor_i^{A}(Be,eB)\simeq D\Ext_A^i(eB, D(Be))$ for any $i\in\mathbb{N}$, $Be{_A}\simeq A\oplus\Hom_A(M,A)$ and ${_A}eB\simeq A\oplus M$. It follows that
$$\Tor_i^{A}(Be,eB)\simeq D\Ext_A^i\big(A\oplus M, DA\oplus D\Hom_A(M,A)\big)\simeq D\Ext_A^i\big(M, D\Hom_A(M,A)\big)$$ for $i>0$. If $\Ext_{A}^{i}\big(M, D\Hom_A(M,A)\big)=0$ for $i\gg 0$, then $\Tor_{i}^{A}(Be,eB)=0$ for $i\gg0$. Now, Corollary \ref{Extension} follows from Corollary \ref{Idempotent-app}(1).
\end{proof}

\medskip
{\bf Acknowledgements.} The research work was partially supported by the National Natural Science Foundation of China (Grant 12401038).

\end{document}